\documentclass[12pt]{article}
\usepackage{amsfonts,epsfig,latexsym,amssymb}
\usepackage{amsmath,amscd,amsthm, enumerate}
\usepackage{graphicx,psfrag,color,longtable,array}
\usepackage[T1]{fontenc}
\usepackage{lmodern,microtype,adjustbox}
\usepackage[hidelinks,bookmarks=false,hypertexnames=false]{hyperref}
\graphicspath{{figures/}}
\date{}
\begin{document}
\title{On the Hausdorff dimension of oscillatory motions in three-body problems}
\author{Anton Gorodetski, Vadim Kaloshin}

\maketitle

\begin{abstract}
We show that for the Sitnikov example and for the restricted planar circular 3–body
problem the set of oscillatory motions often has maximal Hausdorff dimension. Also, we
construct Newhouse domains for both problems.
\end{abstract}
\par\smallskip
%

\newtheorem{Theorem}{Theorem}
\newtheorem{Lemma}{Lemma}[section]
\newtheorem{Proposition}[Lemma]{Proposition}
\newtheorem{Remark}[Lemma]{Remark}
\newtheorem{Corollary}[Lemma]{Corollary}
\newtheorem{Conjecture}{Conjecture}
\newtheorem{Definition}[Lemma]{Definition}

{\it  The initial draft of this paper has
circulated since 2012 and has been cited
multiple times since then
\cite{CG2023,DKRS2019,FGLP2016,GMP2025,GMPS2022,GMS2016,GPSV2021,GSMS2017,LGS2025,SZ2020}.
Some of the results of this paper were formulated and announced
in~\cite{GK2009}.}

\section{Introduction and the statement of main results}

Poincar\'e's work on the three-body problem~\cite{p} initiated the
study of non-integrability and the geometry of asymptotic solutions.
His discovery of the homoclinic picture was fundamental to this development. This picture was at the origin of ``chaos theory''
over a century ago.  Later the results by Birkhoff and Smale gave
a deep insight to the dynamics associated with the homoclinic picture
and constructed a wide range of chaotic orbits. This, however, does not
come close to providing a description of a typical orbit. In the context of
abstract smooth dynamical systems Newhouse found domains (open sets)
in the space of dynamical systems having many surprising phenomena
that are generic! It includes a counterexample to Thom's
conjecture on generic finiteness of the number of coexisting attracting
periodic points (see \cite{n2}). One can also show that a generic
map there has an arbitrarily fast growth of the number of periodic points
along a subsequence (see \cite{k2}). Nowadays the former phenomenon is called
{\it Newhouse phenomena} and domains with such properties {\it Newhouse
domains}. These two examples give a glimpse how complex and
persistently changing the homoclinic picture can be (see also \cite{GK}
for further strange generic properties in those domains).

In the present paper we show that Newhouse phenomena actually exist
for certain  three-body problems (namely, for the Sitnikov problem and
for the restricted planar circular three-body problem). This leads to
a package of highly surprising dynamical properties there.  From
historical perspective applying a variety of the deep techniques
developed in dynamics during the last century to the context of
the celestial mechanics brings us back to the original motivation:
three-body problem.

The classical $3$-body problem consists in studying the dynamics
of $3$ point masses in the plane or in the $3$-dimensional
space mutually attracted under Newton gravitation: Let
$q_1,q_2,q_3$ be point masses in $\mathbb{R} ^d$ for $d=2$ or $3$
with masses $m_1, m_2, m_3$ respectively.
\[
\adjustbox{max width=\linewidth}{$\displaystyle
\ddot q_i = \sum_{j\ne i} m_j \dfrac{ q_j - q_i }{ | q_j-q_i |^3 },
\qquad  q_i \in \mathbb{R} ^d,\quad i=1,2,3.
$}
\]
Denote $r_k$ the vector from $q_i$ to $q_j$
with $i \ne k, j \ne k, i<j$. One possible direction is to study
qualitative behavior of bodies as time tends to infinity either in
the future or in the past. Chazy \cite{Chazy} gave a classification
of all possible types of asymptotic motions:

\begin{Theorem}  (Chazy, 1922 (see also \cite{AKN})) Every solution of
the three-body problem that is defined for all sufficiently large positive times belongs to one of the following seven classes:

\begin{itemize}
\item $\mathcal H^+$ (hyperbolic): $|r_k|\to \infty,\
|\dot r_k| \to c_k>0$ as $t \to +\infty$;

\item ${\mathcal HP}_k^+$ (hyperbolic-parabolic): $|r_k|\to \infty,\
|\dot r_k| \to 0, |\dot r_i|\to c_i>0$  ($i\ne k$);

\item ${\mathcal HE}_k^+$ (hyperbolic-elliptic):
$|r_i| \to \infty, |\dot r_i|\to c_i>0$  ($i\ne k$),
$\sup_{t \ge t_0} |r_k|<\infty$;

\item ${\mathcal PE}_k^+$ (parabolic-elliptic):
$|r_i| \to \infty, |\dot r_i|\to 0$  ($i\ne k$),
$\sup_{t \ge t_0} |r_k|<\infty$;

\item ${\mathcal P}^+$ (parabolic) $|r_k|\to \infty,\
|\dot r_k| \to 0$,

\item $B^+$ (bounded): $\sup_{t \ge t_0} |r_k|<\infty$;

\item $OS^+$ (oscillatory): $\limsup_{t \to \infty} \max_k |r_k|=\infty$,
$\liminf_{t \to \infty} \max_k |r_k|<\infty$;
\end{itemize}
\end{Theorem}

Examples of the first six types were known to Chazy. The existence
of oscillatory motions was proved by Sitnikov \cite{sit} in 1960.
Properties of the set of these motions are the central subject of the present paper.

Recall that energy of the three-body problem
\[
\adjustbox{max width=\linewidth}{$\displaystyle
H=\sum_k \dfrac{m_k |\dot q_k|^2}{2}-\sum_{i<j}
\dfrac{m_i\,m_j}{|q_i-q_j|}
$}
\]
is preserved along the solutions.  Notice that for positive energy
only $\mathcal H^\pm$, ${\mathcal HP}^\pm$, and ${\mathcal HE}^\pm$ are possible and
for negative energy only ${\mathcal HE}^\pm,\ {\mathcal PE}^\pm,\ B^\pm,\ OS^\pm$
are possible. It turns out that all logically possible intersections
of the past and the future final motion do exist as solutions of the
three-body problem. In the beginning of the last century there was
a heated discussion whether there is a symmetry of final type in
the past and in the future. In one of Chazy's papers (1929, see also \cite{AKN})
a false assertion was stated that in the three-body problem the two final
types in the past and the future coincide. It was disproved rigorously
by Sitnikov in 1953. Moreover, for all of them, but one, it is known whether
they form a set of initial conditions of positive or zero Lebesgue measure.
Here are two tables, omitting parabolic limiting cases, from a paper by Alexeev \cite{Ae,Af}.
Letting bodies to be far from each other we see that there is a qualitative
difference if energy $H$ is positive or negative. Indeed, potential energy
(the second term in energy $H$) goes to zero. So it is natural
to distinguish two cases:

\smallskip

\begin{table}[!htbp]
\begin{center}
\begin{adjustbox}{max width=\textwidth}
\begin{tabular}{|c|c|c|}
\hline
\multicolumn{3}{|c|}{Positive energy $H>0$} \\ \hline
 & $H^+$ & $HE_i^+$ \\ \hline
 & & \\
 & Lagrange, 1772  & PARTIAL CAPTURE \\
 & (isolated examples); & Measure $>0$ \\
$H^-$ & Chazy, 1922  & Shmidt (numerical examples), 1947; \\
 & Measure $>0$ & Sitnikov (qualitative methods), 1953 \\ \hline
 & & \\
 & COMPLETE DISPERSAL & $i=j$ \quad Measure $>0$ \\
 & Measure $>0$ & Birkhoff, 1927 \\ \cline{3-3}
  $HE_j^-$  & & $i \ne j$ \quad EXCHANGE, Measure $>0$ \\
  & & Bekker (numerical examples), 1920; \\
  & & Alexeev (qualitative methods), 1956 \\ \hline
\end{tabular}
\end{adjustbox}

\vspace{12pt}

\begin{adjustbox}{max width=\textwidth}
\begin{tabular}{|c|c|c|c|}
\hline
\multicolumn{4}{|c|}{Negative energy $H<0$} \\ \hline
    & $HE_i^+$ & $B^+$ & $OS^+$ \\ \hline
 & & & \\
 & $i=j$ \quad  Measure $>0$ & COMPLETE &  \\
 & Birkhoff, 1927 & CAPTURE & \\ \cline{2-2}
 & Exchange & Measure $=0$ & Measure $=0$ \\
 & $i\ne j$ \quad Measure $>0$ & Chazy,1929 \& Merman,1954; &
 Chazy,1929 \& Merman,1954; \\
 $HE_j^-$ & Bekker, 1920 & Littlewood, 1952; & Alexeev, 1968 \\
 & (numerical examples); & Alexeev, 1968; & $\ne \emptyset$ \\
 &  Alexeev, 1956; & $\ne \emptyset$ & \\
 & (qualitative methods) & & \\  \hline
 & & & \\
 & PARTIAL & Euler, 1772; & Littlewood, 1952 \\
 & DISPERSAL & Lagrange, 1772 & Measure $= 0$ \\
  $B^-$ & $\ne \emptyset$ & Poincar\'e, 1892 & Alexeev, 1968 \\
 & Measure $= 0$ & (isolated examples); & $\ne \emptyset$ \\
 &  & Arnold, 1963 & \\ \hline & & & \\
 & $\ne \emptyset$ & $\ne \emptyset$ & Sitnikov, 1960, \\
     $OS^-$  & Measure $= 0$ & Measure $= 0$ & $\ne \emptyset$ \\
      &  & & Measure $= ?$ \\ \hline
\end{tabular}
\end{adjustbox}
\end{center}
\end{table}

\smallskip

A central question highlighted by these tables is

\begin{center}
{\it Is Lebesgue measure of the set of oscillatory motions positive?}
\end{center}

Arnold in the conference in the honor of 70-th anniversary of Alexeev
called this {\it the central problem of the celestial mechanics.}
The conjecture, which probably\footnote{We are not sure because
in the English version \cite{Ae} Alexeev refers to Kolmogorov,
however in the French one \cite{Af} Kolmogorov's name is surprisingly
omitted.} goes back to Kolmogorov and stated in the paper of Alexeev \cite{Ae,Af},
is that Lebesgue measure of oscillatory motions is zero.

The $3$-body problem is called {\it restricted} if one of
the bodies has mass zero and the other two are strictly positive.
Historically, the first example of oscillatory motions is due to Sitnikov
\cite{sit} for the restricted spatial three-body problem.  In the pioneering
work \cite{Af,Ae}  Alexeev not only extended the Sitnikov example
to the spatial (unrestricted) $3$-body problem, but also found important
use of hyperbolic dynamics for the $3$-body problem. Later
Moser \cite{Mo} gave a conceptually transparent proof of existence
of oscillatory motions for the Sitnikov problem interpreting homoclinic
intersections using symbolic dynamics. This paved a road to a variety of
applications of hyperbolic dynamics to the three-body problem.

{The original development of this approach appears in Alekseev's
papers on quasirandom dynamical systems~\cite{A1968II,A1969III}.}

Dankowicz and Holmes~\cite{DH} gave a Melnikov-method proof of
transverse homoclinic intersections in the Sitnikov problem.

Existence of oscillatory motions for the planar three-body problem had
to wait for another decade. Investigation of planar oscillatory motions
was initiated by  Llibre--Sim\'o \cite{LS} who proved their existence for
the restricted planar circular $3$-body problem (see also \cite{MP}
on the subject of evaluation of the Melnikov integral). In \cite{X} Xia
simplified their proof. Robinson~\cite{R2} studied homoclinic orbits
and oscillation in the planar three-body problem. Oscillatory motions in the restricted planar elliptic problem and
in the planar three-body problem were studied in \cite{Xi2,Xi3}.
For an exposition
of the elliptic restricted problem, see \cite{Bak}.
Splitting of invariant manifolds formed by future (resp. past)
parabolic motions is studied in detail in \cite{MP}.

{Moeckel~\cite{Moe2007} constructed symbolic dynamics in the
unrestricted planar three-body problem combining close approaches to
triple collision with excursions near infinity, including oscillatory
motions. This gives a related construction of prescribed itineraries;
the question considered here is the Hausdorff dimension of the
oscillatory set in the restricted models.}

In this paper we investigate how large is the set {\it oscillatory
motions} and show that for the Sitnikov example and the restricted
planar circular $3$-body problem this set often has maximal
Hausdorff dimension. Full Hausdorff dimension is compatible with zero Lebesgue measure, so this conclusion does not contradict the conjecture.  Now we discuss
the Sitnikov example followed by the restricted planar circular
$3$-body problem and state our results.

\paragraph{Relation to recent work.}
The results of this paper were announced in part in~\cite{GK2009}
and developed in an earlier version of this manuscript. Subsequent
work has established overlapping results by different methods.
Guardia, Mart\'in and Seara~\cite{GMS2016} proved the existence of
oscillatory motions in the RPC3BP for every positive mass ratio and
sufficiently large Jacobi constant. In their introduction they also
observed, attributing the observation to the present authors, that
combining their estimates with our techniques yields full Hausdorff
dimension for each fixed mass and a Baire generic subset of an open
set of Jacobi constants.

{Capi\'nski, Guardia, Mart\'in, Seara and Zgliczy\'nski~\cite{CGMSZ2022}
proved the existence of oscillatory motions and all combinations of
past and future final motions for realistic Sun--Jupiter parameters,
using correctly aligned windows and rigorous computer assistance.
Their result concerns existence at specified physical parameters,
whereas our conclusions concern full Hausdorff dimension on residual
subsets of suitable parameter sets. Capi\'nski and Pasiut~\cite{CP2026}
proved the existence of oscillatory approaches to collision in the
Earth--Moon problem. In that result the trajectories repeatedly
approach a primary arbitrarily closely and return to a prescribed
distance from it; this differs from the unbounded excursions to
infinity considered here.}

Guardia, Mart\'in, Paradela and Seara~\cite{GMPS2022} proved the
existence of oscillatory motions in the unrestricted planar three-body
problem for positive masses that are not all equal, constructing
hyperbolic dynamics conjugate to a shift on infinitely many symbols.

Garrido, Mart\'in and Paradela~\cite{GMP2025} proved the existence of
mass-parameter Newhouse domains associated to parabolic infinity
and, for residual parameters, hyperbolic sets of dimension approaching
two in its homoclinic class. They also obtained results for the
Sitnikov problem and proved additional conclusions concerning
unbounded families of elliptic islands. Their Section~1.3 and
Remark~5 explain how their construction, combined with the dimension
argument of our earlier manuscript, gives an alternative proof of
full Hausdorff dimension for oscillatory motions. 

Baldom\'a, Giralt
and Guardia~\cite{BGG2023} constructed energy-parameter Newhouse
domains for each sufficiently small fixed mass near Lyapunov
periodic orbits associated to $L_3$. This concerns coorbital motions
near a finite energy level, rather than the regime of arbitrarily
large Jacobi constants considered here.

The present paper gives detailed proofs of our previously announced
results, with separate constructions for varying mass at fixed Jacobi
constant and for varying Jacobi constant at each fixed positive mass.
In the latter case we obtain an unbounded open set of Jacobi constants;
the required uniform splitting and coordinate estimates are established
in Sections~\ref{ss:complex-splitting-control} and
\ref{ss:rpc-uniform-energy}. The dimension argument in
Section~\ref{ss:sparse-excursions} prescribes excursions both to
infinity and back to a fixed finite periodic orbit in both time
directions, without decreasing the dimension bound supplied by a
hyperbolic basic set. Our Newhouse-domain constructions do not invoke
the corresponding theorems of~\cite{GMP2025}. They use the classical
first-order Sitnikov splitting calculation of Garc\'ia and
P\'erez-Chavela~\cite{GP}; see also \cite[Lemma B.1]{GMP2025} for
its formulation in coordinates measuring the energy displacement. 

\subsection{The Sitnikov example}

Consider two point masses $q_1$ and $q_2$ of equal mass
$m_1=m_2=1/2$. Suppose they move on the plane so that the center
of mass is at the origin. Assume that their orbits are elliptic of
eccentricity $e>0$ and period $2\pi$. We shall treat $e$ as
parameter. Consider a third massless point $q_3$ moving along the
$z$-axis. Due to symmetry if an initial condition and velocity
belong to the $z$-axis, then the whole orbit of $q_3$ also belongs
to the $z$-axis. Denote by $(t,z(t),\dot z(t))$ an orbit of $q_3$,
where the time $t $ (mod $2\pi$) determines position of primaries.
Denote $r(t)=r_{e}(t)$ distance of primaries to the origin. Then
the equation of motion of the massless body is of the form
 \begin{eqnarray}  \label{Sitnikov-EM}
 \ddot z=-\dfrac{z}{(z^2+r^2(t))^{3/2}}
 \end{eqnarray}
 and the corresponding Hamiltonian is of the form
 \[
\adjustbox{max width=\linewidth}{$\displaystyle
H(t,z,\dot z)=\dfrac{\dot z^2}{2}-\dfrac{1}{\sqrt{z^2+r^2(t)}}.
$}
\]

\begin{figure}[htbp]
    \begin{center}
        \includegraphics[width= 3.5in,angle=0]{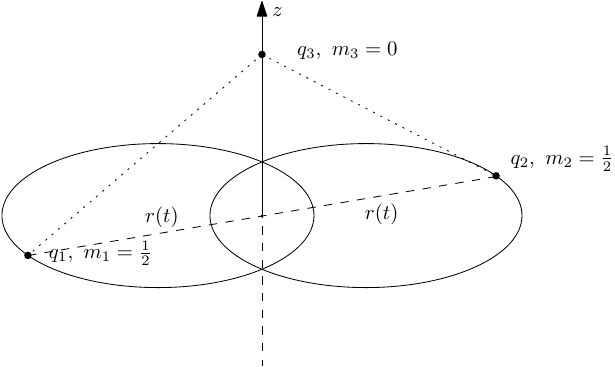}
    \caption{Sitnikov problem.}
    \label{f.sitnikov}
\end{center}
\end{figure}

\begin{Theorem}  \label{Sitnikov-main} There is an open set $\mathcal N \subset$
$(0,1)$ accumulating at zero, of values of eccentricity $e$ such
that for a Baire generic $e\in \mathcal N$ the set of two-sided oscillatory motions in the extended phase space
$(\mathbb R/2\pi\mathbb Z)\times\mathbb R^2$ has Hausdorff dimension $3$. \end{Theorem}

\subsection{The restricted planar circular three-body problem}
Consider the restricted planar circular $3$-body problem. Namely,
consider two massive bodies, the {\it primaries}, performing uniform
circular motion about their center of mass. Normalizing the masses
of the primaries so that their masses sum to one, we obtain
primaries of mass $\mu$ and $1 - \mu$ respectively, where $0 < \mu <
1$ is the \textit{mass ratio}. In addition, we choose coordinates so
that the center of mass of the system is located at the origin, and
we normalize the period of the circular motion to $2\pi$. By
entering into a frame which rotates with the primaries, we can
choose rectangular coordinates $(x, y)$ so that the primaries are
fixed at $(1 - \mu, 0)$ and $(-\mu, 0)$, respectively. Finally, we
introduce a third massless body $P$ into the system, so that it does
not effect the primaries. RPC3BP investigates how $P$ moves.

The distance of $P$ to the primaries is given by   $d_1(x, y) = [(x
- (1 - \mu))^2 + y^2]^{1/2}$ and $d_2(x, y) = [(x + \mu)^2 +
y^2]^{1/2}$.  The standard formula for the \textit{Jacobi constant}
$C$, a conserved integral for RPC3BP, is given by
 \begin{eqnarray}  C_\mu (x,y,\dot x, \dot y)=
 x^2 + y^2 + \frac{2\mu}{d_1} + \frac{2(1 - \mu)}{d_2} -
 (\dot{x}^2 + \dot{y}^2). \label{Jacobi}
 \end{eqnarray}

\begin{figure}[htbp]
    \begin{center}
        \includegraphics[width= 3.5in,angle=0]{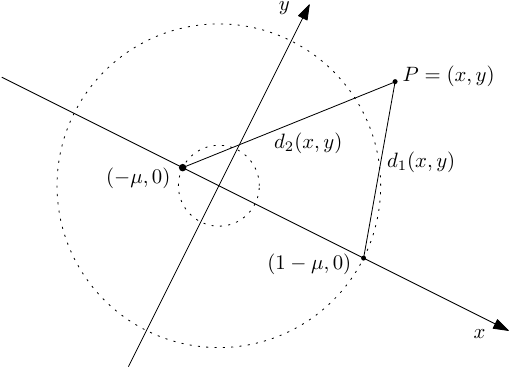}
    \caption{The restricted planar circular $3$-body problem.}
    \label{f.rpc3bp}
\end{center}
\end{figure}

Here is the main result for the RPC3BP. Denote by RPC3BP$(\mu, C)$
the RPC3BP with mass ratio $\mu$ restricted to the energy surface
$\Pi_C=\{C_\mu(x,y,\dot x,\dot y)=C\}$. We shall treat both $\mu$
and $C$ as parameters.

\begin{Theorem}[RPC3BP parameter conclusions]\label{RPC3BP-main}
For every sufficiently large positive Jacobi constant $C$, there is
a nonempty open set $\mathcal N_C\subset(0,1/2)$, with
$\mathcal N_C$ accumulating at zero, and a residual subset
$\mathcal R_C\subset\mathcal N_C$ such that, for every
$\mu\in\mathcal R_C$, the set of two-sided oscillatory motions
in the collision-free part of the energy surface $\Pi_C$ has Hausdorff dimension three.

For every fixed $\mu\in(0,1)$ there is also an unbounded open
set $\mathcal N_\mu$ of positive Jacobi constants and a residual
subset $\mathcal R_\mu\subset\mathcal N_\mu$ such that for
$C\in\mathcal R_\mu$ the same dimension-three conclusion holds.
\end{Theorem}

The fixed-energy assertion is proved in
Section~\ref{ss:fixed-energy-completion}; the fixed-mass assertion
is proved in Section~\ref{ss:rpc-uniform-energy}.

 \begin{Remark}
 Our technique could also be applied to the $3$-body problem on
 the line \cite{ls1,sx}, but we do not elaborate on it here.
 \end{Remark}

In what follows the following motions are also of importance. A motion
of the massless body is called {\it future (resp. past) parabolic}  if the body
escapes to infinity with vanishing speed as time tends to $+\infty$
(resp. $-\infty$).

\subsection{Reduction to area-preserving maps}

A natural way to reduce the Sitnikov example to a $2$-dimensional
Poincar\'e map is as follows. Define
 \begin{eqnarray}  \label{sitnikov-map}
 f_{e}:(z,\dot z)\mapsto
 (z',\dot z')\qquad  (z,\dot z)\in \mathbb{R} ^2
 \end{eqnarray}
where a trajectory of (\ref{Sitnikov-EM}) with initial condition
$(0,z,\dot z)$ at time $0$ is located at $(2\pi,z',\dot z')$ at time $2\pi$.
Since equations of motion are Hamiltonian this map is
{\it area-preserving}.

There are many ways to define a Poincar\'e map for the RPC3BP$(\mu,C)$
for sufficiently large positive $C$, on the collision-free domain near
parabolic infinity where the rotating angle is monotone. Consider the polar coordinates
$(r,\varphi)$ on the $(x,y)$-plane and let $(P_r,P_\varphi)$ be
their symplectic conjugate. Write the Hamiltonian of the RPC3BP in
these coordinates:
\begin{eqnarray}  \label{Hamiltonian}
H(r,P_r,\varphi,P_\varphi)=\dfrac{P_r^2}{2}+\dfrac{P_\varphi^2}{2r^2}-
\dfrac 1r-P_\varphi+\left(\dfrac 1r- \frac{\mu}{d_1} - \frac{1 -
\mu}{d_2}\right)=:H_0 + \Delta H,
\end{eqnarray}
where $d_1$ and $d_2$ are the distances to the primaries as above (\ref{Jacobi}),
\begin{eqnarray*}
d_1 & =& \left( r^2-2(1-\mu)r \cos \varphi +(1-\mu)^2\right)^{1/2}\\
d_2 & =& (r^2+2\mu r \cos \varphi +\mu^2)^{1/2},
\end{eqnarray*}
$P_r$ (resp. $P_\varphi$) is the variable conjugate
to $r$ (resp. $\varphi$). In other words, $P_r=\dot r$ and
$P_\varphi$ is the angular momentum. One can rewrite the Jacobi
constant in the polar coordinates.

Since the Jacobi constant is the first integral of this problem,
there is a $3$-dimensional `energy' surface
$\Pi_C=\{C=C_\mu(r,\varphi,P_r, P_\varphi)\}$. For each fixed $C$, the implicit function theorem gives a branch
$P_\varphi=P_\varphi(r,\varphi,P_r,C)$ near parabolic infinity
on $\Pi_C$ and hence a $3$-dimensional differential equation on
$(r,\varphi,P_r)$. On a sufficiently small neighborhood of this
branch at infinity, $\dot\varphi=-1+P_\varphi/r^2<0$ and
$\varphi(t)$ is strictly decreasing. Choose a $2$-dimensional
surface $S=\{\varphi=0\} \subset \Pi_C$ and a Poincar\'e return
map
 \begin{eqnarray}  \label{RPC3BP-map}
 f_{\mu,C}:(r,P_r)\mapsto (r',P_r'),
 \end{eqnarray}
 where a trajectory of the RPC3BP with an initial condition
 $(r,0,P_r,P_\varphi(r,0,P_r,C))$ that passes through
 $(r',-2\pi,P'_r,P_\varphi(r',-2\pi,P'_r,C))$.
 This gives rise to {\it an area-preserving map}
$f_{\mu,C}:U\to \mathbb{R} ^2$ defined on an open set $U\subset \mathbb{R} ^2$.
When $C$ is fixed, we also write $f_\mu=f_{\mu,C}$.
The rotating Hamiltonian in \eqref{Hamiltonian} is also denoted by
$H_{\rm rot}$, so $C=-2H_{\rm rot}$.

\subsection{Newhouse domains for area-preserving maps}

We say that a saddle periodic point $p$ of an area-preserving map
$f$ exhibits a {\it homoclinic tangency} (HT) if stable and
unstable manifolds $W^s(p)$ and $W^u(p)$ of $p$ respectively have
a point of tangency. We say that $f$ has an HT if some of its saddle
points has an HT. Denote by $\mathcal{HT} $ the closure of the set of
area-preserving maps with HT. It seems that appearance of an HT
is a codimension-$1$ phenomenon and can be destroyed for
an individual saddle. It turned out that $\mathcal{HT} $ has
nonempty interior.  This phenomenon was discovered first for
dissipative $2$-dimensional  maps by Newhouse \cite{n1,n2,n3}.
It took over two decades to extend it to area-preserving setting.
This was done by Duarte \cite{Du1}.

{A related dimension result is due to Downarowicz and
Newhouse~\cite{DN2005}: generically among surface diffeomorphisms,
homoclinic closures containing tangencies have Hausdorff dimension
two. Such genericity in the space of maps does not by itself give
genericity in a specified physical parameter family, nor does it
identify the dimension of the oscillatory subset of a homoclinic
class. These are the additional questions addressed below.}

Call an open set with a dense subset of maps with an HT a {\it
Newhouse domain}; here this term does not require connectedness.
An open parameter Newhouse set is a union of interval components,
each with dense tangency parameters. One of the main results of this paper is a proof of
existence of Newhouse domains for the Sitnikov example and the
RPC3BP. We shall also prove that

\begin{Theorem}  \label{parameter-Newhouse-Sitnikov} Let $\{f_e\}_{0<e<1}$
be the family of maps (\ref{sitnikov-map}). Then there is a Newhouse
domain $\mathcal{N} \subset (0,1)$ accumulating at zero, i.e. for a dense set of $e$ in $\mathcal{N} $
the Poincar\'e map $f_e$ has an HT. \end{Theorem}

\begin{Theorem}[RPC3BP parameter Newhouse domains]
\label{parameter-Newhouse-RPC3BP}
For every sufficiently large positive $C$, the mass-parameter family
$\{f_{\mu,C}\}_{\mu\in(0,1)}$ has a nonempty open Newhouse
domain $\mathcal N_C\subset(0,1/2)$ accumulating at zero.
For every fixed $\mu\in(0,1)$, the energy-parameter family has
Newhouse domains whose union $\mathcal N_\mu$ is unbounded
in the positive $C$ direction.
\end{Theorem}

Robinson \cite{R}, using ideas of Newhouse, showed that for a generic
$1$-parameter unfolding a homoclinic tangency there are Newhouse domains
on the parameter line. In a sense we prove a similar statement. Namely,
the issue is nondegeneracy in the specified one-parameter family.
For the RPC3BP the global input is separated explicitly in
Section~\ref{ss:rpc-parameter-status}.  The proofs of these two theorems are based on
the conservative homoclinic bifurcation results of Duarte
\cite{Du2,Du3,Du4} and their dimension refinement in \cite{Go}.

\subsection{Scheme of the proof}

Since the construction consists of several steps, we describe here
its main ingredients. The local passage estimates and the dimension
argument are common to both problems. The construction of tangencies
requires separate arguments for the eccentricity, mass, and
Jacobi-constant parameters.

Step 1. Using McGehee coordinates, we represent a parabolic end by
a degenerate saddle $O_\infty$, whose stable and unstable manifolds
correspond to parabolic motions; see \cite{Mo} and Section 2.
For the full Sitnikov phase plane there are two such ends, denoted
by $O_+$ and $O_-$, corresponding to $z=+\infty$ and
$z=-\infty$, respectively. The symbol $O_\infty$ denotes the
selected end. To obtain two-sided oscillatory motions, we will
construct orbits that accumulate both at a parabolic end and at a
fixed finite periodic orbit in positive and negative time. The latter
condition guarantees infinitely many returns to a bounded region.

Step 2. Zero eccentricity in the Sitnikov problem and zero mass
parameter in the RPC3BP correspond to integrable systems. In the
Sitnikov problem, the selected branch of $W^u(O_+)$ coincides with
a branch of $W^s(O_-)$; symmetry supplies the reverse connection.
The Melnikov calculation in \cite{GP} gives transverse heteroclinic
intersections for small positive eccentricity. The local inclination
estimates then give transverse homoclinic points for each parabolic
end (Lemma~\ref{lem:transverse-return}). In the RPC3BP the
unperturbed connection is homoclinic to infinity. We use the
expansion in \cite{MP} and the splitting estimates in
\cite{GMS2016}; the latter also supply transverse intersections
for each fixed positive mass and sufficiently large $C$.

Step 3. We study local passages near $O_\infty$. The classical
hyperbolic inclination lemma does not apply at this degenerate
point. We prove convergence of passage images to the unstable
manifold in the $C^2$ metric on outgoing intervals away from the
vertex (Theorems~\ref{t.5} and \ref{t.c2lyambdalemma}). We also
establish quantitative stretching and projective contraction,
shape estimates, and parameter estimates; see
\eqref{stretching}, \eqref{projective-hyperbolicity},
\eqref{e.shape}, and \eqref{e.derivine}.
We first treat a simplified system. For the general system,
fixed barriers and integrable error estimates give bounds uniform
in the passage length on a sufficiently small local square
(Section~\ref{sec:revised-general}). Slopes, curvatures, and
parameter derivatives refer to images at a common physical time;
the rescaled time is used only to estimate the resulting equations.

Step 4. We compare a global splitting graph with a local passage
image of the symmetry curve. An interior contact gives a quadratic
parabolic tangency, and parameter estimates show that it unfolds
nondegenerately. For Sitnikov this produces parameters $e_n\to0$
(Theorem~\ref{quad-tangency}). For the RPC3BP we distinguish two
regimes. With $C$ sufficiently large and fixed, the first-order
mass splitting gives tangencies at $\mu_n\to0$
(Proposition~\ref{prop:rpc-fixed-energy-tangencies}). With
$\mu>0$ fixed, we instead construct tangencies at $C_n\to\infty$
(Proposition~\ref{prop:rpc-fixed-mass-tangencies}). This second
argument requires a complex-parameter extension of the splitting
estimates and uniform matching with the scaled local coordinates;
see Sections~\ref{ss:complex-splitting-control} and
\ref{ss:rpc-uniform-energy}. At equal masses we use the second
harmonic, with its separate normalization. Uniformity in this
regime is asserted for each fixed mass.

Step 5. We transfer the parabolic tangencies to hyperbolic periodic
saddles at nearby parameters, retaining transverse connections to
$O_\infty$ in both directions. Spatial $C^2$ approximation alone
would not preserve nondegenerate unfolding. The additional
parameter estimates in Section~\ref{ss.saddlesnearhomoclinicpoint}
provide this transfer and yield Proposition~\ref{saddle} for
Sitnikov. The same approximation is applied at each finite
parameter in the two RPC3BP tangency constructions.

Step 6. We apply the conservative homoclinic bifurcation
construction of Duarte~\cite{Du2,Du3,Du4} and its dimension
refinement in \cite{Go}, in the local form stated in
Theorem~\ref{t.thick-sets}. For every $\delta>0$, it gives nearby
open parameter sets with hyperbolic basic sets containing the
continued saddle, of dimension greater than $2-\delta$, and
exhibiting persistent tangencies. We first fix a parameter Newhouse
domain using one tolerance. Reapplying the construction near its
dense set of nondegenerate tangencies gives, for each $k$, an open
dense subset admitting such basic sets of dimension $>2-1/k$
(Remark~\ref{rem:parameter-density}). All parameter neighborhoods
are chosen so that the selected transverse connections to infinity
persist. The RPC3BP applications use the direct tangency
constructions in Sections~\ref{ss:fixed-energy-completion} and
\ref{ss:rpc-uniform-energy}.

Step 7. From the transverse connections to $O_\infty$ and the
local inclination estimates, we obtain transverse homoclinic
points of the finite periodic saddle accumulating at infinity.
The proof of Theorem~\ref{2-diml-hd} in
Section~\ref{ss:sparse-excursions} explains this step.

Step 8. We approximate a given hyperbolic basic set in dimension
by a return horseshoe and insert prescribed finite excursions
among long free blocks. These excursions visit shrinking
neighborhoods of infinity and of the finite periodic orbit in
both time directions. Lemma~\ref{lem:sparse-deletion} bounds
the geometric cost on entire intervals, including gaps, and
shows that the deletion map onto the return horseshoe is H\"older
with every exponent below one. Proposition~\ref{prop:dimension-excursions}
therefore gives oscillatory sets with Hausdorff dimension at least
that of the given basic set.

Step 9. On the residual intersection of the open dense parameter
sets from Step 6, the basic-set dimensions tend to two. Applying
Step 8 gives full Hausdorff dimension for two-sided oscillatory
motions on the section. The Sitnikov parameter argument is completed
in Section~\ref{ss.summingup}; the two RPC3BP regimes are completed
in Sections~\ref{ss:fixed-energy-completion} and
\ref{ss:rpc-uniform-energy}. Finally, a smooth flow box and the
product formula \eqref{eq:interval-product-dimension} raise the
dimension from two to three in the extended Sitnikov phase space
or on the RPC3BP energy surface. For the RPC3BP return map, the
proof also checks that the prescribed returns occur at physical
times tending to infinity in both directions.

\medskip

{\bf Acknowledgements.} The authors would like to thank Vassili
Gelfreich for helping us to understand the separatrix splitting.
Discussions of thermodynamic formalism with Omri Sarig
were  enlightening and  valuable. We are grateful to
Don Saari and Sheldon Newhouse for lively interest, support, and
encouragement; they also provided several useful references.
We also thank Pedro Duarte, Mark Levi, and Dmitry Turaev for their interest
in our work and useful discussions.

During the development of this work, A.G. was supported in part by NSF grants DMS--0901627 and IIS-1018433, and also DMS--2247966.
 V.K. was supported in part by NSF grant DMS-1101510, and also the ERC Advanced Grant SPERIG (nr. 885707).

\medskip

{\bf Use of AI tools.} ChatGPT was used during preparation of the current
version of the manuscript to assist with exposition and refinement of selected proof details.

\section{McGehee coordinates}

\subsection{Parabolic motions and McGehee coordinates for the Sitnikov problem}

For the Sitnikov problem equations of motion are

 \begin{eqnarray}  \label{Sitnikov}
 \ddot z+ \dfrac{z}{(z^2+r_e^2(t))^{3/2}}=0, \quad r =
 \dfrac  12 (1-e \cos t)+ O(e^2).
 \end{eqnarray}

After  McGehee's transformation
\[
\adjustbox{max width=\linewidth}{$\displaystyle
z= \dfrac {2}{q^2} \quad \dot z= - p, \quad ds = \frac{q^3}{4}\,dt,
$}
\]
the new equations of motion become
 \begin{eqnarray}  \label{rescaled-pq-equations-of-motion} \begin{cases}
 \dfrac{dq}{ds}=p \\
 \dfrac{dp}{ds}=q\left(1+\dfrac{q^4}{4} r^2 \right)^{-3/2}\\
 \dfrac{dt}{ds}=\dfrac{4}{q^3}.
 \end{cases} \end{eqnarray}

By McGehee's theorem  \cite{McG} this equation has stable and unstable parabolic branches
meeting at the origin, with saddle-type passage dynamics in the physical
sector. These invariant manifolds are $C^\infty$ curves, analytic away
from the origin. This assertion concerns the physical sector, not an
isolated saddle on a full neighborhood of the smooth infinity boundary. One can provide more detailed description of these manifolds (that correspond to parabolic motions of the Sitnikov system):

\begin{Theorem} [Moser \cite{Mo} ch III, sect 2.b)]\label{t.Moser}
 Local unstable and stable manifolds of the system
 \begin{eqnarray}  \label{e.oldtime} \begin{cases}
 \dfrac{dq}{dt}=\dfrac{pq^3}{4} \\
 \dfrac{dp}{dt}=\dfrac{q^4}{4}\left(1+\dfrac{q^4}{4} r^2(t) \right)^{-3/2}
 \end{cases} \end{eqnarray}
 are $C^\infty$-manifolds of the form $q=\chi(p, t)$ and $q=\chi(-p, -t)$, real analytic away from the origin. The Taylor expansion of the function $\chi(p, t)$ has the form
 \[
\adjustbox{max width=\linewidth}{$\displaystyle
q=\chi(p, t)=p+a_4p^5+a_7(t)p^8+\ldots
$}
\]
 where $a_4=\frac{1}{32\pi}\int_{0}^{2\pi}r^2(t)dt$, and $a_7(t)$ is such that $\int_0^{2\pi}a_7(t)dt=0$ and $\frac{d}{dt}a_7(t)=\frac{3}{32}r^2(t)-\frac{3}{64\pi}\int_{0}^{2\pi}r^2(t)dt$.
 \end{Theorem}
Since the complete proof of this statement is not provided in \cite{Mo} (or anywhere else), we provide it here.

\begin{proof}[Proof of Theorem \ref{t.Moser}]
The statements about existence and smoothness (and analyticity) of the invariant manifolds follow from the results in \cite{McG} and \cite{Sl}. Therefore we a priori know that the unstable manifold has the form $q=\chi(p, t)$, and the function $\chi(p, t)$ has the Taylor expansion
\[
\adjustbox{max width=\linewidth}{$\displaystyle
q=\chi(p, t)=p(1+a_1(t)p+a_2(t)p^2+a_3(t)p^3+a_4(t)p^4+a_5(t)p^5+a_6(t)p^6+a_7(t)p^7+\ldots ),
$}
\]
where $a_i(t)$ are $2\pi$-periodic smooth functions. Notice that since we cannot claim that $\chi(p, t)$ is analytic at $p=0$, we cannot assume that the series converges, but we can still use the Taylor formula with as many terms as we need. Take any point $(q, p, t)$ that belongs to the unstable manifold $q=\chi(p, t)$. Let us consider the surface in $(q, p, t)$-space that is formed by all the local unstable manifolds, $S=\cup_{t\in[0, 2\pi)}Graph_{\chi(p, t)}$. The normal vector to $S$ at a  point $(q, p, t)$ has the form
\[
\adjustbox{max width=\linewidth}{$\displaystyle
\begin{gathered}
    N_{(q, p, t)}=\left(-1, \frac{\partial \chi}{\partial p}, \frac{\partial \chi}{\partial t}\right)=\\
=\left(-1, 1+2a_1(t)p+3a_2(t)p^2+4a_3(t)p^3+5a_4(t)p^4+6a_5(t)p^5+\ldots,\right. \\  \left.a_1'(t)p^2+a_2'(t)p^3+a_3'(t)p^4+a_4'(t)p^5+a_5'(t)p^6+\ldots\right),
\end{gathered}
$}
\]
which must be orthogonal to the vector of the autonomous vector field
\begin{eqnarray}  \nonumber \begin{cases}
 \dfrac{dq}{d\tau}=\dfrac{pq^3}{4} \\
 \dfrac{dp}{d\tau}=\dfrac{q^4}{4}\left(1+\dfrac{q^4}{4} r^2(t) \right)^{-3/2}\\
 \dfrac{dt}{d\tau}=1
 \end{cases} \end{eqnarray}
Taking into account that $(1+\alpha)^{-3/2}=1-\frac{3}{2}\alpha+\frac{15}{8}\alpha^2-\ldots$, this orthogonality condition is equivalent to the following:
\[
\adjustbox{max width=\linewidth}{$\displaystyle
\begin{gathered}
  -pq^3+\left(1+2a_1(t)p+3a_2(t)p^2+4a_3(t)p^3+5a_4(t)p^4+\ldots\right)\cdot q^4\left(1-\frac{3}{8}q^4r^2(t)+\ldots\right)+\\
  +4\left(a_1'(t)p^2+a_2'(t)p^3+a_3'(t)p^4+a_4'(t)p^5+a_5'(t)p^6+\ldots\right)=0
\end{gathered}
$}
\]
After we substitute $q=p(1+a_1(t)p+a_2(t)p^2+a_3(t)p^3+a_4(t)p^4+a_5(t)p^5+a_6(t)p^6+a_7(t)p^7+\ldots )$ and represent the left hand side as a power series in $p$, all the coefficients must be equal to zero. In particular, we get:
\begin{itemize}
\item
Coefficient at $p^2$: $a_1'(t)=0$, hence $a_1(t)$ is a constant;

\item
Coefficient at $p^3$: $a_2'(t)=0$, hence $a_2(t)$ is a constant;

\item
Coefficient at $p^4$: $a_3'(t)=0$, hence $a_3(t)$ is a constant;

\item
Coefficient at $p^5$: $4a_1+2a_1-3a_1+4a_4'(t)=0$, hence $a_1=0$, and $a_4(t)$ is a constant (otherwise $a_4(t)$ cannot be a periodic function);

\item
Coefficient at $p^6$: $4a_2+6a_1^2+2a_1^2+3a_2-(3a_2+3a_1)+4a'_5(t)=4(a_2+a_5'(t))=0$, hence $a_2=0$, and $a_5(t)$ is a constant;

\item
Coefficient at $p^7$: $4a_3+4a_3-3a_3+4a_6'(t)=5a_3+4a_6'(t)=0$, hence $a_3=0$, and $a_6(t)$ is a constant;

\item
Coefficient at $p^8$: $4a_4+5a_4-\frac{3}{8}r^2(t)-3a_4+4a_7'(t)=6a_4-\frac{3}{8}r^2(t)+4a_7'(t)=0$. Since $\int_0^{2\pi}a_7'(t)dt=0$, we have $\int_0^{2\pi}(6a_4-\frac{3}{8}r^2(t))dt=0$, hence $a_4=\frac{1}{32\pi}\int_0^{2\pi}r^2(t)dt$;

\item
Coefficient at $p^9$: $4a_5+6a_5-3a_5+4a_8'(t)=7a_5+4a_8'(t)=0$, hence $a_5=0$, and $a_8(t)$ is a constant;

\item
Coefficient at $p^{10}$: $4a_6+ 7a_6-3a_6+4a_9'(t)=8a_6+4a_9'(t)=0$, hence $a_6=0$, and $a_9(t)$ is a constant;

\item
Coefficient at $p^{11}$: $4a_7(t)+8a_7(t)-3a_7(t)+4a_{10}'(t)=9a_7(t)+4a_{10}'(t)=0$, hence $\int_0^{2\pi}a_7(t)dt=0$.

This completes the proof of Theorem \ref{t.Moser}.
\end{itemize}

\end{proof}

   According to Moser \cite{Mo} ch III, sect 2.b)
these manifolds have the following form
\[
\adjustbox{max width=\linewidth}{$\displaystyle
q\ = \ \chi(p,t)\ =\ p+a_4 p^5 + a_7(t) p^8...
$}
\]
and
\[
\adjustbox{max width=\linewidth}{$\displaystyle
q= \chi(-p,-t)=-p-a_4 p^5 + a_7(-t) p^8 ...,
$}
\]
where $a_4$ is some nonzero constant and $a_7(t)$ is a time-periodic
function. Expressions for the invariant manifolds show that for
\[
\adjustbox{max width=\linewidth}{$\displaystyle
x= \dfrac 14 (q-\chi (-p,-t))=\dfrac 14 (p+q) + \dfrac{a_4}{4} p^5 -
\dfrac{a_7(-t)}{4} p^8\cdots
$}
\]
\[
\adjustbox{max width=\linewidth}{$\displaystyle
y\ =\  \dfrac 14 (q-\chi (p,t))\ =\ \dfrac 14 (q-p) - \dfrac{a_4}{4}
p^5 -\dfrac{a_7(t)}{4} p^8  \cdots
$}
\]
they are represented by $x$ and $y$ coordinate axes.

\paragraph{Definition of the weighted remainder classes.}
In order to simplify notation we introduce the following classes.
Here $r=x+y$ is a local size variable, distinct from the physical
polar radius or the Sitnikov primary radius. We work on
$x,y\ge0$, $0<r<\delta$. We say that a function $f(x,y,t,e)$,
smooth on this punctured quadrant, belongs to $O_n$ if the following
bounds hold:
\[
 |\partial_x^i\partial_y^j\partial_t^\ell\partial_e^a f|
 \le C_{i,j,\ell,a}r^{n-i-j}.
\]
As part of this definition, we require these estimates for all the
derivatives used below (at least
three spatial derivatives and one parameter derivative), uniformly in the
periodic variable $t$ and in a fixed compact parameter interval.
Thus the derivative bounds are requirements defining membership in
$O_n$, rather than consequences of the bound $|f|\le Cr^n$ alone.
Smooth extension through the corner is not required by this definition.
For integers $n<m$, we define $f\in O_{n,m}$ to mean that
$f=f_1(x,y,e)+f_2(x,y,t,e)$, where $f_1\in O_n$ and $f_2\in O_m$.
The same definition with a real exponent $\eta>0$ gives the class
$O_\eta$. The weighted bounds do not require the quotient to extend
smoothly through the corner. Positive constants denoted by $c$, $K$,
or $C$ in local estimates may change between occurrences; their
parameter dependence is specified in each argument. In RPC3BP
parameter statements, $C$ always denotes the Jacobi constant.

\begin{Lemma}\label{lem:remainder-calculus}
The following rules hold, with the indicated derivative bounds:
\begin{enumerate}
\item $\partial_x O_{n,m},\partial_y O_{n,m}\subset O_{n-1,m-1}$;
\item $(ax+by)O_{n,m}\subset O_{n+1,m+1}$;
\item $O_{n,m}O_{k,p}\subset O_{n+k,\min(n+p,m+k)}$;
\item if $a,b>0$, then $O_{n,m}/(ax+by)\subset O_{n-1,m-1}$
 on the punctured quadrant;
\item if $f$ is smooth at the corner, its spatial Taylor jet of order
 $n-1$ vanishes, and $f(0,y,t,e)=0$, then $f/x$ is smooth
 and its Taylor jet of order $n-2$ vanishes. The analogous statement
 holds for division by $y$.
\end{enumerate}
\end{Lemma}
\begin{proof}
The first three assertions follow from differentiation and the product
rule, keeping both mixed products $f_1g_2$ and $f_2g_1$. For the fourth,
$ax+by$ is comparable with $r$, and every derivative of its reciprocal
has the corresponding homogeneous bound. For the last assertion use
\[
 f(x,y,t,e)=x\int_0^1\partial_xf(ux,y,t,e)\,du.
\]
For the finitely many derivatives needed here, the integral gives the
claimed bounds whenever the underlying vanishing order is sufficiently
large; in the applications the numerator is a smooth remainder of order
at least six. Equivalently, Taylor's formula first removes its vanishing
jet and then Hadamard's formula divides the remainder by the invariant
axis. No division by $x$ without this axis condition is permitted.
\end{proof}
In particular, division by $x+y$ need not preserve smoothness at the
corner. Also, $(x+x^2\sin t)^2$ illustrates why the time-dependent order
in the product rule is not $m+p$.
As before, $A_{n,m}$ denotes $A+O_{n,m}$; each occurrence can denote a
different remainder. We use $O_n$ when no separation of time dependence
is needed.

\subsubsection{Derivation of equations of motion in the $xy$-coordinates}
We use the invariant manifolds to straighten the axes exactly, rather
than divide a truncated Taylor polynomial by its leading linear term.
Let
\[
 x=\frac{q-\chi(-p,-t)}4,\qquad
 y=\frac{q-\chi(p,t)}4.
\]
The linear part is invertible. The inverse function theorem and the
expansions above give
\[
 p=2(x-y)+O_5,\qquad q=2(x+y)+O_8.
\]
The estimates here and below are estimates for derivatives, not merely
pointwise Taylor remainders. The identity $q=2(x+y)$ is not used.

Write the physical-time vector field as $F=\dot x$, $G=\dot y$.
Its leading terms, calculated from (\ref{e.oldtime}), are
\[
 F=\frac{xq^3}{4}+O_8=2x(x+y)^3+O_8,
 \qquad G=-2y(x+y)^3+O_8.
\]
Indeed, differentiating a correction of order five in the spatial
coordinates against a vector field of order four gives order eight;
the first time-dependent correction also has order eight. Since
$x=0$ and $y=0$ are exact invariant axes, respectively $F$ and $G$
vanish there. The remainders after subtraction of the displayed leading
terms therefore vanish on the corresponding axes. Hadamard's formula,
followed by division by $(x+y)^3$ on the punctured quadrant, gives
\begin{equation}\label{sitnik-1}
 \dot x=2x(x+y)^3(1+A_e),\qquad
 \dot y=-2y(x+y)^3(1+B_e),\qquad A_e,B_e\in O_4.
\end{equation}
The two remainders are independent. Along an individual orbit put
$ds=2(x+y)^3dt$. Then
\begin{equation}\label{sitnik-2}
 x_s=x(1+A_e),\qquad y_s=-y(1+B_e),\qquad
 t_s=\frac{1}{2(x+y)^3}.
\end{equation}
The last equation must be retained when the remainders depend on time.
Derivatives of fixed physical-time maps are computed before this change
of the independent variable.

\begin{Lemma}\label{lem:uniform-separatrices}
The local invariant graphs used above, and the invariant graphs used in
the RPC3BP reduction below, have finite Taylor expansions of arbitrarily
high order with uniform weighted remainders after any fixed finite number
of spatial, phase, and parameter derivatives. Parameters range over a
compact subset of $|e|<1$ in the Sitnikov problem and over a fixed compact
set of $(\mu,C)$ in the RPC3BP. The local neighborhood may depend on
that compact set.
\end{Lemma}
\begin{proof}
We apply the parameter-dependent parabolic invariant-manifold result
\cite{BFM}, Corollary 2.9 (with the weighted remainder conclusion of
Theorem 2.8). We verify its hypotheses for the fixed-period map, without
first straightening the invariant manifolds. Use linear coordinates
$(v,w)$ with $v$ the contracting coordinate and $w$ the expanding one.
For a fixed period $P$ the map has the form
\[
 (v,w)\longmapsto(v+p(v,w)+f(v,w),\ w+q(v,w)+g(v,w)),
\]
where
\[
 p=-c v(v+w)^3,\qquad q=c w(v+w)^3,\qquad c>0,
 \qquad f,g=O_5.
\]
For Sitnikov, $v=(q-p)/4$, $w=(q+p)/4$ and $c=2P$.
For RPC3BP, $v=u+P_r/\sqrt2$, $w=u-P_r/\sqrt2$, and, with
$\theta=-\varphi$ as time, $c=P/(8\sqrt2)$. The energy equation is
solved near infinity using $\partial H/\partial P_\varphi=-1+O(u^4)$.
The vector fields, and hence their fixed-period maps, are smooth with
uniform bounds on compact parameter sets; all derivatives of $f,g$
have the corresponding Taylor bounds. Initial phase can be included
among these parameters.

Using definitions (2.4)--(2.5) and their parameter versions (2.12)
in \cite{BFM}, the constants on
$V=(0,\rho)$ with the absolute-value norm, are
\[
 N=M=4,\quad a_p=b_p=c_p=d_p=c,\quad
 A_p=B_p=4c,\quad B_q=c.
\]
Indeed, $p(v,0)=-cv^4$, $D_vp(v,0)=-4cv^3$ and
$D_wq(v,0)=cv^3$. Also $q(v,0)=0$, and for small $\rho$,
\[
 \operatorname{dist}(v-cv^4,\mathbb R\setminus(0,\rho))
 \ge(c/2)v^4.
\]
Thus the contracting-domain condition holds, and the remaining
inequality for equal degrees is $2+B_q/c_p=3>
\max\{1-A_p/d_p,0\}=0$. These constants are independent of the
physical parameters and initial phase. Corollary 2.9(4) applies because $A_p\ge d_p$ and the map
is smooth in space and parameters. Weighted derivative control is
provided by Theorem 2.8 and the remainder spaces in Proposition 7.1;
the approximation order is chosen after the required derivative order. In one contracting dimension the homogeneous approximations
are polynomials in $v$; taking their order arbitrarily high gives the
weighted Taylor bounds just stated. The graph projection is invertible
because its derivative at zero is one. Apply the same argument to the
inverse map for the unstable graph.

Uniqueness of the local stable and unstable graphs identifies their
bounded-time translates with the graphs at the translated phase.
Consequently the invariant graphs for the maps give periodic invariant
surfaces for the vector field. Their coefficients can therefore be
computed from the continuous-time invariance equation, as in the proof
of Theorem~\ref{t.Moser}. The uniform remainder estimates justify that
coefficient calculation and its parameter derivatives.
\end{proof}
\begin{Remark}\label{rem:normal-input}
Lemma~\ref{lem:uniform-separatrices} supplies the parameter-uniform
input in the preceding normal-form calculation. When evaluating a
one-sided graph outside its physical branch in the coordinate formula,
choose a finite-order smooth extension with the same Taylor jet and
parameter bounds; invariance is used only on the physical axis.
The derivative order can be chosen larger than every order used in
this paper. No convergence of the Taylor series at infinity is assumed.
\end{Remark}

In order to obtain a similar expression for the RPC3BP we discuss
the Kepler problem first. Then we come back to RPC3BP and notice that
it is a small perturbation of the Kepler problem for small mass ratio $\mu$.

\subsection{The Kepler problem (2BP) and the polar coordinates}
\label{sympl-polar} Recall that the polar coordinates for the Kepler
problem are formed by the polar coordinates $(r,\varphi)$ and the
momenta $(P_r,P_\varphi)$ conjugate to $(r,\varphi)$ resp. The
Hamiltonian for the RPC3BP in rotating
 polar
coordinates $(r,\varphi,P_r,P_\varphi)$ is
\begin{eqnarray*}
H_0(r, \varphi, P_r, P_\varphi)  = \frac{P_r^2}{2} +
\frac{P_\varphi^2}{2r^2} - P_\varphi - \frac{1}{r}.
\end{eqnarray*}

Notice that
 \newline
 --- angular momentum $P_\varphi$ is the first integral;
 \newline
 --- at fixed angular momentum $P_\varphi=L$, level sets of
 $E=H_0+L=P_r^2/2+L^2/(2r^2)-1/r$ in the $(r,P_r)$-plane
 describe the radial Kepler motion;
 \newline
 --- for $L\ne0$ and negative inertial energy $E<0$, nonempty
 radial energy levels are compact and the collision-free inertial
 trajectories are periodic (including the circular orbit);
 \newline
 --- motions of the Kepler problem are conic sections (see e.g.
 \cite{AKN});
 \newline
 --- these bounded collision-free trajectories are ellipses in the
 inertial plane $\mathbb{R} ^2\setminus \{0\}$. When $L=0$,
 collision orbits require separate treatment. The sign of the rotating
 Hamiltonian $H_0$ alone does not distinguish bounded from unbounded motion;
 \newline
 --- in the inertial coordinate system the Hamiltonian becomes
 $H_0+P_\varphi$ and both future and past parabolic motions coincide and
 correspond to the curve $\{H_0+P_\varphi=0\}$.

\subsection{McGehee coordinates for the Restricted Planar Circular 3 Body Problem}

The Hamiltonian for RPC3BP in rotating polar coordinates
$(r,\varphi,P_r,P_\varphi)$ is given by (\ref{Hamiltonian}).
So the equations of motion become
 \begin{eqnarray}  \label{EM}
  \begin{cases}
   \begin{array}{rcl}
   \dot r & = & P_r \\
   \dot P_r & = & P_\varphi^2\, r^{-3}-r^{-2}-\partial_r \Delta H \\
   \dot \varphi & = & -1+P_\varphi \, r^{-2} \\
   \dot P_\varphi & = & - \partial_\varphi \Delta H.
   \end{array}
  \end{cases}
 \end{eqnarray}

Directly one could prove the following bounds on the perturbation
term $\Delta H$ and its derivatives satisfy
 \begin{Lemma}  \label{apriori-estimates-noE} \cite{GaK} \ For any $r> 1$
  \begin{equation}
  \begin{aligned}
   \max_\varphi|\Delta H(r,\varphi)|
      &\le\frac{\mu}{r^2(r-1)},\\
   \max_\varphi|\partial_r\Delta H(r,\varphi)|
      &\le\frac{\mu(3r-2)}{r^3(r-1)^2}
       \le\frac{3\mu}{r^2(r-1)^2}.
  \end{aligned}
  \end{equation}
 \end{Lemma}
\begin{proof}
Let $P_n$ denote the Legendre polynomial of degree $n$. For $r>1$,
\[
 \Delta H=-\sum_{n\ge2}c_nP_n(\cos\varphi)r^{-n-1},\qquad
 c_n=\mu(1-\mu)^n+(1-\mu)(-\mu)^n.
\]
The terms of degrees zero and one cancel. For $n\ge2$,
\[
 |c_n|\le\mu(1-\mu)
       \bigl((1-\mu)^{n-1}+\mu^{n-1}\bigr)\le\mu.
\]
Since $|P_n(t)|\le1$ for $|t|\le1$, summing the geometric series
proves the first bound. Termwise radial differentiation is justified
uniformly on every region $r\ge1+\epsilon$, $\epsilon>0$, and gives
\[
 |\partial_r\Delta H|\le\mu\sum_{n\ge2}(n+1)r^{-n-2}
       =\frac{\mu(3r-2)}{r^3(r-1)^2}.
\]
The final inequality follows from $3r-2\le3r$.
\end{proof}

Notice that in the rotating polar coordinates $H$ is $\mu$-close to
the Kepler Hamiltonian $H_0$ on compact collision-free regions with
$r>1$. Put $E=H+P_\varphi$. For $\mu=0$ this is the inertial Kepler
energy, and parabolic motions have $E=0$. For $\mu>0$ it is not in
general a first integral. On the Jacobi surface $H=-C/2$ one has
\[
 2P_\varphi=C+2E.
\]
Thus a bound $|E|\le\varepsilon$ gives
$|2P_\varphi-C|\le2\varepsilon$; an $O(\mu)$ conclusion requires
a separate $O(\mu)$ bound on $E$.

For each fixed $C$, we can express $P_\varphi$ as an implicit
function near parabolic infinity. Indeed, (\ref{Hamiltonian}) gives
\[
 \frac{P_\varphi^2}{r^2}-2P_\varphi-\frac2r
                  +2\Delta H+P_r^2=-C.
\]
In the coordinates $u=r^{-1/2}$, at $(u,P_r)=(0,0)$ the relevant
solution is $P_\varphi=C/2$, and the derivative of the left side
with respect to $P_\varphi$ is $-2$. The implicit function theorem
therefore gives a smooth branch
\[
 P_\varphi=C/2+O(u^2+P_r^2)
\]
near infinity, uniformly on fixed compact parameter sets. On a
sufficiently small such neighborhood,
$\dot\varphi=-1+P_\varphi u^4<0$.
This local reduction requires no global periapsis estimate.
For comparison, an unperturbed Kepler orbit with angular momentum
$L\ne0$ and energy $E$ has
\[
 r_{\min}=\frac{L^2}{1+\sqrt{1+2EL^2}}.
\]
On the parabolic branch $E=0$, $L=C/2$, this becomes $C^2/8$.
A fixed neighborhood of $E=0$ does not by itself give an
$O(\mu)$ perturbation of that value.

Introduce $u=r^{-1/2}$ along with a function
$d_r(u,\varphi)=u^{-3}\partial_r \Delta H(u^{-2},\varphi)$
which are well defined for $u\ge 0$ and $\phi\in \mathbb{T} $. By lemma
\ref{apriori-estimates-noE} we have that $d_r(u,\varphi)$ has
$u$-zero of order $5$ at $u=0$.
Plug $u$ into the equations of motion (\ref{EM}):
 \begin{eqnarray}  \label{EM1}
  \begin{cases}
   \begin{array}{rcl}
   \dot u & = & -\dfrac 12 P_r \,u^3 \\
   \dot P_r & = & -u^3(u-P_\varphi^2\,u^3+d_r(u,\varphi))\\
   \dot \varphi & = & -1+P_\varphi\, u^4.
    \end{array}
  \end{cases}
 \end{eqnarray}

By McGehee's theorem  \cite{McG} this equation has stable and unstable parabolic branches
meeting at the origin, with saddle-type passage dynamics in the physical
sector. The invariant manifolds are analytic away from the origin.
The statement does not assert that the origin is isolated in the
full smooth extension at infinity. We would like to write this system in the form similar to
(\ref{sitnik-1}) and (\ref{sitnik-2}).

For the forward-time expanding coordinate the correct leading choice is
$X=u-P_r/\sqrt2$, $Y=u+P_r/\sqrt2$. Direct substitution in
(\ref{EM1}) gives, with $u=(X+Y)/2$,
\[
 \dot X=\frac{(X+Y)^3X}{8\sqrt2}+O_6,\qquad
 \dot Y=-\frac{(X+Y)^3Y}{8\sqrt2}+O_6.
\]
Here $C$ is fixed in a compact parameter interval; no uniformity as
$C\to\infty$ is asserted. The angular momentum is eliminated on
$H=-C/2$ by the implicit function theorem near infinity, where
$\partial H/\partial P_\varphi=-1+O(u^4)$ and $P_\varphi=C/2+O(u^2+P_r^2)$.

By Lemma~\ref{lem:uniform-separatrices}, the exact local invariant surfaces are
$X=\sigma_s(Y,\varphi)$ and $Y=\sigma_u(X,\varphi)$, with
$\sigma_s,\sigma_u=O_3$ and $\partial_\varphi\sigma_s,
\partial_\varphi\sigma_u=O_6$, in the weighted derivative sense.
To check these orders, expand $\sigma_s(Y,\theta)=\sum_{j\ge2}a_j(\theta)Y^j$.
In the invariance equation the coefficients of degrees two through
four give $a_2'=a_3'=a_4'=0$. There is no degree-five forcing in the
unstraightened field, whose first correction has order six. The
coefficient of degree five is $a_5'=3c_{\rm McG} a_2$, where
$c_{\rm McG}=1/(8\sqrt2)$. Periodicity forces $a_2=0$ and $a_5'=0$.
The unstable calculation uses reversed time and is identical.
These equalities also hold after parameter differentiation.
Set $x=X-\sigma_s(Y,\varphi)$ and
$y=Y-\sigma_u(X,\varphi)$. This is an invertible change near the origin.
The same axis-division argument as above proves
\begin{equation}\label{RPC3BP-1}
 \dot x=\frac{x(x+y)^3}{8\sqrt2}(1+A),\qquad
 \dot y=-\frac{y(x+y)^3}{8\sqrt2}(1+B),\qquad A,B\in O_2.
\end{equation}
For the return map use the monotone phase $\theta=-\varphi$ as time.
Since $d\theta/dt=1+O_4>0$, the equations retain this form with
$O_2$ remainders and are $2\pi$-periodic in $\theta$.
Along an orbit, with $ds=(x+y)^3d\theta/(8\sqrt2)$, they become
\begin{equation}\label{RPC3BP-2}
 x_s=x(1+A),\qquad y_s=-y(1+B),\qquad
 \theta_s=\frac{8\sqrt2}{(x+y)^3}.
\end{equation}
\begin{Remark}\label{rem:rpc-normal-input}
The normal form follows from exact invariant surfaces and weighted
derivative estimates. Only $O_\eta$ remainders for some $\eta>0$
are needed below; no $O_{3,7}$ bound is assumed.
All local uniformity here is on compact parameter sets. In particular,
this does not prove the global large-$C$ assertions of the introduction.
For a different normal-form construction see \cite[Proposition 2.3]{GMP2025}.
\end{Remark}

\section{Dynamics at infinity}\label{dynamics-at-infinity}

In order to study $C^1$ and $C^2$ dynamics at infinity of the Sitnikov
and the RPC3BP we start with a class of differential equations on
$(x,y)\in \mathbb{R} ^2$ which contains both.

First we add one differential
equation on $\lambda $, coupled with the first two, so that it describes
evolution of slopes $\lambda (\cdot)$ of certain class of curves on $\mathbb{R} ^2$ carried
under evolution of the first two equations. This gives us information
about $C^1$-dynamics.

Finally, we add one more differential equation on $\kappa$, coupled
with the first three, so that it describes evolution of quantity
$\kappa(\cdot)$ related to curvature of some curves on $\mathbb{R} ^2$ carried
under evolution of the first two equations. This gives us information
about $C^2$-dynamics sufficient for the proof.

\begin{figure}[tbp]
\begin{center}
\includegraphics[width= 1.9in,angle=0]{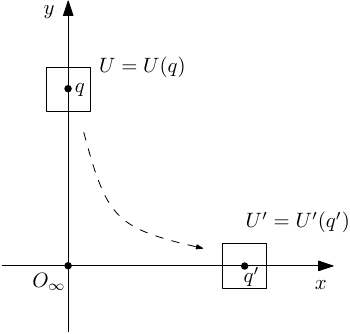}
\caption{Dynamics near infinity.}
\end{center}
\end{figure}

\subsection{Evolution of slopes and second derivatives near degenerate saddles}
We consider the class
\begin{equation}\label{deg-saddle-evolution}
 \frac{dx}{dt}=x(x+y)^3(1+A),\qquad
 \frac{dy}{dt}=-y(x+y)^3(1+B),\qquad A,B\in O_\eta,
\end{equation}
with weighted derivative bounds for some $\eta>0$. This includes
\eqref{sitnik-1} and \eqref{RPC3BP-1}, after accounting for their
constant positive leading factors. For the RPC3BP return map, $t$
in this local calculation denotes the monotone angular time
$\theta=-\varphi$. Every image curve is evaluated at a common value
of this time variable. Rescaling time along individual orbits is
used only after deriving the variational equations.

In Sections~\ref{dynamics-at-infinity}--\ref{sec:revised-general}, $\lambda=dy/dx$ and $\kappa=d^2y/dx^2$ are jet
coordinates; the RPC3BP mass ratio remains $\mu$. The second derivative $\kappa$ is not the Euclidean signed
curvature. We write $a_\eta=a+O_\eta$ for a numerical coefficient
$a$; distinct occurrences denote independent remainders.

\subsection{Derivation of equations for slopes and second derivatives
for evolving curves}\label{derivation-of-C1-C2-evolutions}
Write the physical-time equations as $\dot x=F(x,y,t,e)$,
$\dot y=F^{(2)}(x,y,t,e)$. All spatial derivatives below are at fixed $t,e$.
For a graph transported at a common physical time, direct differentiation
gives the exact identities
\begin{equation}\label{eq:exact-jets}
\begin{aligned}
 \dot\lambda={}&F^{(2)}_x+(F^{(2)}_y-F_x)\lambda-F_y\lambda^2,\\
 \dot\kappa={}&F^{(2)}_{xx}+(2F^{(2)}_{xy}-F_{xx})\lambda
 +(F^{(2)}_{yy}-2F_{xy})\lambda^2-F_{yy}\lambda^3\\
 &+(F^{(2)}_y-2F_x-3F_y\lambda)\kappa.
\end{aligned}
\end{equation}
Indeed, if $u=x_\xi$ and $v=y_\xi$ for a parametrization of the
initial curve, then $\dot u=F_xu+F_yv$, $\dot v=F^{(2)}_xu+F^{(2)}_yv$ and
$\lambda=v/u$. Differentiating this identity once more along the
curve gives the equation for $\kappa$. These identities apply to
nonautonomous vector fields: the initial phase is the same for every
point on the curve, so no $t_\xi$ term occurs.

For $F=x(x+y)^3(1+A)$ and $F^{(2)}=-y(x+y)^3(1+B)$ with
$A,B\in O_\eta$, expansion of these identities yields the following
system, with all remainders of order $\eta$:
\begin{eqnarray}   \label{2jet-deg-saddle-evolution}
\left\{
  \begin{array}{ll}
   \dot{x}=\ \ x(x+y)^3\ 1_{\eta}\\
    \dot{y}=-y(x+y)^3\ 1_{\eta} \\
    \dot{\lambda}=-(x+y)^2[3_{\eta}y +5_{\eta}(x+y) \lambda +
    3_{\eta} x \lambda^2] \\
\dot{\kappa}
=-3(x+y)[2_{\eta}y+(6_{\eta}x+8_{\eta}y)\lambda+(8_{\eta}x+6_{\eta}y)\lambda^2+2_{\eta}x\lambda^3+
\\
\qquad \qquad \qquad \qquad \qquad \qquad \qquad \qquad \qquad +\left( 3_{\eta}x+2_{\eta}y+3_{\eta}x\lambda \right )(x+y)\kappa ]
\end{array}
\right.
\end{eqnarray}
The first two equations describe evolution of position, the third
--- evolution of slope, and the fourth --- of the second derivative.
It is natural to call $\lambda (x,y)$ {\it slope of evolving curve}.
Denote by $\kappa(x,y)$ {\it second derivative of evolving curve}.
We first study the simplified system obtained by setting $A=B=0$.
Section~\ref{sec:revised-general} then proves the estimates with the
weighted remainders retained.

The local passage takes an incoming graph crossing the stable axis
$x=0$ to an outgoing graph near the unstable axis $y=0$. The strips
$U_{\rm in}$ and $U_{\rm out}$ are specified below; their longitudinal
coordinates stay in fixed positive compact intervals. No global
homoclinic intersection is needed for these local estimates.

\subsection{Statement of main results on dynamics at infinity}

Introduce the notation: $X=(x,y)\in \mathbb{R} ^2,\ \overline X=(x,y,\lambda )\in\mathbb{R} ^3,\
\overline {\overline X}=(x,y,\lambda ,\kappa)\in\mathbb{R} ^4$. Let
$\overline {\overline \phi}_t$ be the time $t$ map of the equation
(\ref{2jet-deg-saddle-evolution}). Let
$\overline \phi_t$ be the time $t$ map of first three equations
(\ref{2jet-deg-saddle-evolution}), which is well
defined because the first three equations describe dynamics
in the space of $1$-jets, and are independent of $\kappa$.
Finally, let $\phi_t$ denote the time $t$ map of first two equations
of (\ref{2jet-deg-saddle-evolution}).
Denote
\[
\adjustbox{max width=\linewidth}{$\displaystyle
\overline {\overline X}_t=(x_t,y_t,\lambda _t,\kappa_t)=
\overline {\overline \phi}_t\overline{\overline X},
$}
\]
\[
\adjustbox{max width=\linewidth}{$\displaystyle
\overline X_t=(x_t,y_t,\lambda _t)= \overline \phi_t\overline X.
$}
\]
\[
\adjustbox{max width=\linewidth}{$\displaystyle
X_t=(x_t,y_t)= \phi_t X.
$}
\]
It turns out that for different purposes we need statements
about evolution of $X_t$, $\overline X_t$, and
$\overline {\overline X}_t$. Such an evolution is not always
defined for all time. Let $\gamma =\{(x,y(x)),\ x\in [0,\delta ]\}$
be a smooth curve. Denote
$\overline \gamma =\{(x,y(x),\lambda (x)),\ x\in [0,\delta ]\}$, where
$\lambda (x)=y'(x)$.
Denote $\overline {\overline \gamma }=
\{(x,y(x),\lambda (x),\kappa(x)),\ x\in [0,\delta ]\}$, where
$\kappa(x)=y''(x)$.

The Poincar\'e maps in the applications preserve a smooth area form away from the origin. The local estimates below do not require this additional property.

\noindent\textit{Local passage convention.}
In the next two statements $U_{\rm in}$ is a strip along the stable axis
with $y$ in a fixed compact subinterval of $(0,\delta]$, and
$U_{\rm out}$ is a strip along the unstable axis with $x$ in such an
interval. In the notation of Theorem~\ref{t.5}, $U'=U_{\rm in}$ and
$U=U_{\rm out}$. We consider only image components whose entire local
passage remains in the positive square. The incoming curve crosses the
stable axis and has uniformly bounded slope and, when needed, curvature.
The remainders satisfy the hypotheses of Section~\ref{sec:revised-general}.
The time period is fixed; $f$ is the corresponding fixed-period map.
The smooth area-form assumption above is not used in the proof of stretching.

\begin{Theorem} \label{t.5}  With the notation above, for any $\lambda ^*, \kappa^*>0$
and sufficiently small $\delta >0$ there is
$C=C(\delta ,\lambda ^*,\kappa^*)>1$ such that for any $\gamma $ with
$\max |\lambda (x)|<\lambda ^*$ we have the following estimates.
Then for large $N$
intersection of the image $f^N \gamma  \cap U$ is a graph
of a function $y_N(x)$.  Denote $\lambda _N(x)=y'_N(x)$. Then
the following properties hold
\begin{eqnarray}
\label{e.C1-lb-lemma}
0< -\lambda _N(x)\le C y_N(x) & \\
\label{stretching} \dfrac{1}{Cx^{2.5}}\le |df^N(x)v_x|\le \dfrac{C}{x^{2.5}},
&\text{ where }v_x \text{ is a unit tangent to }  \gamma .
\end{eqnarray}
Moreover, for any two unit vectors $v',v''\in T_{(x,y)}U'$
with slopes $|\lambda '|,|\lambda ''|<\lambda ^*$ the following holds. Let $N=N(x,y)$ be the number of iterates of the Poincar\'e map $f$
to get to $U$, i.e. $f^N(x,y)\in U$. Then
\begin{eqnarray}  \label{projective-hyperbolicity}
\dfrac{x^{5}}{C}\angle(v',v'')\le
\angle(df^N(x,y)(v'),df^N(x,y)(v''))
\le C x^{5}\angle(v',v'').
\end{eqnarray}

In the stretching and angle estimates, $x$ is the incoming coordinate;
in $y_N(x)$ and its derivatives, $x$ is the outgoing coordinate.
Here angles are the smaller angles between the corresponding unoriented lines.

In the case $\gamma $ is also $C^2$ smooth, denote $\kappa(x)=y''(x)$ and
$f^N(\gamma )=\gamma _N$ \ and \ $\overline {\overline \gamma }_N=\{(x,y_N(x),
\lambda _N(x),\kappa_N(x))\}.$ If $\max |\kappa(x)|<\kappa^*$ then for all large
enough $N$ we have
\begin{eqnarray}
\label{e.C2-lb-lemma}
0< \kappa_N(x) < C y_N(x).
\end{eqnarray}

\end{Theorem}

Denote $\Delta=\{(x,y):\ x=y, 0<x<\delta \}$ for small $\delta $ as above.

\begin{Theorem} \label{t.c2lyambdalemma} ($C^2$ behavior of the images
of the diagonal)
Assume also the parameter-uniform hypotheses in Section~\ref{sec:parameter-revised}.
For large enough $N$
the local passage component of $f_e^N(\Delta)\cap U$
is a graph of a function $y_N(x,e)$ such that
\begin{eqnarray} \label{e.shape}
\begin{aligned}
0<&\ \ \ \partial^2_x y_N(x,e)=\kappa_N(x) < C y_N(x,e), \\
0< & -\partial_x y_N(x,e)=-\lambda _N(x) < C y_N(x,e).
\end{aligned}
\end{eqnarray}

We also have
\begin{eqnarray} \label{e.derivine}
\left| \dfrac{d}{de} y_N(x,e) \right| \le C y_N(x,e).
\end{eqnarray}
\end{Theorem}

\begin{Remark}
The local spatial estimates are used in the periodic-saddle
construction in Section~\ref{ss.saddlesnearhomoclinicpoint} and in
the transverse-connection argument in Section~\ref{ss:sparse-excursions}.
The diagonal estimates and the parameter bound give the first-contact
construction in Section~\ref{ss.existenceoftangencies}. Parameter
control for incoming curves is established separately in
Lemma~\ref{lem:incoming-parameter}.
\end{Remark}

\begin{Remark} \label{ss.comparision}

It is interesting to compare the dynamics near the degenerate saddle $O_{\infty}$ that
we study in this section, and the well known dynamical properties of a linear hyperbolic
saddle. This comparison is not directly needed for the proof, but highlights
the difficulties of the local analysis. On the restricted positive
passage sector, both models have an incoming stable edge and an
outgoing unstable edge. This sector picture must be distinguished
from the dynamics on an open neighborhood in the smooth McGehee
extension. In that extension the line $q=0$ consists of fixed points,
so no local conjugacy with an isolated hyperbolic saddle can exist
on a full neighborhood. Results on degenerate Hartman theorems such
as \cite{Mor} require their own hypotheses and an appropriate
choice of domain. Our proof uses the explicit passage estimates,
not a conjugacy at the boundary. In the smooth category the passage
dynamics have the following differences.

1. Expansion rates of the maps along the saddle are different.

2. The classical hyperbolic inclination lemma does not apply at the
degenerate point. Theorems~\ref{t.5} and~\ref{t.c2lyambdalemma}
provide the restricted inclination estimates used here, on outgoing
arcs away from the vertex.

3. Transition times are certainly very much different.

4. Corresponding dynamics in 1- and 2-jets is essentially non-linear and is quite non-trivial in the case of the degenerate saddle. In the case of a linear hyperbolic saddle the dynamics in 2-jets is linear and hyperbolic, and is easy to describe.

5. For a non-linear hyperbolic saddle normal forms (smooth and in some cases analytic) or even linearizing coordinates that essentially simplify the picture are available. In the case of the degenerate saddle even a model case (see Section \ref{s.simplified1}) is highly non-trivial if one is interested in $C^2$ or even $C^1$ dynamical properties.
\end{Remark}

\section{Evolution of curves in the simplified system}\label{s.simplified1}

In order to study the behavior of solutions of the system
(\ref{2jet-deg-saddle-evolution}) we will first neglect some
higher order terms and obtain the required properties for
the simplified system. After that we show that the neglected
terms do not change the obtained results.

After we rescale the time by $ds=(x+y)^3\,dt$ and neglect some
higher order terms  in the system (\ref{2jet-deg-saddle-evolution}),
we get the following {\it simplified model}:
\begin{eqnarray}   \label{rescaled-simplified-2jet-deg-saddle-evolution}
\left\{
  \begin{array}{ll}
   \dfrac{dx}{ds}   =   x\\
   \dfrac{dy}{ds}   = - y\\
   \dfrac{d\lambda }{ds} = \dfrac{-3y}{x+y}-5\lambda -\dfrac{3x}{x+y}\lambda^2 \\
   \dfrac{d\kappa}{ds} = \dfrac{-3}{(x+y)^2}\left[
   2y+(6x+8y)\lambda+(8x+6y)\lambda^2 +2x\lambda^3\right.
\\ \qquad\qquad\left.{}+(x+y)(3x+2y+3x\lambda )\kappa\right]
  \end{array}
\right.
\end{eqnarray}
The initial bounded transient is handled separately in
Lemma~\ref{l.explicitupperbound-general}; comparison arguments requiring
$d\ge0$ start after entry into the slope cone. The normalization
$x(s)=e^{s-2T_0}$, $y(s)=e^{-s}$ in the model computations is a convenient
choice of fixed incoming and outgoing scales; the uniform estimates for
actual strips are proved in Section~\ref{sec:revised-general}.

In this section we prove analogs of Theorems \ref{t.5}
and \ref{t.c2lyambdalemma} for the dynamics defined by
the simplified system (\ref{rescaled-simplified-2jet-deg-saddle-evolution}).

\subsection{Dynamics of 1-jets for the simplified system}

First of all let us understand the asymptotic behavior
of $\lambda$ under the dynamics defined by
(\ref{rescaled-simplified-2jet-deg-saddle-evolution}). We need to consider only the first three equations of the system:
\begin{eqnarray}   \label{e.simplelyambda}
\left\{
  \begin{array}{ll}
   \dfrac{dx}{ds}   =   x\\
   \dfrac{dy}{ds}   = - y\\
   \dfrac{d\lambda }{ds} = \dfrac{-3y}{x+y}-5\lambda -\dfrac{3x}{x+y}\lambda^2
  \end{array}
\right.
\end{eqnarray}

The right hand side of the equation
\[
\adjustbox{max width=\linewidth}{$\displaystyle
\dfrac{d\lambda }{ds} = \dfrac{-3y}{x+y}-5\lambda -\dfrac{3x}{x+y}\lambda^2
$}
\]
is a quadratic polynomial in $\lambda$, and can be rewritten as
\begin{eqnarray} \label{e.difflyambda}
\dfrac{d\lambda }{ds}=-\frac{3x}{x+y}\, (\lambda -\lambda _0(x,y))\,(\lambda -\lambda _1(x,y)).
\end{eqnarray}

\begin{Definition} \label{absorbing-definition} A family of intervals $[\lambda _-(x,y),\lambda _+(x,y)]$
is called {\rm a family of absorbing intervals} or simply
{\rm absorbing intervals} if any solution $(x(s), y(s), \lambda (s))$ of
the system (\ref{e.simplelyambda}) satisfies the following
conditions:
\begin{eqnarray}  \label{absorbing-condition}
\begin{aligned}
\textup{  If } \ & \lambda (s)=\lambda _-(x(s), y(s)), \ \ \ \ \ \text{ then } \ \ \ \ \
\dfrac{d\lambda }{ds}(s)>\dfrac{d}{ds}\lambda _-(x(s), y(s)),\ \ \text{ and }
\\
\textup{if }\  & \lambda (s)=\lambda _+(x(s), y(s)), \ \ \ \ \ \text{ then } \ \
\ \ \ \dfrac{d\lambda }{ds}(s)<\dfrac{d}{ds}\lambda _+(x(s), y(s)).
\end{aligned}
\end{eqnarray}
\end{Definition}
Notice that if for some $s_0$ the value $\lambda (s_0)$ gets into
an absorbing interval then it will stay there for all $s>s_0$.

\begin{figure}[htbp]
\begin{center}
\includegraphics[width= 2.8in,angle=0]{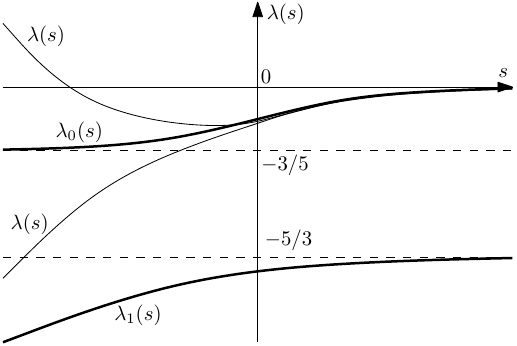}
\caption{Dynamics of slopes $\lambda (s)$ and $\lambda _{0,1}(x(s), y(s))$.}
\label{fig:slopes}\label{lb-graph}
\end{center}
\end{figure}

In order to provide the explicit expressions for $\lambda _0(x,y)$ and $\lambda _1(x,y)$ it is convenient to introduce an intermediate variable
$\tau:=\frac{y}{x}$. Also, denote
\begin{eqnarray} \label{e.p}
P(\tau)=(1+\tau)^2-1.44\tau.
\end{eqnarray}

 Then we have
\begin{eqnarray}  \label{small-tau-new}
\dfrac{d\lambda }{ds}= - \dfrac{3\lambda ^2}{1+\tau} -5 \lambda  - \dfrac{3\tau}{1+\tau}
= - \dfrac{3 }{1+\tau}\left(\lambda -\lambda _1(\tau) \right)\left(\lambda -\lambda _0(\tau) \right),
\end{eqnarray}
where
\begin{eqnarray}  \label{lb-roots}
\lambda _{0,1}(\tau)=
-\dfrac 56 (1+\tau)\pm \dfrac 56 \sqrt{(1+\tau)^2- 1.44 \tau}=
-\dfrac 56 (1+\tau)\pm \dfrac 56 \sqrt{P(\tau)}.
\end{eqnarray}

In the next statement we establish the existence of an absorbing
interval whose size tends to zero as $y\to 0$.

\begin{Lemma} \label{l.absorb}
The family of intervals
\[
\adjustbox{max width=\linewidth}{$\displaystyle
[\lambda _-(x,y), \lambda _+(x,y)]:=\left[2\lambda _0(\tau),\frac{9}{10}\lambda _0(\tau)\right]
$}
\]
is an absorbing family.
\end{Lemma}

\begin{Lemma}  \label{lem:absorb} ($\lambda $-absorbing interval) For any
$\lambda ^*>0$ there is $s_0=s_0(\lambda ^*)>0$  such that for any
$C^1$-smooth curve $\overline \gamma =
\{(x,y(x),\lambda (x)=y'(x)),\ x\in [0,\delta ]\}, \gamma \subset U_{\rm in}$
with small $\delta >0$ and $\max |\lambda (x)|<\lambda ^*$
  assume the initial ratio $y/x$ is sufficiently large depending on the slope bound; we have the following. Let $\{(x(s),y(s),\lambda (s))\}$ be the image of some point $\{(x,y,\lambda )\}\in \overline \gamma $ under the flow (\ref{e.simplelyambda}). Then for any $s>s_0$
\[
\adjustbox{max width=\linewidth}{$\displaystyle
\lambda (s)\in\ \left[2\lambda _0(\tau(s)),\frac{9}{10}\lambda _0(\tau(s))\right]=:
[\lambda _-(x(s), y(s)),\lambda _+(x(s), y(s))]
$}
\]
\end{Lemma}

\begin{Lemma} \label{lem:absorbdiagonal} ($\lambda $-absorbing interval for the diagonal)
There are $s_0>0$ and small $\delta >0$ such that for any point $(x,y,\lambda ), x=y, \lambda =1,$ for any $s>s_0$ the image $\{(x(s),y(s),\lambda (s))\}$  under the flow (\ref{e.simplelyambda}) is such that
\[
\adjustbox{max width=\linewidth}{$\displaystyle
\lambda (s)\in\ \left[2\lambda _0(\tau(s)),\frac{9}{10}\lambda _0(\tau(s))\right]=:
[\lambda _-(x(s), y(s)),\lambda _+(x(s), y(s))]
$}
\]
\end{Lemma}
\begin{Remark}
Lemma \ref{lem:absorb} is needed to prove   Theorem \ref{t.5}, and Lemma \ref{lem:absorbdiagonal} is needed to prove (\ref{e.shape})   in Theorem \ref{t.c2lyambdalemma}.
\end{Remark}

Before we start the proof of Lemma \ref{lem:absorb} and Lemma \ref{lem:absorbdiagonal} we need to
obtain some details on behavior of the function $\lambda _0(x,y)$.

\begin{Proposition} \label{p.tau-and-lambda-not}
The following properties hold:
\

\begin{enumerate}[(i)]
\item
For any $\tau>0$ we have $0.8\le \frac{\sqrt{P(\tau)}}{1+\tau}<1$;
\item
$\lim_{\tau\to +\infty}\lambda_0(\tau)=-0.6$;
\item
$\lim_{\tau\to 0}\frac{\lambda_0(\tau)}{\tau}=-0.6$;
\item
$\frac{d}{ds}\lambda_0(s)=-\frac{\sqrt{P(\tau)}-\tau+1}{\sqrt{P(\tau)}}\lambda_0(s)$;
\item
$-\frac{1.28\lambda_0(s)}{\sqrt{P(\tau)}}<
\frac{d}{ds}\lambda_0(s) < -\frac{2\lambda_0(s)}{\sqrt{P(\tau)}}$;
\item
$\lim_{\tau\to 0}\frac{\frac{d}{ds}\lambda_0}{\lambda_0}=-2$.
\end{enumerate}
\end{Proposition}
\begin{proof}[Proof of Proposition \ref{p.tau-and-lambda-not}]
Here is the proof of the part $(i)$. We have
\[
\adjustbox{max width=\linewidth}{$\displaystyle
\frac{\sqrt{P(\tau)}}{1+\tau}=\frac{\sqrt{(1+\tau)^2-1.44\tau}}{1+\tau}=
\sqrt{1-\frac{1.44 \tau}{(1+\tau)^2}}<1.
$}
\]
On the other hand, $\max_{\tau>0}\frac{\tau}{(1+\tau)^2}=
\frac{1}{4}$, hence
\[
\adjustbox{max width=\linewidth}{$\displaystyle
\frac{\sqrt{P(\tau)}}{1+\tau}=
\sqrt{1-\frac{1.44 \tau}{(1+\tau)^2}}\ge \sqrt{1-\frac{1.44 }{4}}=0.8
$}
\]
Now let us show that $(ii)$ and $(iii)$ hold. One has
\[
\adjustbox{max width=\linewidth}{$\displaystyle
\lambda _{0}(\tau)=
-\dfrac 56 (1+\tau)+ \dfrac 56 \sqrt{(1+\tau)^2- 1.44 \tau}=
- \dfrac{ 1.2\tau}{(1+\tau)+ \sqrt{(1+\tau)^2- 1.44 \tau}}\to -0.6 \ \ \text{as}\ \ \tau\to +\infty,
$}
\]
and
\[
\adjustbox{max width=\linewidth}{$\displaystyle
\frac{\lambda _{0}(\tau)}{\tau}=
 - \dfrac{ 1.2}{(1+\tau)+ \sqrt{(1+\tau)^2- 1.44\tau }}\to -0.6 \ \ \text{as}\ \ \tau\to 0.
$}
\]
Here is how the formula $(iv)$ can be justified. We have
\[
\adjustbox{max width=\linewidth}{$\displaystyle
\frac{d}{ds}\lambda _0(s)=\frac{d}{d\tau}\lambda _0(\tau)\frac{d\tau}{ds}=
-2\tau \frac{d}{d\tau}\lambda _0(\tau)
$}
\]
At the same time we have
\[
\adjustbox{max width=\linewidth}{$\displaystyle
\frac{\frac{d}{d\tau}\lambda _0(\tau)}{\lambda _0(\tau)}=
\frac{d}{d\tau}\left(\ln |\lambda _0(\tau)|\right)=
$}
\]
\[
\adjustbox{max width=\linewidth}{$\displaystyle
=\frac{d}{d\tau}\left(\ln \dfrac{ 1.2\tau}{(1+\tau)+ \sqrt{(1+\tau)^2- 1.44\tau }}\right)=
\frac{d}{d\tau}\left(\ln(1.2\tau)-\ln (1+\tau+\sqrt{P(\tau)}) \right)=
$}
\]
\[
\adjustbox{max width=\linewidth}{$\displaystyle
=\frac{1}{\tau}-\frac{\frac{d}{d\tau}(1+\tau+\sqrt{P(\tau)})}{1+\tau+\sqrt{P(\tau)}}=
\frac{1}{\tau}-\frac{1+\frac{\tau+0.28}{\sqrt{P(\tau)}}}{1+\tau+\sqrt{P(\tau)}}=
\frac{\sqrt{P(\tau)}-\tau+1}{2\tau\sqrt{P(\tau)}}.
$}
\]
Hence
\[
\adjustbox{max width=\linewidth}{$\displaystyle
\frac{d}{ds}\lambda _0(s)=-2\tau \frac{\sqrt{P(\tau)}-\tau+1}{2\tau\sqrt{P(\tau)}} \lambda _0(\tau)=
-\frac{\sqrt{P(\tau)}-\tau+1}{\sqrt{P(\tau)}}\lambda_0(s).
$}
\]
Since $(1+\tau)^2-1.44\tau<(1+\tau)^2$ for any $\tau>0$, we have
\[
\adjustbox{max width=\linewidth}{$\displaystyle
\sqrt{(1+\tau)^2-1.44\tau}-\tau<1, \ \ \text{or}\ \ \sqrt{P(\tau)}-\tau+1<2.
$}
\]
On the other hand,
\[
\adjustbox{max width=\linewidth}{$\displaystyle
\sqrt{P(\tau)}-\tau+1=\frac{2.56\tau}{\sqrt{P(\tau)}-1+\tau}=
\frac{2.56}{\sqrt{\left(1+\frac{1}{\tau}\right)^2-\frac{1.44}{\tau}}-\frac{1}{\tau}+1}>\frac{2.56}{2}=1.28
$}
\]
Therefore, $1.28<\sqrt{P(\tau)}-\tau+1<2$, and so $(iv)$ implies $(v)$.
Finally,
\[
\adjustbox{max width=\linewidth}{$\displaystyle
\lim_{\tau\to 0}\frac{\frac{d}{ds}\lambda_0}{\lambda_0}=
\lim_{\tau\to 0}\left(-\frac{\sqrt{P(\tau)}-\tau+1}{\sqrt{P(\tau)}}\right)=
-2,\ \ \ \text{\rm so $(vi)$ holds.}
$}
\]
\end{proof}

Now we are prepared to start the proof of Lemma \ref{l.absorb}.

\begin{proof}
Write $L=-\lambda_0>0$, $z=(1+\tau)^{-1}$ and
$D=\sqrt P/(1+\tau)=\sqrt{1-(36/25)z(1-z)}$.
Then $(\lambda_0)_s=HL$, where $H=1+(2z-1)/D>0$.
At the upper boundary $\lambda=-9L/10$ the Riccati field is negative
while the boundary derivative is positive. At the lower boundary,
the field minus the boundary derivative equals
\[
 L\left(5-\frac{9L}{1+\tau}-2H\right).
\]
Multiplication of the parenthesis by $D$ gives
$19/2-(74/5)z+(54/5)z^2-(9/2)D$.
Since $D\le1-(18/25)z(1-z)$, this is bounded below by
\[
 5-\frac{289}{25}z+\frac{189}{25}z^2>0.
\]
The last quadratic has positive leading coefficient and negative
discriminant. Both boundary inequalities are therefore strict.
This proves forward invariance of the stated family.
\end{proof}

\begin{proof}[Proof of Lemma \ref{lem:absorb}]
Let us now  show that the condition $\max |\lambda (x)|<\lambda ^*$ implies that a vectors tangent to the curve $\gamma $ will enter this absorbing interval $[\lambda _{-}(x(s), y(s)), \lambda _{+}(x(s), y(s))]=\left[2\lambda _0(\tau(s)),\frac{9}{10}\lambda _0(\tau(s))\right]$ within a finite time. For large $\tau$ (we can guarantee that by choosing small enough $\delta$)  we have $\lambda _0(\tau(s))\simeq -0.6$, see Proposition \ref{p.tau-and-lambda-not} ({\it ii}). Also, in this case we have
\[
\adjustbox{max width=\linewidth}{$\displaystyle
\dfrac{d\lambda }{ds}= - \dfrac{3\lambda ^2}{1+\tau} -5 \lambda  - \dfrac{3\tau}{1+\tau}\simeq -5\lambda -3=-5(\lambda +0.6).
$}
\]
 Therefore, since $|\lambda (0)|$ is bounded by $\lambda ^*$, $\lambda (s)$ within a finite time will enter a small neighborhood of $\lambda _0(\tau(s))\simeq -0.6$, and, hence, the absorbing interval $\left[2\lambda _0(\tau(s)),\frac{9}{10}\lambda _0(\tau(s))\right]$ that contains $-0.6$ as an internal point for large values of $\tau$.

This completes the proof of Lemma \ref{lem:absorb}.
\end{proof}

\begin{proof}[Proof of Lemma \ref{lem:absorbdiagonal}]
For the diagonal initial condition, $\tau(s)=e^{-2s}$ and
$\lambda(0)=1$. Let $s_*$ be the first entrance time into the
absorbing interval, if such an entrance occurs. Before that time,
\[
 \lambda(s)>\frac9{10}\lambda_0(\tau(s))>\lambda_0(\tau(s)).
\]
For $0<\tau\le1$ one has
\[
 \lambda_0-\lambda_1=\frac53\sqrt{1+0.56\tau+\tau^2}
       \ge\frac53,
 \qquad \frac{d\lambda_0}{ds}>0.
\]
Consequently, with $d=\lambda-\lambda_0>0$, the slope equation gives
\[
 d_s=-\frac3{1+\tau}(\lambda-\lambda_1)d
                      -\frac{d\lambda_0}{ds}
       \le-\frac52d.
\]
Since $\lambda_0(1)=-1/3$, it follows that
\[
 0<d(s)\le\frac43e^{-5s/2}\qquad(0\le s<s_*).
\]
On the other hand, failure to enter the interval implies
\[
 d(s)>\frac1{10}|\lambda_0(\tau(s))|
   =\frac{0.12\tau(s)}{1+\tau(s)+\sqrt{P(\tau(s))}}
   \ge\frac1{30}e^{-2s}.
\]
These two inequalities cannot both hold when $s>2\log40$.
Thus entrance occurs within a bounded rescaled time independent of
the initial diagonal point. The absorbing property in
Lemma~\ref{l.absorb} gives membership at all subsequent times.
The comparison for $d$ is used only before this first entrance.
\end{proof}

The next statement can be interpreted as an analogue of a ``cone condition'' in a neighborhood of the degenerate saddle.

\begin{Lemma}  \label{lem:stretch} (stretching lemma)
For any $\lambda ^*>0$ there exist $0<c=c(\lambda ^*)<1$
and $\delta =\delta (\lambda ^*)>0$ with the following property.

a) Let $\overline X'=(x,y,\lambda ')$ and $\overline X''=(x,y,\lambda '')$, $|\lambda '|, |\lambda ''|\le \lambda ^*$,
be two initial conditions with $X=(x,y)\in U'$ and
$0<x<\delta $. Then at any moment of  time $s$
such that $X_s\in U$ we have
\[
\adjustbox{max width=\linewidth}{$\displaystyle
c x^{5}|\lambda '(0)-\lambda ''(0)|\le |\lambda '(s)-\lambda ''(s)|\le \dfrac{x^{5}}{c}|\lambda '(0)-\lambda ''(0)|.
$}
\]

b) Let $(x,y)\in U'$ and $v =(v_x,v_y)$ satisfy
$0<x<\delta $ and $v_x\ne0$, $|v_y/v_x|\le\lambda ^*$. Let $N(x,y)$
be the number of iterates of the Poincar\'e map $f$
to get to $U$, i.e. $f^N(x,y)\in U$. Then
\[
\adjustbox{max width=\linewidth}{$\displaystyle
c x^{-2.5}\|v\| \le |df^N(x)\, v| \le \dfrac{x^{-2.5}}{c}\|v\|.
$}
\]
\end{Lemma}

\begin{proof}
The simplified system is the case $A=B=0$ of
Lemma~\ref{lem:stretchgen}. Its proof gives
$|\lambda'(S)-\lambda''(S)|\asymp e^{-5S}|\lambda'(0)-\lambda''(0)|$
and $\|D\phi_t v\|\asymp e^{5S/2}\|v\|$ between the fixed strips.
Since $e^S\asymp x(0)^{-1}$, both conclusions follow. Here $S$ is
rescaled time, whereas $t=NP$ for an iterate of the Poincar\'e map.
These quantities are not identified. The proof uses the physical-time
variational equation along the base orbit, so it is also valid when
the physical terminal time is prescribed.
\end{proof}

\subsection{Dynamics of 2-jets for the simplified system}

Now we present behavior of convexity of curves. It is useful to
keep in mind that this dynamics depends on
slope so we incorporate slope $\lambda $ into the model. Recall
that due to above discussion of evolution of slope it is natural
to assume that $\lambda (s)$ is bounded and becomes small when $s$ is close to $2T_0$.

\begin{Lemma} \label{l.muest}
For any $\lambda ^*>0$ and $\kappa^*>0$ there exists a constant
$C>1$, depending also on the fixed incoming and outgoing strips, such that
for any 2-jet ${\overline{\overline{X}}}=(x, y, \lambda , \kappa)$
satisfying $|\lambda |\le \lambda ^*$, $|\kappa|\le \kappa^*$, $(x, y)\in U_{\rm in}$,
with initial $x$ sufficiently small, and for any sufficiently long local
passage ending at a moment $s$ such that $X_s\in U_{\rm out}$ we have
${\overline{\overline{X}}}_s=(x_s, y_s, \lambda _s, \kappa_s)$, where
\[
\adjustbox{max width=\linewidth}{$\displaystyle
0<\kappa_s<Cy_s.
$}
\]
\end{Lemma}
 Notice that Lemma \ref{l.muest}, evaluated when the physical time equals $NP$, proves (\ref{e.C2-lb-lemma}).

For general fixed strips, Lemma~\ref{l.muest} follows from
Lemma~\ref{l.finalmu-general} with $A=B=0$.

 \begin{Lemma} \label{l.muestdiagonal}
 Suppose that the image $(x_s, y_s, \lambda _s, \kappa_s)$ of the 2-jet
 $(x_0, y_0, \lambda _0, \kappa_0)$, $x_0=y_0$, $\lambda _0=1, \kappa_0=0$, under
 the flow  (\ref{rescaled-simplified-2jet-deg-saddle-evolution})
 is such that $(x_s, y_s)\in U_{\rm out}$. If $x_0=y_0$ is small enough then
 \[
\adjustbox{max width=\linewidth}{$\displaystyle
0<\kappa_s<Cy_s.
$}
\]
 \end{Lemma}

The normalized model estimate is proved in Lemma~\ref{l.finalmu};
Lemma~\ref{l.finalmu-general}, with $A=B=0$, gives both assertions
on the fixed strips specified here.

Likewise, evaluate Lemma \ref{l.muestdiagonal} at physical time $NP$, not at integer rescaled time, to obtain (\ref{e.shape}).

 \

Substituting $y=\tau x$ into the last equation of
(\ref{rescaled-simplified-2jet-deg-saddle-evolution}) directly gives the following:
\begin{Lemma} \label{lemma:muformula}
The function $\kappa(s)$ along a solution satisfies the equation
\[
\adjustbox{max width=\linewidth}{$\displaystyle
\frac{d\kappa}{ds}= -d(\tau(s), \lambda )\kappa-\frac{B(\tau(s), \lambda )}{x+y},
$}
\]
where
\[
\adjustbox{max width=\linewidth}{$\displaystyle
B(\tau(s),\lambda )=   \dfrac{6 \tau}{1+\tau} +
\dfrac{18+24 \tau+(24+18 \tau)\lambda }{1+\tau}  \,\lambda  +
\dfrac{6\lambda ^3 }{1+\tau}
$}
\]
and
\[
\adjustbox{max width=\linewidth}{$\displaystyle
d(\tau(s),\lambda )=
\dfrac{9+ 6 \tau }{1+\tau}+ \dfrac{9\lambda  }{1+\tau}.
$}
\]
\end{Lemma}

Denote also $b(x(s),y(s),\lambda )=\frac{B(\tau(s), \lambda )}{x+y}$.

\paragraph{Normalization for the model calculation.}
For the calculations through Lemma~\ref{l.finalmu}, put the initial
incoming height $b=y(0)$ equal to one by using
$\widehat x=x/b$, $\widehat y=y/b$. Rescaled time $s$ and slope
are unchanged, whereas the second derivative becomes
$\widehat\kappa=b\kappa$. Set $2T_0=\log(b/x(0))$ and end the
symmetric comparison passage at $x=b$. In these normalized
variables, suppressing hats, one has exactly
\[
 x(s)=e^{s-2T_0},\qquad y(s)=e^{-s},\qquad 0\le s\le2T_0.
\]
The constants $\kappa^*$, $C^*$ and $M$ in this model calculation
refer to these normalized variables. For arbitrary fixed incoming
and outgoing strips, Lemma~\ref{l.finalmu-general} gives the direct
unnormalized estimate $0<\kappa\le Cy/x^2$. In particular it proves
Lemmas~\ref{l.muest}--\ref{l.muestdiagonal} with constants depending
on the strip heights. No normalization of a general passage to
$y(0)=1$ without rescaling curvature is understood.

\subsubsection{Upper and lower bounds for $d(s)$.}\label{mu-absorbing}

Recall that $x(s)=x_0 e^s$ and $\tau(s)=\tau_0 e^{-2s}$ with
$\tau(0)=O(1)/x(0)=O(\exp(2T_0))$, $\tau(T_0)=O(1)$,
$x(T_0)=O(\exp(-T_0))$, and $\tau(2T_0)=O(\exp(-2T_0))$,
$x(2T_0)=O(1)$. Denote $S=T_0$; we have
\[
\adjustbox{max width=\linewidth}{$\displaystyle
b(s,\lambda ,S):= 6 \left(
\dfrac{ \tau(s)}{(1+\tau(s))^2 x(s)} +
\dfrac{3+4 \tau(s)+(4+3 \tau(s))\lambda (s)}{(1+\tau(s))^2 x(s)}  \,\lambda (s)+
\dfrac{\lambda ^3(s) }{(1+\tau(s))^2 x(s)} \right)
$}
\]
and
\[
\adjustbox{max width=\linewidth}{$\displaystyle
d(s,\lambda ):=
\dfrac{9 + 6 \tau(s) }{1+\tau(s)}+ \dfrac{9\lambda (s) }{1+\tau(s)}.
$}
\]

\begin{Lemma} \label{lemma:estimate_on_d}
If $\lambda $ belongs to the absorbing interval $\left[2\lambda _0(\tau(s)),\frac{9}{10}\lambda _0(\tau(s))\right]$ then
\[
\adjustbox{max width=\linewidth}{$\displaystyle
3\le d(s,\lambda )\le 9.
$}
\]
\end{Lemma}
\begin{proof}
Since $\sqrt P\ge\frac45(1+\tau)$, one has
$\lambda\ge2\lambda_0\ge-(1+\tau)/3$. Therefore
\[
 d\ge\frac{9+6\tau-3(1+\tau)}{1+\tau}
   =3+\frac3{1+\tau}>3.
\]
Since $\lambda<0$, $d\le6+3/(1+\tau)\le9$.
\end{proof}

For $\tau\in (0,1)$ the estimates given in Lemma \ref{lemma:estimate_on_d} can be improved.

\begin{Lemma} \label{lemmaimproved:estimate_on_d}
If $\lambda $ belongs to the absorbing interval $\left[2\lambda _0(\tau(s)),\frac{9}{10}\lambda _0(\tau(s))\right]$ and $\tau\in(0,1)$  then
\[
\adjustbox{max width=\linewidth}{$\displaystyle
4\le d(s,\lambda )\le 9.
$}
\]
\end{Lemma}
\begin{proof}
The preceding proof gives
$d\ge3+3/(1+\tau)\ge9/2>4$ for $0<\tau\le1$.
The upper bound is unchanged.
\end{proof}

\subsubsection{Upper and lower bound for $B(\tau(s),\lambda )$.}

We will need the following statements.
\begin{Proposition} \label{lemma:estimate_on_B}
Assume $0<\tau\le1$. If $\lambda $ belongs to the absorbing interval $\left[2\lambda _0(\tau(s)),\frac{9}{10}\lambda _0(\tau(s))\right]$ then
\[
\adjustbox{max width=\linewidth}{$\displaystyle
-42\tau\frac{1+\lambda }{1+\tau}\le B(\tau(s), \lambda )\le -3\tau\frac{1+\lambda }{1+\tau}
$}
\]
\end{Proposition}

\begin{Corollary} \label{c.upperboundforB}
There is a constant $C^*>0$ such that, whenever
$\lambda \in \left[2\lambda _0(\tau(s)),\frac{9}{10}\lambda _0(\tau(s))\right]$,
we have $|B(\tau, \lambda (\tau))|<C^*$.
\end{Corollary}
\begin{proof}
The factorization in the proof of Proposition~\ref{lemma:estimate_on_B}
and Lemma~\ref{l.ineqforB} give
$|B|\le42\tau|1+\lambda|/(1+\tau)$ for every $\tau>0$.
Since $-6/5<\lambda<0$, one has $|1+\lambda|\le1$;
thus any $C^*>42$ works. This does not use the signed inequality
outside $0<\tau\le1$.
\end{proof}

The main part of the proof of Proposition \ref{lemma:estimate_on_B} is the following estimate on $B(\tau(s),\lambda )$.
\begin{Lemma} \label{l.ineqforB}
If $\lambda $ belongs to the absorbing interval $\left[2\lambda _0(\tau(s)),\frac{9}{10}\lambda _0(\tau(s))\right]$ then
\[
\adjustbox{max width=\linewidth}{$\displaystyle
-7\tau\le \tau(1+3\lambda )+\lambda (3+\lambda )\le -0.5 \tau.
$}
\]
\end{Lemma}
\begin{proof}
Put $L=-\lambda_0$ and $a=(1+\tau)L/\tau$.
The root identity gives $L^2/\tau=(5/3)a-1$, and
$3/5\le a\le2/3$, $0\le L^2/\tau\le1/9$.
For $R_\tau(\lambda)=\lambda^2+3(1+\tau)\lambda+3\tau/2$,
\[
 R_\tau(-2L)/\tau\le4/9-18/5+3/2<0,
 \quad R_\tau(-9L/10)/\tau\le9/100-81/50+3/2<0.
\]
Convexity gives $R_\tau<0$ throughout the interval, proving the upper
bound. For $\lambda=-cL$, $9/10\le c\le2$, the expression to be
bounded, divided by $\tau$, is $c^2L^2/\tau-3ca+1\ge-3$.
This proves, in particular, the stated lower bound $-7$.
\end{proof}

Now we are ready to prove Proposition \ref{lemma:estimate_on_B}.

\begin{proof}[Proof of Proposition \ref{lemma:estimate_on_B}] We have
\[
\adjustbox{max width=\linewidth}{$\displaystyle
\begin{gathered}
    B(\tau(s), \lambda )=6\dfrac{1+\lambda (s)}{1+\tau(s)}
\left( \tau(s)+ 3 \lambda (s) \tau(s) + 3 \lambda (s) + \lambda ^2(s)\right)=\\
=6\dfrac{1+\lambda (s)}{1+\tau(s)}
\left( \tau(s)(1+ 3 \lambda (s))+ \lambda (s)(3  + \lambda (s))\right).
\end{gathered}
$}
\]
Now the required estimates follow directly from Lemma \ref{l.ineqforB}.
\end{proof}

\subsubsection{Evolution of 2-jets}

In the following statement we show that solutions of the equation
$\frac{d\kappa}{ds}=-d(s)\kappa-b(s)$ cannot grow above a uniform upper bound
on solutions of the equation $\frac{d\bar\kappa}{ds}=-b(s)$.

\begin{Lemma} \label{l.upperbound}
Assume $d(s)\ge0$ on the interval under consideration.
Denote by $\bar\kappa(\kappa_0, \bar s, T_0)$ the value at $s=T_0$ of the solution of equation
\[
\adjustbox{max width=\linewidth}{$\displaystyle
\frac{d\bar\kappa}{ds}=-b(s)\ \ \ \text{with the initial condition} \ \ \ \bar\kappa(\bar s)=\kappa_0.
$}
\]
Set $M=\sup_{\bar s\in [0, T_0], |\kappa_0|\le \kappa^*}|\bar\kappa(\kappa_0, \bar s, T_0)|$.
Then $|\kappa(T_0)|\le M$, where $\kappa(s)$ is a solution of the equation
$\frac{d\kappa}{ds}=-d(s)\kappa-b(s)$, $|\kappa(0)|\le \kappa^*$.
\end{Lemma}

\begin{proof}
If $|\kappa(T_0)|\le\kappa^*$ the assertion is immediate. If
$\kappa(T_0)>\kappa^*$, let $\hat s$ be the last time at which
$\kappa(\hat s)=\kappa^*$. On $(\hat s,T_0]$ we have $\kappa>0$, so
$\kappa'=-d\kappa-b\le-b$. Integration gives
\[
 \kappa(T_0)\le\kappa^*-\int_{\hat s}^{T_0}b(s)\,ds
               =\bar\kappa(\kappa^*,\hat s,T_0)\le M.
\]
If $\kappa(T_0)<-\kappa^*$, use the last time with
$\kappa(\hat s)=-\kappa^*$; then $\kappa'\ge-b$ until $T_0$, giving
$\kappa(T_0)\ge\bar\kappa(-\kappa^*,\hat s,T_0)\ge-M$.
These non-strict inequalities include $d=0$.
\end{proof}
It is clear that the upper bound $M$ given by Lemma \ref{l.upperbound} can be large if $T_0$ is large (or, equivalently, if we are taking the solution that starts very close to the $y$-axis). The next statement provides an explicit estimate of that upper bound in terms of $T_0$.

\begin{Lemma} \label{l.explicitupperbound}
Assume a fixed bound on the incoming slope and an initial ratio
$\tau(0)$ sufficiently large for Lemma~\ref{lem:absorb} to apply.
If $T_0$ is large enough then $M<2C^*e^{T_0}$, where $C^*>0$
is the bound on $|B|$ in the absorbing interval supplied by
Corollary~\ref{c.upperboundforB}. The threshold may depend on the
initial slope bound and on $\kappa^*$.
\end{Lemma}

\begin{proof}
Let $s_a$ be an entrance time into the absorbing interval, bounded
above by a constant $s_*$ independent of $T_0$. On $[0,s_a]$ the
slope remains uniformly bounded. Since $B(\tau,\lambda)$ is bounded
for bounded $\lambda$ and $\tau\ge0$, and $x+y\ge y=e^{-s}$,
there is a constant $K_0$, independent of $T_0$, such that
\[
 \int_0^{s_a}|b(s)|\,ds\le K_0\int_0^{s_*}e^s\,ds=:K_1.
\]
On $[s_a,T_0]$ the absorbing property and
Corollary~\ref{c.upperboundforB} give $|B|\le C^*$. Thus, for
any $\bar s\in[0,T_0]$ and $|\kappa_0|\le\kappa^*$,
\[
 |\bar\kappa(\kappa_0,\bar s,T_0)|
 \le\kappa^*+K_1+C^*\int_{s_a}^{T_0}e^s\,ds
 \le\kappa^*+K_1+C^*e^{T_0}<2C^*e^{T_0}
\]
for all sufficiently large $T_0$.

When estimating the actual curvature, the comparison in
Lemma~\ref{l.upperbound} is applied only after $s_a$, where $d\ge0$.
On the bounded initial interval, $d$ and $b$ are uniformly bounded,
so the linear curvature equation first gives a uniform bound on
$|\kappa(s_a)|$. Using that bound in place of $\kappa^*$ in the same
integral estimate gives the required bound at $T_0$. No sign
condition on $d$ before entry is needed.
\end{proof}

\begin{Lemma} \label{l.estimateonb}
Assume that
$\lambda(s)\in[2\lambda_0(\tau(s)),\frac9{10}\lambda_0(\tau(s))]$
for every $s\in[T_0,2T_0]$.
There are constants $C_1\ge C_2>0$ such that for $s\in [T_0, 2T_0]$ (i.e. for $\tau\le 1$) we have
\[
\adjustbox{max width=\linewidth}{$\displaystyle
-C_1e^{4T_0-3s}\le b(s)\le -C_2e^{4T_0-3s}.
$}
\]
\end{Lemma}
\begin{proof}[Proof of Lemma \ref{l.estimateonb}.]
Notice first that if $\tau\le 1$ then $\lambda _0(\tau)\in \left[-\frac{1}{3}, 0\right)$. Indeed,
\[
\adjustbox{max width=\linewidth}{$\displaystyle
\lambda _0|_{\tau=1}=-\frac{5}{6}\cdot 2+\frac{5}{6}\sqrt{2^2-1.44}=-\frac{5}{3}+\frac{5}{6}\cdot 1.6=-\frac{1}{3},
$}
\]
and $\lambda _0(1)<\lambda _0(\tau)<0$ if $\tau<1$. This implies that if $\lambda \in [2\lambda _0, \frac{9}{10}\lambda _0]$ then $\lambda \in \left[-\frac{2}{3}, 0\right)$. Therefore for $\tau\le 1$ we have
\[
\adjustbox{max width=\linewidth}{$\displaystyle
\frac{1+\lambda }{1+\tau}\in \left[\frac{1-\frac{2}{3}}{2}, 1\right)=\left[\frac{1}{6}, 1\right).
$}
\]
Proposition \ref{lemma:estimate_on_B} implies now that for $\tau\le 1$ we have $-42\tau\le B(\lambda , \tau)\le -\frac{1}{2}\tau$.
Since $x(s)=e^{s-2T_0}, y(s)=e^{-s}$, and $\tau=\frac{y(s)}{x(s)}=e^{2T_0-2s}$, we have
$
b(s)=\frac{B(\lambda , s)}{x(s)+y(s)}=\frac{B(\lambda , s)}{e^{s-2T_0}+e^{-s}},
$
and therefore
\[
\adjustbox{max width=\linewidth}{$\displaystyle
-42\frac{e^{2T_0-2s}}{e^{s-2T_0}+e^{-s}}\le b(s)\le -\frac{1}{2}\frac{e^{2T_0-2s}}{e^{s-2T_0}+e^{-s}}.
$}
\]
Since $\tau\le 1$, we have $s\in[T_0, 2T_0]$, and $e^{s-2T_0}\ge e^{-s}$. Hence
\[
\adjustbox{max width=\linewidth}{$\displaystyle
\begin{gathered}
    -42e^{4T_0-3s}=-42\frac{e^{2T_0-2s}}{e^{s-2T_0}}\le
    -42\frac{e^{2T_0-2s}}{e^{s-2T_0}+e^{-s}}\le b(s)\le \\
    \le -\frac{1}{2}\frac{e^{2T_0-2s}}{e^{s-2T_0}+e^{-s}}
    \le -\frac{1}{2}\frac{e^{2T_0-2s}}{2e^{s-2T_0}}=-\frac{1}{4}e^{4T_0-3s}.
\end{gathered}
$}
\]
\end{proof}

\begin{Lemma} \label{l.finalmu}
Given constants $C^*>0$ and $C_1\ge C_2>0$, there is a constant $C>0$ such that for all
large enough $T_0$ the following holds. Suppose $\kappa(s)$, $s\in [T_0, 2T_0]$
is a solution of the equation
\begin{eqnarray} \label{e.solutionmu}
\frac{d\kappa}{ds}=-d(s)\kappa-b(s), \ \ |\kappa(T_0)|\le 2C^*e^{T_0},
\end{eqnarray}
and the coefficients $d(s), b(s)$ satisfy the following estimates for $s\in[T_0, 2T_0]$ :
\[
\adjustbox{max width=\linewidth}{$\displaystyle
4\le d(s)\le 9, \ \ \ \text{and}\ \ b(s)\in[-C_1e^{4T_0-3s}, -C_2e^{4T_0-3s}].
$}
\]
Then $\kappa(2T_0)\in (0, Ce^{-2T_0})$.
\end{Lemma}
\begin{proof}
Variation of constants and $4\le d\le9$ give
\begin{align*}
 \kappa(2T_0)
 &\le 2C^*e^{-3T_0}
       +C_1\int_{T_0}^{2T_0}e^{-4(2T_0-s)}e^{4T_0-3s}\,ds\\
 &\le 2C^*e^{-3T_0}+C_1e^{-2T_0},\\
 \kappa(2T_0)
 &\ge-2C^*e^{-3T_0}
       +C_2\int_{T_0}^{2T_0}e^{-9(2T_0-s)}e^{4T_0-3s}\,ds\\
 &=-2C^*e^{-3T_0}+\frac{C_2}{6}
                           (e^{-2T_0}-e^{-8T_0}).
\end{align*}
The lower bound is positive for large $T_0$, and the upper bound
is less than $Ce^{-2T_0}$ for any fixed $C>C_1+2C^*$.
\end{proof}

\section{Evolution of curves in the general case}\label{sec:revised-general}
We retain the strategy of the simplified system, but use fixed barriers
and integral estimates. This avoids differentiating moving roots with
respect to the rapidly varying phase.
Throughout this section, $r=x+y$, $a=1+A$, $b=1+B$, and
\begin{equation}\label{e.generalsyst}
 F=xr^3a(x,y,t,e),\qquad F^{(2)}=-yr^3b(x,y,t,e).
\end{equation}
A constant positive factor multiplying both equations is harmless: it
is included in $ds/dt$ below. We assume $A,B\in O_\eta$ for some
$\eta>0$, with two spatial derivatives uniformly bounded as specified in
Lemma~\ref{lem:remainder-calculus}. Parameter estimates additionally
require the same bounds for $\partial_e A,\partial_e B$. For the Sitnikov
normal form $\eta=4$; the local argument does not require this particular
value. We work in $0<x,y\le\delta$, with $\delta$ fixed and sufficiently
small. All constants may depend on $\delta$, the remainder bounds and
the initial slope and curvature bounds, but not on the passage length.

\subsection{Time conventions and uniform passage estimates}
All slopes and curvatures are those of images at a common physical time.
Only after their differential equations have been computed do we
reparametrize each solution by $ds=r^3dt$. In particular,
\begin{equation}\label{e.generalsystlb}
 x_s=ax,\qquad y_s=-by,\qquad t_s=r^{-3},\qquad
 \tau_s=-(a+b)\tau,\quad \tau=y/x.
\end{equation}
A fixed rescaled-time map is not the map being differentiated.

\begin{Lemma}\label{lem:integrable-errors}
For a passage contained in this square, and every $p>0$,
\[
 \int_0^S r(s)^p\,ds\le C_p\delta^p.
\]
Consequently, uniformly on the whole passage,
\[
 x(s)\asymp x(0)e^s,\quad y(s)\asymp y(0)e^{-s},\quad
 x(s)y(s)\asymp x(0)y(0).
\]
Here $u\asymp v$ means $C^{-1}v\le u\le Cv$ with a uniform constant.
\end{Lemma}
\begin{proof}
Take $1/2\le a,b\le3/2$. Since $x$ increases and $y$ decreases,
\[
 \int_0^S x^pds\le2\int_{x(0)}^{x(S)}x^{p-1}dx
 \le2\delta^p/p,
 \qquad \int_0^S y^pds\le2\delta^p/p.
\]
Use $(x+y)^p\le C_p(x^p+y^p)$ and integrate
$(\log x)_s=1+A$, $(\log y)_s=-1-B$.
The logarithmic errors are bounded by $C\delta^\eta$, uniformly in $S$.
\end{proof}
For an incoming point with $y(0)$ in a fixed compact interval of
$(0,\delta]$ and an outgoing point with $x(S)$ in such an interval,
\[
 S=\log\frac{x(S)}{x(0)}+O(\delta^\eta).
\]
The physical time is $\int_0^S r(s)^{-3}ds$ (with the constant factor
restored if necessary), not $S$. To study the Poincar\'e map we evaluate
these equations at physical times $NP$, where $P$ is the fixed forcing
period. No estimate $N=S+O(1)$ is asserted.

\subsection{Fixed barriers for slopes}
Let $\mathcal P(\tau,\lambda)=-3\tau/(1+\tau)-5\lambda-
3\lambda^2/(1+\tau)$. Differentiating $F,F^{(2)}$ at fixed $t,e$ gives
\begin{equation}\label{eq:revised-slope}
 \lambda_s=\mathcal P(\tau,\lambda)+E,\qquad
 |E|\le Cr^\eta\left(\frac{\tau}{1+\tau}+|\lambda|
                   +\frac{\lambda^2}{1+\tau}\right).
\end{equation}
The factors $y$ in $F^{(2)}_x$ and $x$ in $F_y$ are essential here.
Put $h(\tau)=\tau/(1+\tau)$.

\begin{Lemma}\label{l.absorbgen}\label{lem:absorbgen}\label{lem:absorbdiagonalgen}
For sufficiently small $\delta$, the intervals
$[-2h(\tau),-h(\tau)/2]$ are forward invariant for the slope equation.
Every bounded incoming slope enters these intervals in uniformly bounded
rescaled time, provided the initial ratio $\tau(0)$ is sufficiently large.
For a diagonal initial jet $\tau(0)=1$, $\lambda(0)=1$, one has
$-2h(\tau)\le\lambda\le\tau$ throughout the outgoing passage.
Provided the local passage lasts long enough, after a uniformly bounded further time, in both cases
\begin{equation}\label{eq:outgoing-cone}
 0<\tau\le1/20,\qquad -\frac65\tau\le\lambda\le-\frac45\tau.
\end{equation}
In the incoming case the bound on this further time is measured from
crossing a fixed sufficiently small ratio, not from the initial point.
\end{Lemma}
\begin{proof}
For the model equation, at $\lambda=-kh$ the difference between the
vector field and the moving boundary, divided by $h$, is
\[
 -3+5k-3k^2\frac{\tau}{(1+\tau)^2}-\frac{2k}{1+\tau}.
\]
For $k=2$, on putting $z=(1+\tau)^{-1}$, this is
$7-16z+12z^2\ge5/3$. For $k=1/2$ it is at most $-1/2$.
The errors in both comparisons are bounded by $C\delta^\eta$.
Thus both boundary inequalities remain strict.
For bounded incoming slopes and large $\tau$, the equation is uniformly
close, on any fixed bounded slope interval and bounded time interval,
to $\lambda_s=-5\lambda-3$. Its solutions enter the interior of
$[-2,-1/2]$ in a time depending only on the initial bound. Choose that
time first, then choose $\delta$ small and $\tau(0)$ large enough.
Continuous dependence on this fixed interval, followed by the strict
boundary inequalities, proves entry and subsequent invariance.

For the diagonal, the lower boundary is the same, while at
$\lambda=\tau$ the model difference $\lambda_s-\tau_s$ equals
$-6\tau$; the error is $O(\delta^\eta\tau)$ for $\tau\le1$.
This proves $|\lambda|\le2\tau$. Put $k=-\lambda/\tau$. On bounded
$k$ intervals its equation is
\[
 k_s=\frac{3}{1+\tau}-3k+
             \frac{3k^2\tau}{1+\tau}+O(r^\eta)(1+|k|+k^2).
\]
Once $\tau\le1/20$, the model right hand side is greater than $0.54$
at $k=4/5$ and less than $-0.53$ at $k=6/5$.
On the bounded interval $-1\le k\le2$ it is strictly decreasing in $k$
for $\tau\le1/20$. Thus outside $[4/5,6/5]$ it points toward this
interval with speed at least $0.4$, after reducing $\delta$.
Both the incoming cone and the diagonal cone put $k$ in $[-1,2]$
when $\tau=1/20$. This proves the last assertion and its uniformity.
\end{proof}

\subsection{Stretching and projective contraction}
\begin{Lemma}\label{lem:stretchgen}
For an incoming--outgoing passage as above, write $\varepsilon=x(0)$.
For any two initial slopes $|\lambda_1(0)|,|\lambda_2(0)|\le L$,
\begin{equation}\label{e.x5gen}
 |\lambda_1(S)-\lambda_2(S)|\asymp
 \varepsilon^5|\lambda_1(0)-\lambda_2(0)|.
\end{equation}
For every nonzero initial vector $v$ with $|v_y/v_x|\le L$, the
differential of the fixed physical-time map satisfies
\[
 \|D\phi_{t(S),t(0)}v\|\asymp\varepsilon^{-5/2}\|v\|.
\]
The same assertions hold when the terminal physical time is $NP$,
provided the orbit ends in the fixed outgoing strip and stays in the
local square during the passage.
\end{Lemma}
\begin{proof}
The Riccati equation is quadratic in $\lambda$. Subtracting the two
exact equations gives
\[
 (\lambda_1-\lambda_2)_s=
 \left[-5-\frac{3(\lambda_1+\lambda_2)}{1+\tau}
             +O(r^\eta)(1+|\lambda_1|+|\lambda_2|)\right]
 (\lambda_1-\lambda_2).
\]
All slopes remain bounded. After a bounded initial interval the broad
cone gives $|\lambda_i|/(1+\tau)\le2\tau/(1+\tau)^2$.
Since $-(\log\tau)_s=a+b\ge1$,
\[
 \int\frac{\tau}{(1+\tau)^2}\,ds
 \le\int_0^\infty\frac{d\tau}{(1+\tau)^2}=1.
\]
Lemma~\ref{lem:integrable-errors} therefore gives
$|\lambda_1(S)-\lambda_2(S)|\asymp e^{-5S}|\lambda_1(0)-\lambda_2(0)|$.
This includes identical initial slopes by uniqueness.

For the first component $u$ of a tangent vector, the physical-time
variational equation, reparametrized along the base orbit, gives
\[
 (\log|u|)_s=\frac{4+\tau+3\lambda}{1+\tau}+O(r^\eta)
 =\frac52+\frac32\frac{x-y}{x+y}+\frac{3\lambda}{1+\tau}
       +O(r^\eta).
\]
Also $(\log r)_s=(x-y)/r+O(r^\eta)$. Integration yields
\[
 |u(S)/u(0)|\asymp e^{5S/2}\bigl(r(S)/r(0)\bigr)^{3/2}.
\]
The endpoint ratio is bounded above and below. Bounded slopes identify
$|u|$ with the vector norm. Finally $e^S\asymp\varepsilon^{-1}$.
These are derivatives at the fixed terminal physical time; no derivative
of an orbit-dependent exit map is substituted for them.
\end{proof}

\subsection{Curvature estimates}
\begin{Lemma}\label{lemma:muformula-general}
The exact curvature equation has the form
\[
 \kappa_s=-d^g\kappa-\frac{B^g}{r},
 \qquad
 d^g=\frac{9+6\tau+9\lambda}{1+\tau}+O(r^\eta)(1+|\lambda|),
\]
where $B^g=B(\tau,\lambda)+E_B$, $B$ is the polynomial expression
in Lemma~\ref{lemma:muformula}, and
\[
 |E_B|\le Cr^\eta
       \left(\frac{\tau}{1+\tau}+|\lambda|+\lambda^2+|\lambda|^3\right).
\]
\end{Lemma}
\begin{proof}
Apply the exact jet identity (\ref{eq:exact-jets}) and differentiate
$x r^3A$ and $-y r^3B$ twice. The term independent of $\lambda$ in
the curvature source is $F^{(2)}_{xx}$ and retains the factor $y$.
The other terms contain respectively $\lambda,\lambda^2,\lambda^3$.
The weighted derivative bounds give precisely the stated estimate.
\end{proof}
\begin{Lemma}\label{lemma:estimate_on_d-general}
\label{lemmaimproved:estimate_on_d-general}\label{l.ineqforB-general}
After entry into the broad cone one has $d^g\ge5/2$ and $|B^g|\le C$.
In the outgoing cone (\ref{eq:outgoing-cone}), after reducing $\delta$,
\begin{equation}\label{eq:outgoing-curvature-coeff}
 8\le d^g\le10,\qquad -40\tau\le B^g\le-5\tau.
\end{equation}
\end{Lemma}
\begin{proof}
In the broad cone $\lambda\ge-2\tau/(1+\tau)$, so the model value
of $d$ is at least $6-15z+18z^2\ge23/8$, $z=(1+\tau)^{-1}$.
The polynomial expression for $B$ is bounded there.
In the outgoing cone write $\lambda=-k\tau$, $4/5\le k\le6/5$.
Then
\[
 d=\frac{9+(6-9k)\tau}{1+\tau},\qquad
 \frac B\tau=\frac{6-(18+24\tau)k+(24+18\tau)k^2\tau
                         -6k^3\tau^2}{1+\tau}.
\]
For $0\le\tau\le1/20$, the first expression lies strictly between
$8$ and $10$, and the second strictly between $-40$ and $-5$.
For example its upper bound follows from
$6-18(4/5)+(24+18/20)(6/5)^2/20<-6$ before division by $1+\tau$.
The error in $B$ is $O(r^\eta\tau)$, not merely $O(r^\eta)$;
hence these strict bounds survive uniformly as $\tau\to0$.
\end{proof}

\begin{Lemma}\label{l.explicitupperbound-general}
Let $x=y=h_*$ at the diagonal crossing of an incoming orbit. At entry
into the outgoing cone one has $|\kappa|\le C/h_*$. The same estimate
holds for a diagonal initial jet with $x(0)=y(0)=h_*$ and $\kappa(0)=0$.
\end{Lemma}
\begin{proof}
Before entry into the broad cone only a bounded rescaled time elapses;
the initial point is in the fixed incoming strip. The linear curvature
equation therefore gives a uniform bound at this entry. Thereafter,
$d^g\ge5/2$ and $|B^g|\le C$, so variation of constants bounds
$|\kappa(s)|$ by a constant plus $C\int_0^s r(u)^{-1}du$.
Before the diagonal, $y(u)\ge h_*\exp((s_*-u)/2)$, where $s_*$
is the diagonal crossing time. Thus this integral up to $s_*$ is at
most $2/h_*$. From the diagonal to entry into the outgoing cone the
rescaled time is bounded and $x\asymp h_*$, $y\asymp h_*$, with
constants independent of the incoming point. This adds at most $C/h_*$.
For a diagonal initial jet the entire initial interval is bounded, and
the same variation-of-constants argument applies, without needing a
positive lower bound on $d^g$ before entry.
\end{proof}

\begin{Lemma}\label{l.estimateonb-general}\label{l.finalmu-general}
For sufficiently long incoming--outgoing passages, and for diagonal
initial jets sufficiently close to the origin, the outgoing curvature
satisfies
\[
 0<c\,\frac{y}{x^2}\le\kappa\le C\,\frac{y}{x^2}.
\]
In particular, on a fixed outgoing strip $0<\kappa\le C y$.
\end{Lemma}
\begin{proof}
In the outgoing cone put $w=\kappa x^2/y$. The preceding equations give
\[
 w_s=-D(s)w+Q(s),\qquad
 D=d^g-2a-b,\qquad Q=-\frac{B^g}{\tau(1+\tau)}.
\]
By (\ref{eq:outgoing-curvature-coeff}) we may arrange
$4\le D\le8$ and $4\le Q\le40$.
At entry, $x,y\asymp h_*$ and the preceding lemma implies $|w|\le C_0$.
Variation of constants therefore gives, after elapsed time $v$,
\[
 -C_0e^{-4v}+\tfrac12(1-e^{-8v})\le w(v)
 \le C_0e^{-4v}+10.
\]
For $v$ larger than a constant depending only on $C_0$, this puts $w$
between two positive constants. The remaining outgoing time tends to
infinity as $h_*\to0$, by Lemma~\ref{lem:integrable-errors}.
This proves the assertion, including strict positivity.
\end{proof}
The slope cone, stretching lemma and curvature estimate prove the local
spatial conclusions of Theorems~\ref{t.5} and~\ref{t.c2lyambdalemma}
under the stated normal-form hypotheses. The parameter conclusion is
proved in Section~\ref{sec:parameter-revised}. In particular, none of
these conclusions relies on uniform finite-time continuity over an
unbounded time interval.

To justify that the diagonal image is a graph on a fixed outgoing interval
for all large $N$, parametrize the initial diagonal by $h>0$. Its tangent
has slope one. The diagonal cone prevents a vertical tangent, and the
variational equation for its $x$-component keeps that component strictly
positive. Thus $h\mapsto x(t;h,h)$ is strictly increasing while the local
passage is defined. For fixed $t$, it tends to zero as $h\to0$; for a
fixed positive initial $h$ it reaches the outgoing boundary in finite
time. The intermediate value theorem therefore gives the required
local branch. Uniform bounded-time exit for $h$ bounded away from zero
shows that the initial $h$ corresponding to an outgoing point tends
to zero as $t\to\infty$, uniformly on compact outgoing intervals and
compact parameter sets. The same argument applies to an incoming graph
once its initial point is sufficiently close to the stable axis, using
Lemma~\ref{lem:stretchgen} for nonvanishing of its projected tangent.

\section{Splitting estimates for the RPC3BP}

Melnikov expansions for this problem were studied by Mart\'inez and
Pinyol~\cite{MP}. We use the differentiable estimates of Guardia,
Mart\'in and Seara~\cite{GMS2016}, in the normalization specified below.
These estimates control the actual invariant-manifold splitting and
its derivatives, which are the quantities required in the tangency
construction.

\subsection{Mass dependence and differentiable splitting estimates}
\label{ss:rpc-splitting-repair}

On a fixed compact section, after fixing an energy and a transverse
coordinate, smooth dependence gives
\[
 D_\mu=\mu\,\partial_\mu D_0+O_C(\mu^2).
\]
This Taylor formula is useful when the mass tends to zero at fixed
energy. For a fixed positive mass and increasing energy, we instead
use the relative splitting estimates in the following proposition.

\begin{Proposition}[Differentiable splitting input from \cite{GMS2016}]
\label{prop:rpc-actual-splitting}
Use the scaling of \cite[Sections 2--3]{GMS2016}, with angular
momentum at infinity $G_0>0$, $r=G_0^2\widetilde r$,
$P_r=G_0^{-1}\widetilde y$, and the time rescaling $t=G_0^3\widetilde t$. Here $v$ is the
unperturbed homoclinic time used to parametrize the generating functions,
and $\xi$ is their auxiliary angular variable.
Thus the rotating Hamiltonian is $-G_0$ before scaling and the
Jacobi constant in (\ref{Jacobi}) is $C=2G_0$.
The symbol $\theta$ in this splitting calculation denotes the phase
$\xi+G_0^3v$, rather than the angular time $-\varphi$.
Let $L$ be the modified Melnikov
potential defined by the convergent integral (28) of that paper.
For $0<\mu\le1/2$ and sufficiently large $G_0$,
\begin{align*}
 L&=L^{[0]}+2\sum_{\ell\ge1}L^{[\ell]}\cos(\ell\theta),\\
 L^{[1]}&=-\frac{\sqrt\pi}{4\sqrt2}\,
       \mu(1-\mu)(1-2\mu)G_0^{-3/2}e^{-G_0^3/3}
                                      (1+O(G_0^{-1})),\\
 L^{[2]}&=2\sqrt\pi\,\mu(1-\mu)G_0^{1/2}e^{-2G_0^3/3}
                                      (1+O(G_0^{-1})).
\end{align*}
In particular, differentiating the cosine series gives first- and
second-harmonic amplitudes $2|L^{[1]}|$ and $4|L^{[2]}|$.
Thus the leading equal-mass amplitude is
$8\sqrt{\pi}\,\mu(1-\mu)G_0^{1/2}e^{-2G_0^3/3}$.
This is the normalization used in
Proposition~\ref{prop:normalized-potential-parameter}.
On the common compact domain $v_-<v<v_+$ of the invariant-manifold
generating functions, put $D=T_1^u-T_1^s$. There is a constant
$E=E(\mu,G_0)$ such that, for $0\le m+n\le2$,
\begin{equation}\label{eq:rpc-generating-error}
 \left|\partial_v^m\partial_\xi^n(D-L-E)\right|
 \le K\mu^2\left((1-2\mu)G_0^{-2+3m}e^{-G_0^3/3}
             +G_0^{-1/2+3m}e^{-2G_0^3/3}\right).
\end{equation}
The stable and unstable graphs on $\varphi=\varphi_0$ satisfy
\begin{equation}\label{eq:rpc-graph-from-generating}
 Y^u-Y^s=\widetilde y_h(v)^{-1}
 \left(\partial_vD-\widetilde r_h(v)^{-2}\partial_\xi D\right)
       \big|_{\xi=\varphi_0-\widetilde\alpha_h(v)}.
\end{equation}
\end{Proposition}
\begin{proof}
The coefficient formulas are Proposition 3.1 of \cite{GMS2016};
(\ref{eq:rpc-generating-error}) is their Theorem 3.2, including its
zeroth-order assertion. Formula (\ref{eq:rpc-graph-from-generating})
is their equation (29) applied to the difference. The energy identity
follows from $H=-G_0$ at parabolic infinity and $C=-2H$.
\end{proof}

Let us make explicit what derivatives this result controls. On a
compact subinterval of $(v_-,v_+)$, $\widetilde y_h$ is bounded away
from zero and all coefficients in
(\ref{eq:rpc-graph-from-generating}) are smooth and bounded. Define
\[
 {\cal D}_v=\partial_v-\widetilde\alpha_h'(v)\partial_\xi,
 \qquad
 {\cal E}_0=\mu^2\left((1-2\mu)G_0^{-2}e^{-G_0^3/3}
              +G_0^{-1/2}e^{-2G_0^3/3}\right).
\]
Replacing $D$ by $L$ in the graph formula makes an error bounded
by $K_1G_0^3{\cal E}_0$; its first $v$ derivative is bounded by
$K_1G_0^6{\cal E}_0$. Indeed, differentiating the restriction uses
${\cal D}_v$, and only derivatives of $D-L-E$ of order at most two
occur. Equivalently, the graph error and its derivative divided by
the phase scale $G_0^3$ have the same bound $K_1G_0^3{\cal E}_0$.
These are estimates on the actual graph difference, not formal
derivatives of a scalar big-O term.

The Fourier tail can likewise be differentiated in the phase.
The coefficient bound proved in the appendix of \cite{GMS2016}
has the form $|L^{[\ell]}|\le K G_0^{\ell-3/2}
 e^{-\ell G_0^3/3}$ for $\ell\ge3$, with a uniform constant.
Consequently, for each fixed integer $j\ge0$,
\begin{equation}\label{eq:rpc-phase-tail}
 \left\|\partial_\theta^j
       \left(2\sum_{\ell\ge3}L^{[\ell]}\cos(\ell\theta)\right)
 \right\|_\infty
 \le K_jG_0^{3/2}e^{-G_0^3}.
\end{equation}
Indeed, put $q=G_0e^{-G_0^3/3}\le1/2$ and use
$\sum_{\ell\ge3}\ell^jq^\ell\le C_jq^3$. This is a genuine
termwise-differentiation argument using an absolutely convergent
majorant. It is not an estimate for derivatives in $G_0$ or $\mu$.

For $\mu\ne1/2$ fixed, the first harmonic dominates as $G_0\to\infty$.
At $\mu=1/2$, symmetry under interchange of the primaries removes
all odd harmonics and the second harmonic is leading. Thus a single
mass-independent first-harmonic formula cannot describe the full
parameter range. The preceding estimates give the transverse
homoclinic intersections established in \cite[Theorem 2.2]{GMS2016}
for every fixed mass ratio and sufficiently large $G_0$; ratios
above $1/2$ follow by relabeling the primaries.

\begin{Remark}[Limits of the splitting input]
The cited differentiable estimate has order two for generating
functions, hence order one for the graph splitting. It must not be
silently upgraded to a second graph derivative or a derivative with
respect to the energy parameter. In particular, transverse
homoclinic intersections for all large $C$ do not by themselves
produce quadratic tangencies unfolding in $C$. The latter is the
additional input needed for the fixed-mass RPC3BP conclusion.
\end{Remark}

\

In the next two sections we use the properties of the local dynamics obtained above to construct a set of oscillatory motions of Hausdorff dimension close to two for an open set of parameters and to complete the proof of our main results.

\section{Oscillatory motions arising from hyperbolic dynamics}\label{s.oscfromhyp}

Existence of a hyperbolic set, denoted $\Lambda $, homoclinically related to $O_{\infty}$ implies existence of oscillatory motions. Here we show that the Hausdorff dimension of the set of oscillatory motions in this case is not less than the Hausdorff dimension of the hyperbolic set $\Lambda $. Namely, we prove the following result.

\begin{Theorem} \label{2-diml-hd}
Suppose that stable (unstable) manifold of $O_{\infty}$
has a point of transverse intersection with an unstable (stable) manifold of
some saddle periodic point of a zero-dimensional basic set $\Lambda$
(a transitive locally maximal hyperbolic set). Choose a periodic saddle
$p$ in this basic set. Then
\[
\adjustbox{max width=\linewidth}{$\displaystyle
\dim_H\left\{q\in\mathbb R^2:\{p,O_\infty\}\subset\alpha(q)\cap\omega(q)\right\}\ge\dim_H\Lambda.
$}
\]
\end{Theorem}

We prove this result in Section~\ref{ss:sparse-excursions} by
prescribing sparse excursions in the homoclinic class of a hyperbolic
periodic point. The older enlargement, thermodynamic, and holonomy
arguments are discussed afterwards, with their status explicitly stated.
The new dimension proof also guarantees infinitely many returns to a
fixed compact neighborhood of $p$ in both time directions. Thus the
points counted in the theorem are oscillatory, not merely recurrent
to the compactification boundary.

\subsection{Sparse excursions in a hyperbolic homoclinic class}
\label{ss:sparse-excursions}

The dimension argument can be carried out without a compact symbolic
extension containing the degenerate saddle. We use a fixed hyperbolic
horseshoe and insert prescribed finite excursions between sufficiently
long free blocks. The constants for different excursions need not be
uniform. The lengths of the free blocks are chosen after these constants.

{For a related use of approximation by hyperbolic sets to obtain
full-dimensional exceptional sets, see Barrientos, Nakano, Raibekas
and Roldan~\cite{BNRR2022}, particularly their discussion of surface
diffeomorphisms and the standard map. Their results concern points
with nonconvergent Birkhoff averages. Here we impose a different
condition: excursions toward parabolic infinity and returns to a
finite periodic orbit in both time directions. The construction below
establishes the required dimension estimate for these prescribed
excursions.}

We first record the hyperbolic constructions used below. All rectangles
in this subsection are contained in the smooth phase space, away from
infinity. Closures at infinity refer to the specified compactification.
For a periodic saddle $p$, its homoclinic class $H(p,f)$ is the
closure of transverse intersections of the stable and unstable
manifolds of its periodic orbit. We write $H(p)$ when the map is fixed.

For a hyperbolic basic set $\Lambda$, write
$h_\Lambda^s=\dim_H(\Lambda\cap W^s_{\rm loc}(x))$ and
$h_\Lambda^u=\dim_H(\Lambda\cap W^u_{\rm loc}(x))$ for its
stable and unstable slice dimensions, which are independent of
$x\in\Lambda$.

\begin{Lemma}[A return horseshoe and finite excursions]
\label{lem:hyperbolic-return-data}
Let $f$ be a $C^{1+\alpha}$ surface diffeomorphism, $\alpha>0$, and
let $\Lambda$ be a zero-dimensional basic set containing a periodic
saddle $p$. Given $\epsilon>0$, there are a coordinate rectangle
$Q=[-1,1]^2$ near $p$, an integer $L$, and finitely many return
branches $F_i=f^L$, $1\le i\le m$, with the following properties.
Their source strips are pairwise disjoint, as are their image strips;
they give a full-shift horseshoe $K\subset Q$ with
\[
 \dim_H K>\dim_H\Lambda-\epsilon.
\]
Moreover, its stable and unstable slice dimensions may be required to satisfy
\[
 h_K^u>h_\Lambda^u-\epsilon,\qquad
 h_K^s>h_\Lambda^s-\epsilon.
\]
Each branch has a cross representation on all of $Q$,
\[
 u=A_i(u',v),\qquad v'=B_i(u',v),
\]
whose component derivative row norms are at most $q<1/10$. The
coordinate projections of the strips are taken with strict crossing
margins, and the cross maps extend to neighborhoods of $Q$.

If $U$ is an open set containing a transverse homoclinic point of $p$, and $V$ is
any neighborhood of $p$, there is another branch $G=f^T$ from a
vertical strip across $Q$ onto a horizontal strip across $Q$, with
cross-component row norms at most $q$, such that every orbit segment
in this branch visits both $U$ and $V$. Its cross maps also extend
past the boundary. No uniform bound on $T$ or on higher derivative
constants as $U,V$ vary is asserted.
\end{Lemma}

\begin{proof}
We recall the return-word construction for a hyperbolic basic set;
see \cite{kh,PT} for Markov rectangles and the hyperbolic inclination
lemma, and \cite{Bo2,PV} for the pressure and dimension facts used
here. Choose a sufficiently small Markov rectangle at $p$, with
$C^{1+\alpha}$ coordinates straightening its local stable and
unstable manifolds. For a periodic $p$, use its period iterate for
this local choice; all branch times below remain iterates of $f$.
Long admissible words beginning and ending in this rectangle define
full crossing return strips. Distinct strips have disjoint compact
closures after choosing the rectangle boundaries outside the Cantor
set. The branch charts extend with $C^{1+\alpha}$ regularity across
these boundaries.
Uniform hyperbolicity makes the longitudinal cross derivatives small
for long words; small variation of the stable and unstable directions
in the rectangle makes the two transverse cross derivatives small.
This gives the stated row bounds.

For completeness, the dimension approximation uses all sufficiently
long return words, rather than a single typical family of words.
In a finite transitive Markov graph the partition sums over loops at
a fixed vertex have the same exponential growth rate as unrestricted
partition sums, along the appropriate period class. Indeed, bounded
connecting words join each admissible word to that vertex. The added
length is bounded, the number of possible connectors is finite, and
H\"older potentials have bounded variation on cylinders. Applying
these observations to the geometric potentials for $f$ and $f^{-1}$
shows that the pressures per unit time of the full shifts on long
return words converge to the corresponding pressures on $\Lambda$.
More explicitly, if $Z_l^u(t)$ is the sum of the weights
$\exp(-tS_l\log J^u)$ over the selected loops, bounded cylinder
variation gives
\[
 P_{f^l|K_l}(-t\log J^u_{f^l})=\log Z_l^u(t)+O_t(1),\qquad
 l^{-1}\log Z_l^u(t)\longrightarrow P_{f|\Lambda}(-t\log J^u).
\]
Fixed entrance and exit lengths only replace $l$ by $l+O(1)$.
The pressure convergence is locally uniform in $t$, since expansion
and contraction rates on the original basic set are uniformly bounded.
The same assertions hold for inverse stable expansion. Their zeroes
therefore converge to the unstable and stable slice dimensions. For a $C^{1+\alpha}$ surface basic set the slice formula is
$\dim_H\Lambda=h_\Lambda^u+h_\Lambda^s$; see
\cite{MM,PV}. The dependence on the map used here concerns
hyperbolic continuations, not continuity through a loss of
hyperbolicity; compare the correction \cite{MMerr}. A fixed refinement to a smaller rectangle, or
fixed entrance and exit words, does not change these pressure limits.
Thus one can obtain both the dimension estimate and the stated
strict crossing and cone bounds. If the Markov graph has a period,
use return lengths in its period class.

For the last assertion choose a transverse homoclinic point in $U$
and a finite central segment of its orbit. Its two tails approach
the orbit of $p$. Adding sufficiently long hyperbolic entrance and
exit segments gives, by transversality and the inclination lemma,
a full return strip in $Q$. In cross coordinates its two components
converge in $C^1$ to the local stable and unstable graphs of $p$.
The rectangle was chosen so that their slopes are small; hence the
same row bounds hold after increasing the two tail lengths. At a
fixed point of the central segment the whole crossing strip can be
made arbitrarily close to the chosen homoclinic point. In the long
tails the whole strip also passes arbitrarily close to $p$. This
ensures the visits to $U$ and $V$. These are finite-time assertions:
for this branch, compactness and invertibility give finite derivative
bounds and positive lower bounds for the nonzero longitudinal cross
derivatives. The return length need not be a first return time.
\end{proof}

\begin{Lemma}[Deleting sparse prescribed branches]
\label{lem:sparse-deletion}
Let $F_1,\ldots,F_m$ be the finite return family in
Lemma~\ref{lem:hyperbolic-return-data}, with horseshoe $K$. Let
$G_j^+$ and $G_j^-$, $j\ge1$, be any two sequences of additional
full crossing branches in $Q$ satisfying the same cross-component
row bound $q$. Each additional branch is a $C^1$ diffeomorphism on
a neighborhood of its compact strip. There is a fixed schedule for
inserting these branches, in order in positive and negative time,
such that the resulting set $E\subset Q$ has a continuous deletion
map $\pi:E\to K$ which is onto and satisfies
\begin{equation}\label{eq:sparse-holder}
 d(\pi x,\pi y)\le C_\beta d(x,y)^\beta
 \quad (x,y\in E),\qquad 0<\beta<1.
\end{equation}
In particular, $\dim_H E\ge\dim_H K$.
\end{Lemma}

\begin{proof}
The case of one ordinary branch is immediate, since $K$ is a single
point. Assume there are at least two. We use the maximum norm in
the coordinates of $Q$; metric comparison constants on this fixed
rectangle do not affect the conclusion.
The extra branches and their return times are prescribed independently
of all free symbols. They need not be disjoint from the ordinary
branches, since the schedule specifies which family is used at each
step. Within a free step the ordinary source and image strips are
disjoint.

\emph{Coding and graphs.}
Write $(A_n,B_n)$ for the cross maps at step $n$ of a prescribed
bi-infinite sequence of branches. The equations
\[
 u_n=A_n(u_{n+1},v_n),\qquad
 v_n=B_{n-1}(u_n,v_{n-1})
\]
define a contraction of $Q^{\mathbb Z}$ in the supremum metric,
with constant $q$. Thus each sequence has exactly one orbit.
Agreement on $2r+1$ centered steps gives distance at most $Cq^r$
at the central step. In particular the coding from the compact
space of free symbol sequences is continuous. It is injective:
step times are independent of free symbols, and distinct free symbols
use disjoint strips. Its image $E$ is compact. Removing all prescribed
steps identifies the free symbol sequence with a unique point of $K$.

A fixed past defines a horizontal graph across $Q$, and a fixed
future defines a vertical graph. One can obtain these graphs by
solving the half-infinite cross equations with $u_0$, respectively
$v_0$, prescribed, or by graph transforms. Their Lipschitz constants
are at most $\rho=1/2$. For example, the image of $v=g(u)$ is given by
\[
 U=A(u',g(U)),\qquad \widetilde g(u')=B(u',g(U)).
\]
For $\operatorname{Lip}(g)\le\rho$, the first equation is a
contraction and the new slope is at most $q/(1-q\rho)<\rho$.
The graph transform contracts the $C^0$ distance by a constant less
than one. The inverse graph transform gives the vertical assertion.
These observations also prove that any past and future can be joined
by a unique bracket belonging to $E$. Lipschitz regularity is enough
for the interval estimates below; differentiations on these graphs
are understood almost everywhere.

\emph{The bounded cost of one free block.}
Here is the estimate that prevents an error proportional to the
length of each free block. Suppose two orbit segments follow the
same $r$ ordinary branches and remain in $Q$ at their endpoints.
Their points at step $i$ satisfy
\begin{equation}\label{eq:free-block-position}
 |z_i-\bar z_i|\le Cq^{\min\{i,r-i\}},\qquad 0\le i\le r.
\end{equation}
To see this, iterate the cross-equation contraction in the interior
of the common block. After $\min\{i,r-i\}$ iterations the value at
$i$ has not used data outside the block, and changing the boundary
data changes it by at most that power of $q$ times the diameter.

For a horizontal tangent of slope $s$, its image slope and its
horizontal expansion factor are
\[
 s'=B_{u'}+\frac{B_v s A_{u'}}{1-A_vs},\qquad
 J^u=\left|\frac{1-A_vs}{A_{u'}}\right|.
\]
The derivative of the slope expression with respect to $s$ has
absolute value at most
\[
 q_{\rm sl}=\frac{q^2}{(1-q\rho)^2}<1.
\]
For the finite ordinary family, $A_{u'}$ is bounded away from zero
on the closed cross domains. Its coefficients are H\"lder continuous
with some exponent $0<\alpha_0\le1$. Consequently, along the two
segments, their horizontal slopes satisfy
\[
 |s_{i+1}-\bar s_{i+1}|
 \le q_{\rm sl}|s_i-\bar s_i|
   +C\bigl(|z_i-\bar z_i|^{\alpha_0}
                  +|z_{i+1}-\bar z_{i+1}|^{\alpha_0}\bigr).
\]
Summing this inequality and using
(\ref{eq:free-block-position}) gives a bound independent of $r$
for the sum of slope differences. The same is then true for the
logarithmic expansion differences:
\begin{equation}\label{eq:free-block-cost}
 \left|\sum_{i=0}^{r-1}\log J_i^u
       -\sum_{i=0}^{r-1}\log\bar J_i^u\right|\le C_0.
\end{equation}
Both initial slopes may be arbitrary in the horizontal cone. The
constant depends only on the finite ordinary family. Reversing the
argument gives the analogous estimate for vertical expansion under
inverse branches. Also $J^u\ge(1-q\rho)/q>1$, and the same lower
bound holds for inverse vertical expansion.

\emph{The cost of a prescribed branch.}
For every $G_j^\pm$ choose $D_j\ge1$ bounding its logarithmic
horizontal expansion on the horizontal cone and the logarithmic
vertical expansion of its inverse on the vertical cone. Enlarge it
to cover both branches with index $j$. These bounds are finite:
the branches are diffeomorphisms, the cross domains are compact,
and the denominators $1-A_vs$ and their reversed counterparts are
bounded away from zero. No uniform bound on the numbers $D_j$ is
required. Notice that this is a bound on all points of each branch,
including intervals which contain gaps of the coded set. When discussing
stable and unstable holonomies we also enlarge $D_j$ to bound the
logarithms of the derivative norms of all partial iterates along the
$j$th excursion and all inverse partial iterates, on its compact
branch. There are finitely many such iterates for each $j$. This
extra requirement does not change the choice of schedule below.

\emph{Cylinder diameters and separation across gaps.}
Take $x,y\in E$ with the same past, and suppose their first different
free future symbol occurs after $k$ common free symbols. Let
$\bar x=\pi x$, $\bar y=\pi y$, and write
\[
 S_k^u(\bar x)=\sum_{i=0}^{k-1}\log J_F^u(F^i\bar x),
\]
where $F$ denotes the piecewise ordinary return map on $K$ and the
expansion is measured in horizontal coordinates on its unstable
plaque. There is $a>0$ such that $S_k^u\ge ak$.

If $j$ prescribed steps occur before this first different free step,
set
\[
 B(k)=C_0(j+1)+\sum_{i=1}^jD_i.
\]
There are constants $c,C>0$, independent of the code and $k$, such
that
\begin{equation}\label{eq:sparse-cylinder-comparison}
 |\bar x-\bar y|\le C e^{-S_k^u(\bar x)},\qquad
 |x-y|\ge c e^{-S_k^u(\bar x)-B(k)}.
\end{equation}
We justify the lower bound on the intervening interval. At the first
different free step the two points lie in distinct ordinary source
strips. Their distance is bounded below by the positive minimum
separation of these finitely many compact strips. Since they lie on
one horizontal graph, their horizontal coordinate separation has a
uniform positive lower bound $g$ as well.

Pull back the entire graph interval between them through the common
prefix. Each full crossing branch maps a graph interval onto a full
horizontal graph; its inverse is monotone in the horizontal coordinate.
Thus every point in this pulled-back interval follows the common
prefix, even when it lies outside the final coded set. On each free
block compare its expansion with the corresponding block of the
ordinary orbit of $\bar x$. The words agree throughout that block,
so (\ref{eq:free-block-cost}) applies to every such interval point.
Every prescribed step contributes at most $D_i$. The total horizontal
expansion on the interval is therefore at most
$\exp(S_k^u(\bar x)+B(k))$. Integrating the derivative over the
interval proves the lower bound in
(\ref{eq:sparse-cylinder-comparison}). The upper bound follows by
applying (\ref{eq:free-block-cost}) to the single ordinary prefix of
length $k$: its final interval has length at most the width of $Q$.
Cone bounds compare horizontal separation with distance on the graph.
This argument does not assume a distortion estimate only on the Cantor
set. The same interval proof applies backwards to pairs with the
same future.

\emph{Choice of the schedule.}
After the excursion branches and $D_j$ have been fixed, choose
strictly increasing integers $N_j$ so that
\begin{equation}\label{eq:sparse-schedule}
 \frac{C_0(j+1)+\sum_{i=1}^jD_i}{N_j}\le\frac1j.
\end{equation}
In each direction insert the $j$th prescribed step after exactly
$N_j$ free steps. For $N_j\le k<N_{j+1}$, the quantity $B(k)$ is
bounded by the numerator in (\ref{eq:sparse-schedule}); hence
$B(k)=o(k)$. This includes prefixes ending immediately after an
excursion. A bounded initial prefix can be absorbed in constants.
For $0<\beta<1$, eventually
$\beta B(k)\le(1-\beta)S_k^u$. Equation
(\ref{eq:sparse-cylinder-comparison}) then gives
\[
 |\pi x-\pi y|\le C_\beta|x-y|^\beta
\]
for all pairs on the same horizontal plaque. The backward estimate
proves it on every vertical plaque, with constants uniform over
all pasts and futures.

For general $x,y\in E$, let $z\in E$ have the past of $x$ and the
future of $y$. Then $x,z$ share a horizontal graph and $z,y$ share
a vertical graph. Their cones imply
\[
 \max\{|x-z|,|z-y|\}\le C_\rho|x-y|.
\]
Indeed, in the maximum norm the horizontal and vertical displacement
vectors cannot cancel by more than a factor $\rho<1$. Combining the
two plaque estimates with the triangle inequality proves
(\ref{eq:sparse-holder}). Since $\pi$ maps onto $K$,
$\dim_H K\le\beta^{-1}\dim_H E$ for every $\beta<1$.
Letting $\beta\uparrow1$ completes the proof.
\end{proof}

\begin{Proposition}[Dimension preserved by prescribed excursions]
\label{prop:dimension-excursions}
Let $f$ be a $C^{1+\alpha}$ surface diffeomorphism and $\Lambda$ a
zero-dimensional basic set containing a periodic saddle $p$.
Suppose transverse homoclinic points of $p$ accumulate at a boundary
point $O$ of a specified compactification. Then
\begin{equation}\label{eq:excursion-dimension}
 \dim_H\{x:\{p,O\}\subset\alpha_f(x)\cap\omega_f(x)\}
       \ge\dim_H\Lambda.
\end{equation}
Limit sets involving $O$ are taken in the compactification, whereas
Hausdorff dimension is measured in the smooth phase space. The same
conclusion holds if $\Lambda$ is homoclinically related to $p$.
\end{Proposition}
\begin{proof}
Choose a return horseshoe $K$ as in
Lemma~\ref{lem:hyperbolic-return-data}, with dimension greater than
$\dim_H\Lambda-\epsilon$. Choose transverse homoclinic points in
successively smaller neighborhoods $U_j$ of $O$. That lemma supplies
branches visiting $U_j$ and a $1/j$-neighborhood of $p$.
Use these branches on both sides of the schedule in
Lemma~\ref{lem:sparse-deletion}. Every resulting full orbit visits
these neighborhoods at unbounded positive and negative iterate times.
It therefore has both $p$ and $O$ in its two limit sets. These points
can also be taken in $H(p)$: truncate a prescribed itinerary in both
directions and close its two ends by the local unstable and stable
axes of $p$. The finite cross equations give a transverse homoclinic
point of $p$ following that finite itinerary. The uniform cross
contraction makes these points converge to the point with the full
itinerary as both truncation lengths tend to infinity. The lemma
gives a compact subset of these points of dimension at least
$\dim_H K$. Let $\epsilon\downarrow0$.

If $\Lambda$ is homoclinically related to $p$, first include it and
the orbit of $p$ in a hyperbolic basic set, using the transverse
connections and the hyperbolic inclination lemma. A finite Markov
extension realizes these connections; its hyperbolicity follows by
adding sufficiently long entrance and exit segments in the two
hyperbolic sets. Apply the preceding argument to this basic set,
whose dimension is at least $\dim_H\Lambda$.
\end{proof}

\begin{Corollary}[The nested hyperbolic-set criterion]
\label{cor:nested-oscillatory}
Suppose a periodic saddle $p$ has zero-dimensional basic sets
$\Lambda_n$ homoclinically related to it, with
$\dim_H\Lambda_n\to2$, and transverse homoclinic points of $p$
accumulating at $O_\infty$. Then the set of two-sided oscillatory
orbits has Hausdorff dimension two in the two-dimensional phase
section. In particular this applies to a nested sequence of basic
sets containing the orbit of $p$. Mere set-theoretic containment in
$H(p)$ is not substituted for the homoclinic relation.
\end{Corollary}
\begin{proof}
Apply Proposition~\ref{prop:dimension-excursions} to each $\Lambda_n$.
Recurrence to $p$, which lies in the finite part of phase space,
ensures returns to a fixed bounded region; recurrence to $O_\infty$
ensures unbounded excursions. The lower bound tends to two, the
ambient dimension. Smooth coordinate changes on a compact rectangle
near $p$ preserve the computed Hausdorff dimension.
\end{proof}

\begin{proof}[Proof of Theorem~\ref{2-diml-hd} by sparse excursions]
Let $p$ be a saddle in the basic set $\Lambda$ with the stated
connections to $O_\infty$. We first verify that transverse
homoclinic points of $p$ accumulate at $O_\infty$; this step uses
only the local inclination statement already proved for the
parabolic passage, away from its vertex.

An arc of $W^u(p)$ crosses $W^s(O_\infty)$ transversely. Its long
passage images converge in $C^1$, on a fixed outgoing arc away from
$O_\infty$, to $W^u(O_\infty)$. At its transverse intersection
with $W^s(p)$, the nearby images consequently cross $W^s(p)$
transversely. These are homoclinic points of $p$. Their passage
segments enter arbitrarily small neighborhoods of $O_\infty$:
for any fixed finite initial segment of the limiting stable orbit,
the incoming points converge to that orbit, and this stable orbit
approaches $O_\infty$. Iterates of transverse homoclinic points
are again transverse homoclinic points to the periodic orbit.
Thus the required accumulation follows. If the two connections use
different periodic saddles of $\Lambda$, their hyperbolic homoclinic
relation supplies both connections for a common saddle.

Proposition~\ref{prop:dimension-excursions} now gives the desired
bound, with recurrence to a finite periodic orbit as well as to
infinity in both time directions. Neither a symbolic enlargement
containing $O_\infty$ nor a nonhyperbolic version of Manning's theorem
is used.
\end{proof}

\begin{Remark}[Scope of the dimension construction]
The sparse-excursion proof uses compact hyperbolic return rectangles
in the smooth phase space. It prescribes both visits to infinity and
returns to a finite periodic orbit. An alternative coding that includes
$O_\infty$ would require additional hypotheses: in the smooth McGehee
extension the line $q=0$ consists of fixed points, so the maximal
invariant set of an open neighborhood of $O_\infty$ contains an interval.
Such a set cannot be a zero-dimensional finite-type symbolic extension.
The construction above avoids this obstruction by using only finite
excursions through the parabolic region.
\end{Remark}

\begin{Remark}[Finite returns and physical oscillation]
\label{rem:finite-returns}
Membership of $O_\infty$ in both limit sets alone does not imply
oscillation: an orbit tending to infinity can have that property in
one time direction without returning to any bounded region. In the
construction above the $j$th excursion is required to visit both a
shrinking neighborhood of $O_\infty$ and a $1/j$-neighborhood of a
fixed finite periodic point $p$. The visits occur at unbounded
positive and negative iterate times. Thus a single compact
neighborhood of the finite periodic orbit is visited infinitely
often in both directions.

Let $R$ be the physical distance defining oscillation (for the
Sitnikov problem, $R=|z|$). It is bounded on this compact neighborhood,
whereas $R\to\infty$ on approach to $O_\infty$. Consequently the
constructed orbits satisfy
\[
 \liminf_{n\to+\infty}R(f^n x)<\infty,
 \qquad \limsup_{n\to+\infty}R(f^n x)=\infty,
\]
and the same assertions as $n\to-\infty$. For the fixed-period
Poincar\'e map these iterates occur at physical times $nP$, so the
same lower- and upper-limit assertions hold for the continuous-time
motion. For a return map with variable return times one must in
addition verify that these physical return times tend to infinity;
no such extra hypothesis is needed for the period map used here.
\end{Remark}

\section{Newhouse domains in the Sitnikov problem and construction of a thick hyperbolic set}

Here we show that for an open set of parameters $f_e$ has a locally maximal hyperbolic set
with persistent tangencies, homoclinically related to $O_{\infty}$, of Hausdorff dimension
close to two, and use it to complete the proof of the main results of the paper. In order to
do so, we first use $C^2$-lambda lemma for dynamics near $O_{\infty}$ (i.e. Theorem \ref{t.c2lyambdalemma}) and the explicit form of the Melnikov function for the Sitnikov
problem from \cite{GP} to prove (in Subsection \ref{ss.existenceoftangencies}) that there
are arbitrarily small values of the parameter $e$ that correspond to a quadratic tangency
between $W^s(O_{\infty})$ and $W^u(O_{\infty})$ that unfolds generically with $e$. Then
in Subsection \ref{ss.saddlesnearhomoclinicpoint} we show that there are small values of
$e$ such that $f_e$ has a hyperbolic saddle close to a homoclinic orbit of $O_\infty$
whose certain finite parts of stable and unstable manifolds are $C^2$-close to the stable
and unstable manifolds of $O_{\infty}$. Unfolding of quadratic tangencies in two-dimensional
conservative case leads to appearance of hyperbolic sets of Hausdorff dimension close to 2,
see \cite{Go}, and we apply this fact to these saddles in Subsection \ref{ss.bifurcation}.
Finally, in Subsection \ref{ss.summingup} we apply the standard genericity arguments and
the results of Section \ref{s.oscfromhyp} to these hyperbolic sets of large Hausdorff
dimension and prove the main results of the paper, Theorems \ref{Sitnikov-main} and \ref{parameter-Newhouse-Sitnikov}.

\subsection{Existence of quadratic tangencies between $W^u(O_{\infty})$ and $W^s(O_{\infty})$}\label{ss.existenceoftangencies}
We retain the construction using the images of the diagonal, but first
specify the coordinates in which the Melnikov formula is used. Its phase
is the time along the unperturbed separatrix, not an affine function of
the McGehee coordinate.

For the Sitnikov map on the full $(z,\dot z)$-plane, distinguish the
parabolic points $O_+$ and $O_-$ corresponding to $z=+\infty$ and
$z=-\infty$. The local coordinates of Section 2 at $O_+$ have
$z=2/q^2$, $\dot z=-p$. Their expanding axis $y=0$ is the local
branch of $W^u(O_+)$ with $\dot z<0$. At $e=0$ this branch continues
to $O_-$ in forward time: it coincides with the corresponding global
branch of $W^s(O_-)$. This is the pair of branches to which the splitting
formula below applies. The two ends must therefore be distinguished, even when the local
notation $O_\infty$ is used for either one separately.

Choose the origin of physical time at a symmetry of the primaries, so
$r_e(-t)=r_e(t)$. On the time-zero section the involution
\[
 R(z,\dot z)=(z,-\dot z),\qquad R f_e R=f_e^{-1},
\]
fixes both $O_+$ and $O_-$. In the exact local coordinates used above,
$R(x,y)=(y,x)$ at this phase. Thus
$\Delta=\{x=y>0\}$ is a portion of $\operatorname{Fix}R$.
The spatial symmetry $J(z,\dot z)=(-z,-\dot z)$ commutes with $f_e$
and exchanges the two parabolic points.

\begin{Theorem}\label{quad-tangency} (quadratic tangency)
There exists a sequence $e_n>0$, $e_n\to0$, such that
$W^s(O_\infty)$ and $W^u(O_\infty)$ have a quadratic tangency
that unfolds generically with $e$. On the full Sitnikov phase plane
this assertion holds for each choice $O_\infty=O_+$ and $O_\infty=O_-$.
\end{Theorem}

\begin{proof}[Proof of Theorem~\ref{quad-tangency}]

We will split the proof into several steps:

\subsubsection*{The splitting formula in the local coordinates}
Let $H_0(z,v)=v^2/2-(z^2+1/4)^{-1/2}$ be the unperturbed Hamiltonian.
Parametrize the incoming positive-$z$ tail of $H_0=0$ by its physical
time $u$. On $y=0$ at $e=0$ this gives a function $u=u(x)$ with
\begin{equation}\label{eq:melnikov-phase-local}
 \begin{gathered}
 u_x=\frac{1}{2x^4}(1+O_4),\qquad
 x_u=2x^4(1+O_4),\qquad x_{uu}=O(x^7),\\
 u(x)\longrightarrow-\infty\quad(x\to0).
 \end{gathered}
\end{equation}
These identities follow from (\ref{sitnik-1}) restricted to the
unperturbed expanding axis. All error estimates in this subsection
include the displayed spatial derivatives.

The first-order energy splitting of the two branches is a nonzero
sinusoid in $u$. We use the Melnikov calculation of Garc\'ia and
P\'erez-Chavela~\cite{GP}; see also \cite[Lemma B.1]{GMP2025} for
its formulation in coordinates measuring the energy displacement
to first order. We need only its first variation and the nonvanishing of
its coefficient. A constant shift of $u$ absorbs its sign and the
choice of the origin on the unperturbed orbit. Denote the resulting
phase by $\vartheta=u+\vartheta_0$ and the positive coefficient by
$\sigma$.

Write the global branch $W^s(O_-)$ in the parameter-dependent local
chart at $O_+$ as $y=W_e(x)$. The unstable branch $W^u(O_+)$ is exactly
$y=0$. At $e=0$, the derivative of the energy in the transverse direction is
\[
 \partial_y H_0(x,0)=-8x(1+O_4).
\]
Consequently, after the constant phase shift just described,
\begin{equation}\label{eq:correct-local-melnikov}
 W_e(x)=a(x)\bigl(e\sin\vartheta(x)+e^2\rho(\vartheta(x),e)\bigr),
 \qquad a(x)=\frac{\sigma}{8x}(1+O_4)>0.
\end{equation}
In particular,
\begin{equation}\label{eq:melnikov-amplitude-bounds}
 |a_x/a|\le C/x,\qquad |a_{xx}/a|\le C/x^2.
\end{equation}
To justify the coordinate conversion, differentiate the two branches
at $e=0$, when they coincide. A change of their common longitudinal
parametrization contributes only a tangent vector, annihilated by
$dH_0$. Parameter derivatives of the chart cancel in the difference
of the two branches. The transverse difference is therefore the
energy splitting divided by $\partial_yH_0(x,0)$. The minus sign is
absorbed in $\vartheta_0$. This proves the first-order coefficient
in (\ref{eq:correct-local-melnikov}); it does not identify $x$ with $u$.

Fix a compact phase interval on this tail before letting $e\to0$.
Lemma~\ref{lem:uniform-separatrices} and finite-time continuation of
its invariant graphs imply joint smoothness in the graph coordinate
and $e$ on this interval. The continuation is nonsingular: the
unperturbed orbit has nonzero velocity on every finite segment crossing
$z=0$, and $r_e(t)>0$ for the small eccentricities considered here.
Taylor's formula in $e$, in the $C^2$ norm of the graph coordinate,
therefore gives (\ref{eq:correct-local-melnikov}) with
\begin{equation}\label{eq:melnikov-remainder-bounds}
 \sup\bigl(|\rho|+|\rho_\vartheta|+|\rho_{\vartheta\vartheta}|
                     +|\rho_e|\bigr)\le M.
\end{equation}
Here and below the bound is on the chosen compact phase interval and
$0\le e\le e_0$. The constant $M$ may depend on that interval; no
uniform remainder estimate all the way to infinity is asserted.
The derivative estimates follow from smooth dependence and Taylor's
formula, not from differentiating an unspecified scalar $O(e^2)$ term.

\subsubsection*{A positive lobe and an interior first contact}
We now give the contact argument with constants independent of the
iterate. Choose an integer $j$ sufficiently negative and use the phase
coordinate
\[
 \theta=\vartheta(x)-2\pi j,\qquad
 I=[-\pi/6,7\pi/6],\qquad \theta_c=\pi/2.
\]
Since $u(x)\to-\infty$ and $u_x>0$, the corresponding $x$-interval
lies arbitrarily close to zero as $j\to-\infty$. Write its upper
endpoint as $\delta$. Formula (\ref{eq:melnikov-phase-local}) implies
that its length is $O(\delta^4)$ and that $x\asymp\delta$ on this
interval. Work in a slightly larger local square and fix this interval
once and for all. In particular, all subsequent choices of small
$e_0$ and large $N$ are made after the phase interval has been fixed.

For sufficiently large $N$, let $y=\Delta_N(x,e)>0$ be the local
passage component of $f_e^N(\Delta)$ over this interval. The results of
Section~\ref{sec:revised-general}, uniformly for small $e$, give
\begin{equation}\label{eq:diagonal-relative-bounds}
 |\partial_x\Delta_N|\le C\Delta_N/x,\qquad
 |\partial_x^2\Delta_N|\le C\Delta_N/x^2,
 \qquad |\partial_e\Delta_N|\le K\Delta_N.
\end{equation}
The constants in the two spatial bounds can be fixed for all sufficiently
small local squares; making the square smaller only changes the
threshold for $N$. Indeed, the outgoing slope cone and the equation
for $w=\mu x^2/y$ give those constants directly. The parameter constant
$K$ is fixed after the local interval has been chosen.
Moreover, $\Delta_N\to0$ uniformly on this interval and on a sufficiently
small compact parameter interval as $N\to\infty$.

Put
\[
 H_N(\theta,e)=\frac{\Delta_N(x(\theta),e)}{a(x(\theta))}.
\]
This is positive. The chain rule, (\ref{eq:melnikov-phase-local}),
(\ref{eq:melnikov-amplitude-bounds}), and
(\ref{eq:diagonal-relative-bounds}) show that
\[
 |\partial_\theta H_N|\le C\delta^3H_N,
 \qquad |\partial_\theta^2H_N|\le C\delta^6H_N,
 \qquad |\partial_eH_N|\le K H_N.
\]
For example, the second derivative is the sum of the second
$x$-derivative of $\Delta_N/a$, multiplied by $x_\theta^2=O(x^8)$,
and its first $x$-derivative multiplied by $x_{\theta\theta}=O(x^7)$.
This explains why a nonlinear change to the Melnikov phase does not
lose the required curvature control.

Take the phase interval close enough to infinity that, with
$\epsilon=1/100$,
\begin{equation}\label{eq:normalized-diagonal-bounds}
 |H_{N,\theta}|\le\epsilon H_N,
 \qquad |H_{N,\theta\theta}|\le\epsilon H_N.
\end{equation}
Decrease $e_0>0$ so that
$\exp(\epsilon |I|+Ke_0)<2$. With
$d_N=H_N(\theta_c,0)>0$, integration of the logarithmic derivative
bounds gives
\begin{equation}\label{eq:height-comparison}
 \frac12d_N\le H_N(\theta,e)\le2d_N,
 \qquad \theta\in I,\quad0\le e\le e_0,
 \qquad d_N\longrightarrow0.
\end{equation}
In these coordinates the difference of the two graphs, divided by the
positive factor $a$, is
\[
 G_N(\theta,e)=e\sin\theta+e^2\rho(\theta,e)-H_N(\theta,e).
\]
Reduce $e_0$ further so that $Me_0<1/4$. At either endpoint of $I$,
$\sin\theta=-1/2$, and hence
\[
 G_N(\theta,e)\le-e/4-H_N(\theta,e)<0.
\]
At $e=0$ the function is negative everywhere. On the other hand,
for $N$ sufficiently large that $4d_N<e_0$ and $16Md_N<1$,
\[
 G_N(\theta_c,4d_N)\ge4d_N-16Md_N^2-2d_N>0.
\]
The maximum of $G_N(\cdot,e)$ over the compact interval $I$ is a
continuous function of $e$. Let $e_N$ be its first zero in $(0,4d_N)$.
At a maximizing point $\theta_N$ we have
\begin{equation}\label{eq:first-contact}
 G_N(\theta_N,e_N)=0,\qquad
 \partial_\theta G_N(\theta_N,e_N)=0.
\end{equation}
The endpoint inequalities show that $\theta_N$ is an interior point.
Also $0<e_N<4d_N\to0$. Thus existence and interiority have been proved
before using any curvature calculation.

\subsubsection*{Quadraticity and nondegenerate unfolding}
At a contact, (\ref{eq:height-comparison}) gives
\[
 \sin\theta_N=H_N(\theta_N,e_N)/e_N+O(e_N)
 \ge1/8-Me_N>0
\]
for all sufficiently large $N$. The derivative equation in
(\ref{eq:first-contact}) and (\ref{eq:normalized-diagonal-bounds}) give
\[
 |\cos\theta_N|\le Me_N+
       \epsilon H_N(\theta_N,e_N)/e_N
 \le Me_N+\epsilon(1+Me_N)<1/2.
\]
It follows that $\theta_N\in(\pi/3,2\pi/3)$ and
$\sin\theta_N>\sqrt3/2$. Consequently,
\begin{align*}
 \partial_\theta^2G_N(\theta_N,e_N)
 &\le-\frac{\sqrt3}{2}e_N+Me_N^2+
                 \epsilon e_N(1+Me_N)<-e_N/2,\\
 \partial_eG_N(\theta_N,e_N)
 &\ge\frac{\sqrt3}{2}-2Me_N-Me_N^2-
                 K e_N(1+Me_N)>1/2,
\end{align*}
again for all sufficiently large $N$. Thus the contact is quadratic.
Since $W_e-\Delta_N=aG_N$, its second $x$-derivative at contact is
$a\theta_x^2G_{N,\theta\theta}\ne0$, and its parameter derivative
there is $aG_{N,e}\ne0$.
The implicit function theorem continues the critical point locally
as $\theta=\theta_N(e)$. The signed critical value
$g_N(e)=G_N(\theta_N(e),e)$ satisfies
$g_N'(e_N)=G_{N,e}(\theta_N,e_N)>1/2$.
This is the required nondegenerate unfolding of the contact between
$W^s(O_-)$ and $f_e^N(\Delta)$.

\subsubsection*{From contact with the diagonal to a homoclinic tangency}
Pull both arcs back by $f_e^{-N}$. The diagonal image becomes $\Delta$,
and the other arc remains an arc of $W^s(O_-)$ by invariance.
For each fixed $N$ this is a smooth parameter-dependent local
diffeomorphism. At a quadratic contact it multiplies the second-order
normal separation and the parameter derivative of the critical value
by nonzero factors. Thus the pulled-back contact is quadratic and
unfolds nondegenerately relative to $\Delta$.

We spell out the symmetry step. In the original physical coordinates
near the pulled-back point, write this stable arc as
\[
 \dot z=h_N(z,e).
\]
At the contact $(z_N,e_N)$ it satisfies
\[
 h_N=0,\qquad \partial_z h_N=0,\qquad
 \partial_z^2h_N\ne0,\qquad \partial_eh_N\ne0.
\]
Here the last two nonvanishing statements are invariant under the
preceding coordinate changes: at contact both the separation and its
first spatial derivative vanish, so derivatives of the coordinate
change add no terms to the first nonzero normal jets.
Reversibility gives $R(W^s(O_-))=W^u(O_-)$, and its reflected arc is
$\dot z=-h_N(z,e)$. Their difference is exactly $2h_N(z,e)$.
Hence these stable and unstable arcs have a quadratic homoclinic
tangency with nonzero unfolding speed at $e=e_N$. The spatial symmetry
$J$ supplies the same conclusion for $O_+$.
Taking a subsequence if necessary gives distinct parameters tending to
zero and proves Theorem~\ref{quad-tangency}.
\end{proof}

\begin{Remark}
This proof combines the first Melnikov coefficient, smooth dependence
on compact arcs, and the relative local estimates proved above. The approximation
by periodic saddles in the next subsection is a separate step.
\end{Remark}

\subsection{Construction of a hyperbolic periodic saddle with a quadratic
homoclinic tangency}\label{ss.saddlesnearhomoclinicpoint}

We retain the approximation by periodic saddles, but distinguish two
steps: constructing a family of return branches with parameter estimates,
and transferring a generically unfolding tangency to those branches.
Spatial $C^2$ convergence alone is not sufficient for the second step.
For example, the graph differences
\[
 d_l(x,e)=x^2+e-l^{-1}\sin(le)
\]
converge uniformly in $e$, in the spatial $C^2$ norm, to $x^2+e$.
Nevertheless their only tangency parameter near zero is $e=0$, where
the unfolding derivative is zero. Indeed, $e-l^{-1}\sin(le)$ is
strictly increasing, vanishes only at zero, and has derivative zero
there. Each $d_l$ is real analytic. Thus analyticity does not remove
the need for a uniform parameter estimate.

The required estimate follows by extending the graph equation to
incoming curves. We do not need convergence of mixed derivatives in
space and the parameter. In this subsection $J\Subset(0,1)$ is a
fixed compact parameter interval, and all passages start at the same
forcing phase and end at a fixed physical time $T=NP$.
The normal-form hypotheses are those of Section~\ref{sec:revised-general},
including their uniform parameter bounds. Constants may depend on
$J$, on the fixed incoming and outgoing strips, and on uniform bounds
for the incoming curves; they do not depend on $N$.

\begin{Lemma}[Parameter estimate for an incoming curve]
\label{lem:incoming-parameter}
Let $\delta_{\rm loc}$ be the fixed local-square size.
Let $y=\gamma_e(x)$, $0\le x\le\varepsilon_0$, be a smooth family
of incoming curves with uniformly bounded first and second spatial
derivatives and
\[
 0<b_0\le\gamma_e(x)\le b_1<\delta_{\rm loc}.
\]
Write its local image at time $T$ over a fixed outgoing interval
$0<a_0\le x\le a_1<\delta_{\rm loc}$ as $y=Y(x,T,e)$. For all sufficiently
large $T=NP$, this component is defined and
\begin{align}
 \|Y(\cdot,T,e)\|_{C^2_x}&\le C T^{-2/3},
 \label{eq:incoming-spatial}\\
 |\partial_eY(x,T,e)|&\le C Y(x,T,e)
                 \bigl(1+\|\partial_e\gamma_e\|_\infty\bigr).
 \label{eq:incoming-parameter}
\end{align}
If $\theta$ varies only the incoming curve and not the vector field,
then
\begin{equation}\label{eq:incoming-shape-parameter}
 |\partial_\theta Y(x,T,e,\theta)|
       \le C Y(x,T,e,\theta)\|\partial_\theta\gamma_{e,\theta}\|_\infty.
\end{equation}
The same statements hold for reversed passages, with $x$ and $y$
interchanged.
\end{Lemma}

\begin{proof}
The slope, stretching, and curvature estimates in
Section~\ref{sec:revised-general} apply uniformly to this bounded
class of incoming graphs. They give a nonvanishing projection to the
outgoing coordinate and
\[
 |Y_x|\le CY/x,\qquad |Y_{xx}|\le CY/x^2
\]
on the fixed outgoing strip, for all sufficiently long passages.
Existence over the entire outgoing interval follows by varying the
initial $x$ from zero: its image starts on the stable axis and passes
through that interval before leaving the local square. The projected
tangent does not vanish. The initial point of every such long
passage tends to the stable axis uniformly in the stated class;
otherwise the monotone local flow has a uniformly bounded exit time.

Here is a useful quantitative height estimate. Let $x(0)=\xi$ and
$y(0)=b\in[b_0,b_1]$, and put $m=\xi b$.
Lemma~\ref{lem:integrable-errors} gives
\[
 r(s)\asymp\xi e^s+b e^{-s},\qquad x(s)y(s)\asymp m.
\]
With $v=s-\tfrac12\log(b/\xi)$, the physical passage time is
comparable to
\[
 m^{-3/2}\int (e^v+e^{-v})^{-3}\,dv.
\]
The integration interval contains $[-1,1]$ for all sufficiently long
incoming--outgoing passages. The integral over $\mathbb R$ is finite,
so $T\asymp m^{-3/2}$. Since the two longitudinal endpoint coordinates
are bounded above and below, both $\xi$ and $Y$ are comparable to
$T^{-2/3}$. This proves (\ref{eq:incoming-spatial}). A fixed leading
factor in the vector field changes only the comparison constants.

For the parameter derivative, the graph equation is
$Y_t+F Y_x=F^{(2)}$. Put $V=\partial_eY$, $\lambda=Y_x$, and $W=V/Y$.
Differentiating at fixed $x,t$, and only then introducing
$ds=c r^3dt$ when the leading constant is $c$, gives
\[
 \begin{gathered}
 W_s=K W+L,\qquad \tau=y/x,\\
 K=-\frac{3(y+x\lambda)}{x+y}+O(r^\eta)(1+|\lambda|),
 \qquad L=-B_e-\frac\lambda\tau A_e.
 \end{gathered}
\]
Initially $|W(0)|\le b_0^{-1}\|\partial_e\gamma_e\|_\infty$.
For a uniformly bounded initial interval of rescaled time, the slope
is bounded and $\tau\ge1$, for all sufficiently long passages.
After this interval, Lemma~\ref{lem:absorbgen} gives
\[
 -\frac{2\tau}{1+\tau}\le\lambda
                     \le-\frac{\tau}{2(1+\tau)}.
\]
For $\tau\ge1$ the lower bound implies
$y+x\lambda\ge y(y-x)/(x+y)\ge0$, so the main term of $K$ is
nonpositive. For $\tau\le1$ its positive part is bounded by $3\tau$.
Consequently, on every subinterval of the passage,
\[
 \int K_+\,ds\le C,\qquad
 \int|L|\,ds\le C\delta_{\rm loc}^\eta.
\]
Indeed, $|\lambda|/\tau$ is uniformly bounded, the outgoing integral
of $\tau$ is bounded because $\tau_s\le-\tau$, and the integral of
$r^\eta$ is bounded by Lemma~\ref{lem:integrable-errors}.
It is the positive part of $K$, not its absolute value, that must
be bounded on the incoming part of the passage.

Variation of constants gives
\[
 |W(S)|\le C\bigl(b_0^{-1}\|\partial_e\gamma_e\|_\infty
                                      +\delta_{\rm loc}^\eta\bigr),
\]
which proves (\ref{eq:incoming-parameter}). If only the initial graph
varies, the forcing $L$ is zero and the same argument proves
(\ref{eq:incoming-shape-parameter}). Reversing time and interchanging
the coordinates preserves the normal-form hypotheses and proves the
last assertion. All differentiations above concern a fixed physical
time; the rescaled time is used only to estimate the variational equation.
\end{proof}

\begin{Lemma}[Local passage in cross coordinates]\label{lem:local-cross}
For $a\in[a_0,a_1]$, $b\in[b_0,b_1]$, and $T=NP$ sufficiently large,
there is a unique local passage
\[
 (U_N(a,b,e),b)\longmapsto(a,V_N(a,b,e)).
\]
The functions are smooth for each finite $N$, and, uniformly on this
fixed endpoint rectangle and $J$,
\begin{equation}\label{eq:local-cross-estimate}
 \|U_N\|_{C^1_{a,b}}+\|V_N\|_{C^1_{a,b}}
       +\|\partial_eU_N\|_\infty+\|\partial_eV_N\|_\infty
       \le C N^{-2/3}.
\end{equation}
In addition, images of incoming graphs in a fixed bounded $C^2$
class satisfy (\ref{eq:incoming-spatial}).
\end{Lemma}
\begin{proof}
Apply the preceding lemma first to the horizontal graphs $y=b$ and
then to their time-reversed counterparts. It gives
$U_N,V_N=O(N^{-2/3})$, the parameter bounds,
$\partial_aV_N=O(V_N)$, and $\partial_bU_N=O(U_N)$.
The projected stretching estimate gives
\[
 \partial_aU_N=(\partial_{x(0)}x(T))^{-1}=O(U_N^{5/2}).
\]
To obtain the remaining derivative, note that
\[
 \operatorname{div}(F,F^{(2)})/(c r^3)
       =3(x-y)/r+O(r^\eta)
       =3(\log r)_s+O(r^\eta).
\]
Liouville's formula and the bounded endpoint ratio therefore give
$|\det D\phi_T|\asymp1$ on these passages. Implicit differentiation
at fixed terminal $a$ now gives
\[
 \partial_bV_N=\frac{\det D\phi_T}{\partial_{x(0)}x(T)}
                         =O(U_N^{5/2}).
\]
This proves (\ref{eq:local-cross-estimate}). The final assertion is
(\ref{eq:incoming-spatial}). We deliberately do not assert convergence
of all second endpoint derivatives of $U_N,V_N$: it is the curvature
of images of bounded incoming graphs that will be used below.
\end{proof}

\begin{Lemma}[A transverse return at each fixed positive parameter]
\label{lem:transverse-return}
For all sufficiently small fixed $e>0$, each of $O_+$ and $O_-$ has
a transverse homoclinic point on its physical branches. Such a point
has a smooth continuation on a parameter interval about $e$.
\end{Lemma}
\begin{proof}
The splitting formula (\ref{eq:correct-local-melnikov}), on a fixed
compact phase interval about a simple zero of its sine term, gives
a transverse intersection from $O_+$ to $O_-$ for every sufficiently
small $e>0$. Indeed, after division by $e a(x)>0$, its graph difference
is $\sin\vartheta+e\rho$, whose zero is simple by the spatial
remainder bound. Reversibility supplies a transverse connection in
the opposite direction.

Take a compact arc of $W^u(O_-)$ through the latter connection to
$W^s(O_+)$. Its forward local passage near $O_+$ is $C^1$ close,
on a fixed outgoing strip, to $W^u(O_+)$. Continue a suitable part
a fixed finite number of iterates to the first transverse connection.
For a sufficiently long passage it crosses $W^s(O_-)$ transversely.
This is a homoclinic point of $O_-$. The spatial symmetry gives one
for $O_+$. Only the inclination estimate away from the parabolic
point is used. At a fixed positive $e$, all iterates just chosen are
finite. Smooth dependence of the compact separatrix arcs and
transversality give continuation on a sufficiently small interval
about that parameter. No uniform return time as $e\to0$ is required.
\end{proof}

Fix $n$ sufficiently large and one of the points $O_\infty=O_\pm$.
Choose a transverse homoclinic point $q^*_{e_n}$ by
Lemma~\ref{lem:transverse-return}, and then fix a compact interval
$J$ about $e_n$, contained in $(0,1)$, on which it continues.
All assertions below are uniform on this fixed interval; $J$ may
depend on $n$.

\begin{Proposition}[Periodic saddles and parameter-uniform approximation]
\label{saddle-sequence}
For all sufficiently large $N$ there is a family of hyperbolic
periodic saddles $s_{e,N}$, $e\in J$, tending uniformly to $q_e^*$.
The saddles depend smoothly on $e$ and are homoclinically related to
$O_\infty$. On fixed compact stable and unstable arcs containing the
parabolic tangency point, their invariant manifolds have graph
representations satisfying
\begin{equation}\label{e.c2limit}
 \|w_N^{s,u}-w^{s,u}\|_{C^2_x}
       +\|\partial_e(w_N^{s,u}-w^{s,u})\|_{C^0}
       \le C N^{-2/3}.
\end{equation}
Here the norms are uniform on $J$, the graph coordinate is fixed,
and $w^{s,u}$ represent the corresponding arcs of $W^{s,u}(O_\infty)$.
The periods divide $N+L$, where $L$ is fixed after choosing $J$ and
the transverse orbit.
\end{Proposition}

\begin{proof}
\emph{The return rectangle.} Choose smooth parameter-dependent
coordinates $(u,v)$ on a fixed rectangle $Q=[-d,d]^2$ about $q_e^*$,
with the unstable branch of $O_\infty$ given by $v=0$ and its stable
branch by $u=0$. Transversality permits simultaneous straightening.
Choose fixed nonnegative integers $L_1,L_2$ so that
\[
 H_e=f_e^{L_1}:Q\longrightarrow\hbox{the incoming strip},\qquad
 K_e=f_e^{L_2}:\hbox{the outgoing strip}\longrightarrow Q
\]
are the relevant finite-time maps near the two tails of the
homoclinic orbit. In the indicated charts write
$H_e=(h_1,h_2)$ and $K_e=(k_1,k_2)$. Then
\[
 h_1(0,v,e)=0,\qquad h_{1,u}(0,v,e)\ne0,
 \qquad k_2(a,0,e)=0,\qquad k_{1,a}(a,0,e)\ne0.
\]
Shrink $Q,J$ so these transversality bounds are uniform and the
longitudinal coordinates of both strips lie in fixed compact
subintervals of $(0,\delta_{\rm loc})$. Enlarge their ambient coordinate
neighborhoods slightly so that all bounds hold up to the chosen
rectangle boundaries.

Using Lemma~\ref{lem:local-cross}, the branch
$F_{e,N}=K_e\circ f_e^N\circ H_e=f_e^{N+L_1+L_2}$ is determined by
\[
 h_1(u,v,e)=U_N(a,h_2(u,v,e),e),\qquad
 u'=k_1(a,V_N(a,h_2(u,v,e),e),e).
\]
For each $(u',v)\in Q$ the limiting equations as $N\to\infty$ have
$u=0$ and the unique solution of $k_1(a,0,e)=u'$. The Jacobian in
$(u,a)$ is invertible uniformly. The implicit function theorem,
on this fixed compact rectangle, gives the cross representation
\[
 u=A_N(u',v,e),\qquad v'=B_N(u',v,e).
\]
It represents actual orbit segments staying in the local passage
square between the two fixed global maps. The estimates
(\ref{eq:local-cross-estimate}) and the identities along the axes give
\begin{equation}\label{eq:return-norm}
 \|A_N\|_{C^1_{u',v}}+\|B_N\|_{C^1_{u',v}}
 +\|\partial_eA_N\|_\infty+\|\partial_eB_N\|_\infty
       \le C N^{-2/3}.
\end{equation}
For the parameter derivatives, differentiate the displayed implicit
equations at fixed $u',v$. The potentially order-one terms cancel
because $h_1(0,v,e)=0$ and $k_2(a,0,e)=0$ for every $e$.
The remaining terms contain $U_N,V_N$ or one of their first spatial
or parameter derivatives. Thus their bounds are precisely those
in (\ref{eq:local-cross-estimate}). This calculation does not use a
complex parameter neighborhood.

\emph{The saddle and its cones.} Increase $\varepsilon_N=C N^{-2/3}$
if necessary to bound the row sums of the spatial derivatives in
(\ref{eq:return-norm}). For large $N$ the map
\[
 (u,v)\longmapsto(A_N(u,v,e),B_N(u,v,e))
\]
is a contraction of $Q$ into its interior, with constant
$\varepsilon_N<1/4$. Its fixed point is the unique fixed point of
this return branch. It tends to $(0,0)$ in the moving coordinates,
and the implicit function theorem gives smooth parameter dependence.
Moreover,
\[
 |\partial_e s_{e,N}|\le
 (1-\varepsilon_N)^{-1}
 \|\partial_e(A_N,B_N)\|_\infty=O(\varepsilon_N)
\]
in those coordinates.

For tangent vectors the cross equations give
\[
 |du|\le\varepsilon_N(|du'|+|dv|),\qquad
 |dv'|\le\varepsilon_N(|du'|+|dv|).
\]
If $|dv|\le|du|$, then
\[
 |du'|\ge\frac{1-\varepsilon_N}{\varepsilon_N}|du|,
 \qquad |dv'|\le\frac{\varepsilon_N}{1-\varepsilon_N}|du'|.
\]
The reverse inequalities give the expanding vertical cone for the
inverse map. Hence the fixed point of the return is a hyperbolic
saddle.

\emph{Spatial control of its invariant graphs.} The unstable graph
transform sends $v=g(u)$ to the graph of its returning image. In
cross coordinates it is
\[
 \widetilde g(u')=B_N(u',g(U(u')),e),\qquad
 U(u')=A_N(u',g(U(u')),e).
\]
The equation for $U$ is a contraction. The first-derivative bounds
in (\ref{eq:return-norm}) give a $C^0$ contraction of the graph
transform, with constant $C\varepsilon_N$, on graphs in a fixed
small slope cone.

To estimate curvature, we use the actual passage, rather than
unproved second-derivative estimates for the cross maps. Choose a
fixed sufficiently small $C^2$ ball of graphs across $Q$. The map
$H_e$ sends these graphs to incoming graphs with a uniform $C^2$
bound and with their height bounded away from zero. By
(\ref{eq:incoming-spatial}), their local images have $C^2$ norm
at most $C\varepsilon_N$ over the outgoing strip. The fixed smooth
map $K_e$ sends the zero graph to $v'=0$, and its longitudinal
projection has derivative bounded away from zero. The chain rule
therefore gives
\[
 \|\widetilde g\|_{C^2}\le C\varepsilon_N.
\]
For large $N$ this is inside the chosen ball. Iterating the graph
transform from the zero graph gives its unique invariant graph.
The uniform second-derivative bound passes to the $C^1$ limit as a
Lipschitz bound on the first derivative. The invariant graph is the
unstable manifold of the hyperbolic fixed point and hence is $C^2$;
its $C^2$ norm is consequently bounded by $C\varepsilon_N$.
The inverse graph transform gives the same bound for the stable
graph. Thus no assertion about convergence of every second endpoint
derivative of $U_N,V_N$ is needed.

\emph{Parameter control.} Let $g_e$ now be a smooth family of graphs
in this same $C^2$ ball. Its image under $H_e$, written over the
incoming coordinate, has parameter derivative bounded by
$C(1+\|\partial_e g_e\|_\infty)$. Lemma~\ref{lem:incoming-parameter}
and the final change of graph coordinate under $K_e$ give
\begin{equation}\label{eq:return-graph-parameter}
 \|\partial_e\widetilde g_e\|_\infty
       \le C\varepsilon_N(1+\|\partial_e g_e\|_\infty).
\end{equation}
The constant is independent of the size of the parameter derivative.
To check the last change of coordinate, solve
$u'=k_1(a,Y(a,e),e)$ for $a$. Its denominator
$k_{1,a}+k_{1,y}Y_a$ is uniformly bounded away from zero; also
$k_2(a,0,e)=0$. Every term in the derivative of the resulting height
contains $Y$, $Y_a$, or $Y_e$, and has the required bound.

For each fixed $N$, the parameter-dependent hyperbolic manifold
theorem gives the derivative of the invariant graph. Substituting
that graph in (\ref{eq:return-graph-parameter}) and absorbing
$C\varepsilon_N\|\partial_e g_e\|_\infty$ yields
\[
 \|\partial_e g_e\|_\infty\le 2C\varepsilon_N
\]
for all large $N$. The reverse passage gives the stable analogue.
This is a bound on parameter derivatives at a fixed graph coordinate.
It is not inferred from spatial convergence or from analyticity.

\emph{Continuation to the tangency and transverse connections.}
The images under $H_e$ of the invariant unstable graphs cross the
local stable axis. Their local images at time $NP$ cover any fixed
slightly larger outgoing strip, by the existence assertion in
Lemma~\ref{lem:incoming-parameter}. Choose such a strip containing
a fundamental segment of the unstable axis. This is possible since
near the origin one iterate displaces the axis coordinate by
$O(x^4)$, whereas the chosen strip has fixed positive width.
A finite continuation of a suitable portion of this strip reaches
the prescribed compact unstable arc at the tangency. The continuation
is fixed independently of $N$; after shrinking $J$ it works uniformly
in $e$. Applying the spatial and parameter estimates to the full
outgoing strip, followed by this finite continuation, proves the
unstable part of (\ref{e.c2limit}). The same construction backwards
proves the stable part. This use of a full outgoing strip is necessary:
a very short arc about $q_e^*$ need not itself cover a fundamental
segment under iteration.

Finally the unstable graph across $Q$ crosses the stable axis
$u=0$ transversely, and the stable graph crosses the unstable axis
$v=0$ transversely. Since these axes are branches of the invariant
manifolds of $O_\infty$, the saddles have the required transverse
connections in both directions. This completes the proof.
\end{proof}

\begin{figure}[htbp]
\begin{center}
\includegraphics[width=5.4in]{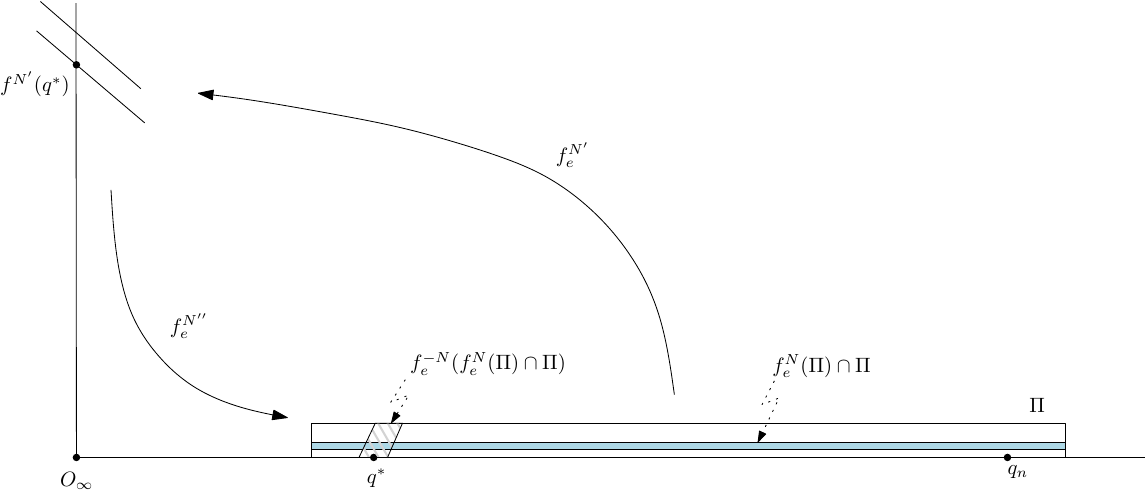}
\caption{The return rectangle near the transverse homoclinic point.
The cross-map and graph-transform estimates are uniform on a fixed
parameter interval about each $e_n>0$.}
\label{BirkhoffSmaleFig}
\end{center}
\end{figure}

\begin{Proposition}[Transfer of the unfolding]\label{saddle}
For every sufficiently large $n$, fix the tangency parameter $e_n$
from Theorem~\ref{quad-tangency}. By Proposition~\ref{saddle-sequence}, the return branches and their
continuations exist on a fixed neighborhood of $e_n$. There are parameters $e_{l,n}\to e_n$ at which
$s_{e_{l,n},l}$ has a quadratic homoclinic tangency that unfolds
generically. These saddles are homoclinically related to $O_\infty$.
\end{Proposition}

\begin{proof}
Let $d(x,e)=w^u(x,e)-w^s(x,e)$ be the signed graph difference at
the parabolic tangency, and put $d_l=w_l^u-w_l^s$. Translate the
contact to $(0,e_n)$. Then
\[
 d(0,e_n)=d_x(0,e_n)=0,\qquad
 d_{xx}(0,e_n)\ne0,\qquad d_e(0,e_n)\ne0.
\]
On a sufficiently small fixed rectangle, $d_{xx}$ and $d_e$ have
constant signs and are bounded away from zero. The convergence
(\ref{e.c2limit}) gives the same bounds for $d_{l,xx}$ and $d_{l,e}$.
The equation $d_{l,x}(x,e)=0$ therefore has a unique solution
$x=x_l(e)$ near zero; this solution is $C^1$ for each $l$. The
solutions converge uniformly to the critical branch $x(e)$ of $d$.
Indeed, the mean value theorem bounds their difference by a constant
times $\|d_{l,x}-d_x\|_\infty$.

Set $g_l(e)=d_l(x_l(e),e)$ and $g(e)=d(x(e),e)$. Then
\[
 g_l\longrightarrow g\quad\hbox{uniformly},\qquad
 g_l'(e)=d_{l,e}(x_l(e),e)\longrightarrow d_e(x(e),e)=g'(e)
\]
uniformly on a smaller parameter interval. Since $g(e_n)=0$ and
$g'(e_n)\ne0$, its values at the two endpoints have opposite signs.
The intermediate value theorem and strict monotonicity give a
unique zero $e_{l,n}$ of $g_l$, and $e_{l,n}\to e_n$. At this zero,
$d_{l,xx}\ne0$ proves quadraticity, and $g_l'\ne0$ proves generic
unfolding. The transverse connections were proved in
Proposition~\ref{saddle-sequence}.
\end{proof}

\begin{Remark}
The periodic-saddle approximation uses spatial $C^2$
estimates for images of curves, but only $C^1$ estimates for the
endpoint cross maps. The parameter derivative is controlled by the
inhomogeneous graph equation and an absorption argument on the
invariant graphs. In particular, the conclusion does not require
uniform convergence of mixed spatial-parameter derivatives, nor
uniformity of the chosen parameter intervals as $e_n\to0$.
\end{Remark}

\subsection{Conservative homoclinic bifurcations and hyperbolic sets of
large Hausdorff dimension}\label{ss.bifurcation}

In order to construct transitive invariant sets of full Hausdorff
dimension we use the notion of a {\it homoclinic class.}

Let $P$ be a hyperbolic saddle of a diffeomorphism $f$.
The {\it homoclinic class} $H(P,f)$ is the closure of the transverse
intersections of $W^s(\mathcal O(P))$ and $W^u(\mathcal O(P))$.
For a periodic saddle, these manifolds denote the unions over its
periodic orbit; this convention makes the class invariant under $f$.

It is known that $H(P,f)$ is a transitive invariant set of $f$, see
\cite{N5}. Moreover, consider all basic sets (locally maximal
transitive hyperbolic sets) that contain the saddle $P$. A homoclinic
class $H(P,f)$ is a smallest closed invariant set that contains all
of them.

\begin{Theorem}[Local form of the construction in \cite{Go}]
\label{t.thick-sets}
Let $\{f_a\}$ be a smooth family in
$\operatorname{Diff}^{\infty}(M^2,\mathrm{Leb})$ unfolding a quadratic
homoclinic tangency of a hyperbolic fixed point $s_0$ at $a=0$,
with nonzero normal unfolding speed. Let $W$ be an open neighborhood
of the closure of this homoclinic orbit together with $s_0$, and let
$J$ be any sufficiently small open parameter interval about zero.
For every $\delta>0$ there is a nonempty open set
$\mathcal U_\delta\subset J$ accumulating at zero such that:
\begin{enumerate}[(1)]
\item For each $a\in\mathcal U_\delta$ there is a zero-dimensional
hyperbolic basic set $\Delta_a\subset W$ containing the continuation
$s_a$, with $\dim_H\Delta_a>2-\delta$, and exhibiting persistent
homoclinic tangencies whose orbits remain in $W$.
\item There is a dense subset $D_\delta\subset\mathcal U_\delta$
such that, for every $a\in D_\delta$, the saddle $s_a$ has a
quadratic homoclinic tangency unfolding with nonzero normal speed
in the parameter $a$. Its entire orbit remains in $W$.
\end{enumerate}
\end{Theorem}

\begin{proof}[Explanation of the local form]
We use the construction of \cite[Theorem 4 and Section 6]{Go}.
Its first step, obtained there from Duarte's conservative bifurcation
theorem, supplies a wild basic set containing the continued saddle
and a dense set $D_0$ of parameters with quadratic homoclinic
tangencies of that saddle unfolding generically. Thus quadraticity
and nonzero speed are supplied by this construction, rather than
inferred from the word ``tangency.'' The subsequent renormalization of Mora and Romero~\cite{MR1},
in the form recalled in \cite[Theorem 8]{Go},
near parameters in $D_0$ produces open parameter sets with horseshoes
of dimension greater than $2-\delta$. The transverse connections
join each horseshoe to the continued saddle in a larger basic set.
Restricting $D_0$ to the resulting open set gives (2).

We record why the construction can be confined to $W$. Choose a
small saddle rectangle and neighborhoods of the finitely many points
of the central homoclinic segment, all with closures in $W$. The
remaining two tails lie in the saddle rectangle. The return maps
used in the construction are compositions of this finite transition
and arbitrarily long local saddle passages. Shrink the transition
neighborhoods so that every intermediate image stays in $W$.
For all sufficiently small parameters these inclusions persist.
The rescaled return domains converge to the tangency; hence their
invariant horseshoes and all intermediate iterates stay in these
chosen neighborhoods. The same applies to the wild set and the
transverse connecting strips used to adjoin the saddle. Increasing
the local passage lengths gives the required constructions in any
prescribed $J$. This proves the asserted localization, including
localization of the tangency orbits.

The invariant sets are realized by finitely many crossing rectangles
with disjoint stable and unstable strips. Their two slices are Cantor
sets (or finite sets), so the resulting basic sets are zero-dimensional.
The finite connecting construction preserves this property. These
observations give precisely the localized conclusions used here;
the additional residual conclusions of \cite[Theorem 4]{Go} are not
needed below.
\end{proof}

\begin{Remark}[Parameter domains and residual sets]
\label{rem:parameter-density}
First fix one tolerance $\delta_0>0$, for example $\delta_0=1$,
and choose a domain $\mathcal N=\mathcal U_{\delta_0}$. This domain
is kept fixed while the desired dimension tends to two. The initial
tangency parameter belongs to its closure and need not be an interior
point. The domains obtained in later applications with tolerance $1/k$
are auxiliary subsets used to prove density inside this fixed domain.

Here is the density argument used in our applications. Fix an open
parameter domain $\mathcal N$ with a dense set $D$ of nondegenerately
unfolding quadratic tangencies of a continued saddle $p_a$. Suppose
its selected transverse connections to infinity persist on this domain.
For $k\ge1$ let $V_k$ be the set of parameters admitting a
zero-dimensional hyperbolic basic set containing $p_a$ and of
dimension greater than $2-1/k$.
Hyperbolic continuation and continuity of dimension make $V_k$ open.
Given an open interval $J\subset\mathcal N$, choose $a_*\in D\cap J$.
Apply part (1) of the theorem at $a_*$ with $\delta=1/k$. Its open
parameter sets accumulate at $a_*$, so one of their points lies in $J$.
The continued saddle retains the selected transverse connections there.
Thus $V_k$ is dense in $\mathcal N$, and $\bigcap_kV_k$ is residual
in this fixed domain. The tangencies used at $a_*$ are those in
part (2) of Theorem~\ref{t.thick-sets}.
\end{Remark}

In this application a generic unfolding means a smooth conservative
family with a quadratic contact and nonzero normal splitting speed.
More explicitly, if $d(x,a)$ is a signed difference of the two local
manifold graphs at the contact $(x_*,a_*)$, the conditions are
$d=d_x=0$, $d_{xx}\ne0$, and $d_a\ne0$ there. These are the
nondegeneracy conditions used in the conservative renormalization in
\cite[Section 6]{Go}. The first-contact and periodic-saddle arguments
above verify both nonzero derivatives. A cutoff that agrees with the
family near the full tangency orbit preserves them.

\paragraph{Local return maps and periodic saddles.}
We explain the application of this theorem to the return maps used
here. The construction in \cite[Section 6]{Go} takes place in a
neighborhood of the orbit of a hyperbolic saddle and its homoclinic
tangency. In the RPC3BP these orbits, after the periodic-saddle
approximation, have compact closure in the collision-free phase
space. This closure does not contain parabolic infinity: both
ends of the homoclinic orbit converge to the finite periodic saddle.

One can realize the relevant local return maps as restrictions of
global area-preserving maps as follows. On the orbit tubes used in
the large-$C$ construction, the rotating angle is monotone. Indeed,
on the global pieces the scaled Kepler continuation has
$r\ge c_0G^2$ and $|P_\varphi|\le c_1G$; the local tails satisfy
the same inequalities after decreasing the fixed local square.
Thus $1-P_\varphi/r^2\ge1/2$ for large $G$. These strict
inequalities persist for the periodic approximations and on a
small compact parameter interval about each tangency. For the
fixed-energy, small-mass construction the same bounds follow by
continuation from the Kepler orbit before choosing that interval.
All constants here may depend on the particular tangency.

On such a tube solve the energy equation for $P_\varphi$ and use
$\theta=-\varphi$ as time. The reduced equations in the canonical
section coordinates $(r,P_r)$ are Hamiltonian with the smooth,
$2\pi$-periodic Hamiltonian $P_\varphi(r,P_r,\theta)$. Choose a
smooth periodic cutoff equal to one on a neighborhood of the
compact orbit tube and supported in a slightly larger tube where
this Hamiltonian is defined. Multiply by the cutoff and extend
by zero to the whole $(r,P_r)$-plane. Its time-$2\pi$ map is a
global smooth area-preserving diffeomorphism, and agrees with the
physical return map wherever the orbit remains in the smaller tube.
The cutoff can be fixed on a sufficiently small parameter interval,
so this is a smooth family and its unfolding derivative is unchanged.
Choose the section neighborhood $W$ in Theorem~\ref{t.thick-sets}
so small that every one-period flow segment starting there remains
in the tube where the cutoff is one. Such a neighborhood exists by
compactness of the retained orbit and continuous dependence of the
flow. The localization in that theorem puts the basic sets and all
iterates of their tangency orbits in $W$. The cutoff and physical
maps therefore agree along these entire orbits, so the basic sets
and tangencies also belong to the physical return map. The selected
transverse connections to infinity persist separately by openness
for the physical family. No extension through infinity or a
collision is involved. The Sitnikov time map is already a global
area-preserving diffeomorphism; the same cutoff argument is optional
there.

If the saddle has period $m$, first apply the theorem to $f^m$.
Its local stable and unstable manifolds and the nonzero normal
unfolding derivative are those of the periodic orbit of $f$.
The produced sets are returned to $f$ by taking
$\bigcup_{j=0}^{m-1}f^j\Lambda$. This finite saturation is
hyperbolic, has the same Hausdorff dimension, and has the required
transverse connections. Transitivity is preserved by the cyclic
saturation; mixing is not required. Thus using a periodic saddle
or a locally defined physical return map introduces no change in
the dimension or parameter conclusions used below.

 In dissipative case Newhouse \cite{n1} showed that near every  surface
diffeomorphism with a homoclinic tangency there are open sets
(nowadays called {\it Newhouse domains}\,) of maps with persistence
homoclinic tangencies. Moreover, in these open sets there are
residual subsets of maps with infinitely many attracting periodic
orbits. Later Robinson \cite{R} showed that this result can be
formulated in terms of generic one parameter unfolding of a
homoclinic tangency.

In area preserving case  Duarte \cite{Du1}, \cite{Du2}, \cite{Du4} showed
that homoclinic tangencies also lead to similar phenomena, the role of sinks
is played by elliptic points. Theorem \ref{t.thick-sets} is a stronger version
of the Duarte's result: we can control the Hausdorff dimension of the hyperbolic
sets and homoclinic classes that appear in the construction.

\subsection{Summing up}\label{ss.summingup}

Let $\{f_e\}$ be the $2\pi$-shift along the orbits of the Sitnikov
problem (\ref{Sitnikov}). Denote by $OS(e)\subset \mathbb{R}^2$ the set
of oscillatory motions of the Sitnikov problem (i.e. the set of initial
conditions $(z, \dot z)\in \mathbb{R}^2$ at time zero that correspond to
a two-sided oscillatory motion). Here we construct
an open set $\mathcal{N}\subset (0,1)$ accumulating at zero, and
a residual subset $R\subset \mathcal{N}$ such that for $e\in R$
we have $\text{\rm dim}_H\,OS(e)=2$. The flow-box argument below then proves
Theorem \ref{Sitnikov-main}.

Due to Proposition \ref{saddle} there exists a sequence of parameters
$\{e_n\}_{n\in \mathbb{N}}\subset (0,1)$, $e_n\to 0$ as $n\to \infty$, such that $f_{e_n}$
has a hyperbolic periodic saddle $s_n$ that satisfies the following properties:

\

a) in McGehee coordinates $s_n$ is homoclinically related to $O_{\infty}$;

b) $s_n$ has a quadratic homoclinic tangency that unfolds monotonically with $e$.

\

For each $n$, choose a compact interval $J_n$ with $e_n$ in its
interior on which the saddle and its selected transverse connections
to infinity continue. The periodic-saddle construction supplies such
an interval, which may depend on $n$. Set
$I_n=\operatorname{int}J_n\cap(e_n/2,2e_n)$.
Apply Theorem~\ref{t.thick-sets} at $e_n$ with the fixed tolerance
$\delta_0=1$ and with parameter interval $I_n$. This gives a nonempty
open set $\mathcal N_n\subset I_n$ accumulating at $e_n$ and a dense
set of quadratic tangencies unfolding nondegenerately. Choose a
countable dense subset $D_n$ of this set; this is possible because
$\mathcal N_n\subset\mathbb R$ has a countable base.
The domains $\mathcal N_n$ are now fixed independently of $k$.
Their union $\mathcal N=\bigcup_n\mathcal N_n$ is a Newhouse domain
for $\{f_e\}$ and accumulates at zero.
This proves Theorem \ref{parameter-Newhouse-Sitnikov}.

Applying the density argument in Remark~\ref{rem:parameter-density}
on each fixed $\mathcal N_n$ shows that for every
$n, k\in \mathbb{N}$ there exists an open dense subset
$V_{k, n}\subset \mathcal{N}_n$ such that for each $e\in V_{k, n}$
the map $f_e$ has a hyperbolic basic set
$\Lambda_{k,n}(e)$, $s_n(e)\in \Lambda_{k,n}(e)$, with
$\text{\rm dim}_H\Lambda_{k,n}(e)>2-\frac{1}{k}$.
Set $V_k=\cup_{n}V_{k,n}$, $V_k$ is open and dense in $\mathcal{N}$.
The transverse connections to $O_\infty$ are retained in the parameter
neighborhoods under consideration. The strengthened conclusion of
Theorem~\ref{2-diml-hd} supplies both excursions to infinity and returns
to the finite orbit of $s_n(e)$; see also Remark~\ref{rem:finite-returns}.
Thus for $e\in V_k$ we have
$\text{\rm dim}_H\,OS(e)>2-\frac{1}{k}$. Therefore for
$e\in R=\cap_{k\in \mathbb{N}} V_k$ we have $\text{\rm dim}_H\,OS(e)=2$.
We recall the product formula for a nonempty compact
$E\subset\mathbb R^2$ and a nondegenerate bounded interval $I$:
\begin{equation}\label{eq:interval-product-dimension}
 \dim_H(E\times I)=\dim_HE+1.
\end{equation}
For completeness, the lower bound is the product-measure argument
underlying Marstrand's product inequality~\cite{Mar}: a Frostman
measure of exponent $s<\dim_HE$, multiplied by interval length,
has ball masses bounded by a constant times $r^{s+1}$. If
$\dim_HE=0$, use $\{x\}\times I$ for any $x\in E$.
For the upper bound, cover $E$ by balls of diameters $d_i<1$ with
$\sum_i d_i^s$ arbitrarily small, where $s>\dim_HE$. The cylinder
over the $i$th ball is covered by at most $K_I/d_i$ balls of
diameter $K d_i$. This gives vanishing $(s+1)$-dimensional content.
Let $s\downarrow\dim_HE$ for the upper bound and
$s\uparrow\dim_HE$ for the lower bound.

To pass to the extended phase space, fix one of the compact scheduled
sets $E$ in the time-zero section. The extended vector field has time
component $1$. For a sufficiently short interval $I$ about zero, the
map $(x,t)\mapsto\Phi^t(0,x)$ is a smooth flow-box diffeomorphism
on a neighborhood of $E\times I$, with bounded derivatives and inverse
derivatives on a smaller compact neighborhood. Hence its image of
$E\times I$ has Hausdorff dimension $\dim_HE+1$. Every point in this
image is two-sided oscillatory, since a finite time shift preserves
oscillation. The section sets have dimensions arbitrarily close to two,
so these flow boxes give dimension at least three. The ambient dimension
is three. This proves Theorem~\ref{Sitnikov-main}.

\section{Variation with respect to parameters}\label{sec:parameter-revised}
The estimate needed in Theorem~\ref{t.c2lyambdalemma} concerns the
height of an image curve at fixed $x$ and fixed physical time. We prove
it directly in these variables. This also accounts for the dependence
of the initial point on $e$ needed to keep the terminal $x$ fixed.

Assume the parameter-uniform normal-form hypotheses of
Section~\ref{sec:revised-general}, including
$\partial_e A,\partial_e B\in O_\eta$. The forcing period and the leading
constant multiplying the vector field are independent of $e$.
For the initial diagonal, write its local image at time $t$ as
$y=Y(x,t,e)$. The slope barriers show that the projection to $x$ has
nonzero derivative throughout this outgoing passage. Hence $Y$ is
well defined on each local image interval, is differentiable in $e$,
and satisfies the graph equation
\begin{equation}\label{sitnik-e}
 Y_t+F(x,Y,t,e)Y_x=F^{(2)}(x,Y,t,e),\qquad Y(x,0,e)=x.
\end{equation}
Put $V=\partial_eY$ and $\lambda=Y_x$. Differentiation at fixed $x,t$
gives the full equation
\begin{equation}\label{sitnik-e-variation}
 V_t+F V_x=(F^{(2)}_y-F_y\lambda)V+F^{(2)}_e-F_e\lambda,\qquad V(x,0,e)=0.
\end{equation}
All derivatives of $F,F^{(2)}$ in this formula are partial derivatives
before substitution $y=Y(x,t,e)$.

\begin{Proposition}\label{prop:parameter-graph-bound}
There exists $C$, independent of the initial diagonal point and the
length of the local passage, such that
\[
 |\partial_eY(x,t,e)|\le C\delta^\eta Y(x,t,e).
\]
Thus at every fixed iterate $N$ whose local image meets the outgoing
strip, $y_N(x,e)=Y(x,NP,e)$ satisfies (\ref{e.derivine}), uniformly in $N$.
\end{Proposition}
\begin{proof}
Along a graph characteristic, put $W=V/Y$. Since $Y$ evolves by $F^{(2)}$,
(\ref{sitnik-e-variation}), reparametrized along this characteristic,
becomes
\[
 W_s=K(s)W+L(s),\quad
 K=\frac{F^{(2)}_y-F_y\lambda-F^{(2)}/y}{r^3},\quad
 L=\frac{F^{(2)}_e-F_e\lambda}{yr^3}.
\]
A fixed leading constant is included in the definition of $s$ if present.
With $F=xr^3(1+A)$, $F^{(2)}=-yr^3(1+B)$, we obtain
\[
 K=-\frac{3(y+x\lambda)}r+O(r^\eta)(1+|\lambda|),
 \qquad L=-B_e-\frac{\lambda}{\tau}A_e.
\]
The diagonal cone in Lemma~\ref{lem:absorbdiagonalgen} gives
$0<\tau\le1$ and $|\lambda|\le2\tau$ throughout the passage.
Consequently
\[
 |K|\le C(\tau+r^\eta),\qquad |L|\le Cr^\eta,
 \qquad \int_0^S\tau\,ds\le1.
\]
By Lemma~\ref{lem:integrable-errors}, $\int_0^S|K|ds\le C$ and
$\int_0^S|L|ds\le C\delta^\eta$. Since $W(0)=0$, variation of constants
gives $|W(S)|\le e^{\int|K|}\int|L|\le C\delta^\eta$.
The terminal time $NP$ is fixed while differentiating in $e$; the
orbit-dependent number $S$ is used only to estimate this derivative.
\end{proof}
\begin{Remark}
This proves the parameter estimate for the exact normal-form class.
The parameter-uniform input for the Sitnikov family is supplied by
Lemma~\ref{lem:uniform-separatrices} and the exact normal-form calculation. The incoming-curve extension needed for the periodic saddles is proved in
Lemma~\ref{lem:incoming-parameter}. Neither estimate, by itself, gives bounds
for mixed derivatives such as $\partial_x\partial_eY$; these must be
proved separately wherever a later argument requires them.
\end{Remark}

\section{A criterion for quadratic tangency}
\label{sec:general-contact-repair}

We use a first-contact criterion. Existence of a contact, its order, and its unfolding are
three separate assertions. A sign condition on $M''$ alone proves
none of them. In particular, the implication
$M>c$, $M'<-c'\Rightarrow M''>0$ can hold vacuously when $M<0$
everywhere, although a positive diagonal image cannot meet such a
graph for a positive sufficiently small parameter.

\begin{Lemma}[First contact with a positive lobe]\label{lem:separation}
Let $I=[-\pi/6,7\pi/6]$, and let $m$ be a $C^2$ function with
$\|m-\sin\|_{C^2(I)}\le1/100$. Suppose $H_N(\theta,a)>0$
and $\rho(\theta,a)$ are smooth on $I\times[0,a_0]$ and,
with constants independent of $N$,
\begin{gather*}
 \tfrac12d_N\le H_N\le2d_N,\qquad d_N\downarrow0,\\
 |H_{N,\theta}|+|H_{N,\theta\theta}|\le\tfrac1{100}H_N,
 \qquad |H_{N,a}|\le K H_N,\\
 |\rho|+|\rho_\theta|+|\rho_{\theta\theta}|+|\rho_a|\le M.
\end{gather*}
Here $\theta$ is a phase variable, $a$ is the unfolding parameter,
and $\rho$ is the remainder in the global graph.
Set $F_N(\theta,a)=a m(\theta)+a^2\rho-H_N$.
For all sufficiently large $N$ there are
$\theta_N\in\operatorname{int}I$ and $0<a_N<4d_N$ with
\[
 F_N=F_{N,\theta}=0,\qquad
 F_{N,\theta\theta}<-a_N/2,\qquad F_{N,a}>1/2
\]
at $(\theta_N,a_N)$. Hence these two graphs have a quadratic
contact which unfolds nondegenerately with $a$.
\end{Lemma}
\begin{proof}
Decrease $a_0$ so that $Ma_0<1/10$.
At either endpoint of $I$, $m\le-0.49$, so $F_N<0$.
At $a=0$ it is negative everywhere, whereas
\[
 F_N(\pi/2,4d_N)\ge3.96d_N-16Md_N^2-2d_N>0
\]
for large $N$. The continuous function
$a\mapsto\max_{\theta\in I}F_N(\theta,a)$ has a first
zero $a_N\in(0,4d_N)$. Its maximizing point is interior,
so $F_N=F_{N,\theta}=0$ there. This proves existence without any
curvature assumption on a contact that has not yet been found.

The value equation gives
$m(\theta_N)=H_N/a_N+O(Ma_N)\ge1/8-o(1)$,
and $H_N/a_N\le1.01+o(1)$. The derivative equation gives
\[
 |m'(\theta_N)|\le\tfrac1{100}H_N/a_N+Ma_N
                              \le0.0101+o(1).
\]
Thus $\sin\theta_N>0$ and $|\cos\theta_N|<1/2$ for large $N$.
It follows that $\sin\theta_N>\sqrt3/2$. Consequently,
\begin{align*}
 F_{N,\theta\theta}/a_N
 &\le-\sqrt3/2+0.01+Ma_N
                         +0.01(1.01+Ma_N)<-1/2,\\
 F_{N,a}
 &\ge\sqrt3/2-0.01-2Ma_N-Ma_N^2
                         -Ka_N(1.01+Ma_N)>1/2.
\end{align*}
The implicit function theorem continues the critical point locally.
The derivative of its critical value is $F_{N,a}$ because
$F_{N,\theta}=0$. This is the asserted unfolding.
\end{proof}

\begin{Proposition}[The reversible geometric application]
\label{prop:general-reversible-contact}
Let $f_a$ be a smooth family reversible under a fixed smooth
involution $R$, and let $R$ fix a saddle $O$ (hyperbolic or parabolic)
with smooth stable and unstable arcs. Suppose that the signed
difference between an arc of $W^s(O)$ and an arc of
$f_a^N(\operatorname{Fix}R)$, in a common graph coordinate,
is a nonvanishing smooth factor times a smooth function
$F(u,a)$ which, at $(u_*,a_*)$, satisfies
\[
 F=F_u=0,\qquad F_{uu}\ne0,\qquad F_a\ne0.
\]
Then $f_{a_*}$ has a quadratic homoclinic tangency to $O$
which unfolds nondegenerately. In particular this applies to the
contacts supplied by Lemma~\ref{lem:separation} and
Proposition~\ref{prop:energy-contact}.
\end{Proposition}
\begin{proof}
Pull the contact back by the finite diffeomorphism $f_a^{-N}$.
It is now a contact of $W^s(O)$ with $\operatorname{Fix}R$.
Near this regular fixed curve of the reflection, choose smooth
coordinates in which $R(u,v)=(u,-v)$. Write the stable arc as
$v=h(u,a)$. At contact,
\[
 h=h_u=0,\qquad h_{uu}\ne0,\qquad h_a\ne0.
\]
To justify invariance of these statements, a smooth change of graph
coordinates multiplies the first nonzero normal jet by a nonzero
factor. Terms differentiating the factor or the moving longitudinal
coordinate vanish because the separation and its first derivative
are zero at contact. Reversibility and $R(O)=O$ give the unstable
arc $v=-h(u,a)$. Their difference is $2h$, proving both
quadraticity and nondegenerate unfolding. The point is homoclinic
because it belongs to the stable and unstable manifolds of the same
saddle. If $R$ exchanges two different saddles, this last conclusion
must instead be stated as a heteroclinic one.
\end{proof}

\begin{Remark}[Curvature and orientation]
For a graph $y=y(x)$ with $y'\ne0$, direct differentiation gives
\[
 \frac{dx}{dy}=\frac1{y'},\qquad
 \frac{d^2x}{dy^2}=-\frac{y''}{(y')^3}.
\]
The signed curvature of an oriented parametrized curve is
$(x'y''-y'x'')/((x')^2+(y')^2)^{3/2}$. Thus changing to a
y-parametrization requires tracking the orientation; one cannot
reuse the same signed-curvature formula with $x$ and $y$ merely
interchanged. At $y'=1$, the second derivatives of the inverse
graphs have opposite signs.

In the outgoing cone the valid local estimate in this manuscript is
$8\le d^g\le10$, together with $-40\tau\le B^g\le-5\tau$;
see (\ref{eq:outgoing-curvature-coeff}). These bounds alone do not
establish a contact. The first-contact argument above proves its
existence using the stated graph bounds.
The Sitnikov proof verifies these hypotheses by choosing its tail
interval sufficiently close to infinity and reducing the relative
derivative constant if necessary. For another model the phase
lobe, derivative bounds, and parameter derivative must be verified
before this criterion can be applied.
\end{Remark}

\subsection{RPC3BP: the parameter input and the two reductions}
\label{ss:rpc-parameter-status}

We separate the parameter-dependent global input from the dimension
argument. For a fixed Jacobi constant $C$, consider the following
hypothesis, denoted $(H_C)$: there is a nonempty open set
$\mathcal N_C$ of mass ratios with a dense subset of parameters at
which $f_{\mu,C}$ has a quadratic homoclinic tangency to a hyperbolic
periodic saddle, unfolding nondegenerately with $\mu$, and that
saddle has transverse connections in both directions to $O_\infty$.
All connections and finite orbit segments are in the collision-free
physical domain of the return map.

\begin{Proposition}[From the global input to full dimension]
\label{prop:rpc-conditional-completion}
Hypothesis $(H_C)$ implies the dimension conclusion of
Theorem~\ref{RPC3BP-main} and the fixed-$C$ Newhouse assertion in
Theorem~\ref{parameter-Newhouse-RPC3BP}.
\end{Proposition}
\begin{proof}
For each positive integer $k$, let $V_k$ consist of parameters in
$\mathcal N_C$ admitting a hyperbolic basic set of dimension
$>2-1/k$ with transverse connections to infinity in both directions.
The set $V_k$ is open: a fixed basic set has a hyperbolic continuation,
its dimension is continuous, and the finitely many selected transverse
connections persist. To prove density, take an arbitrary open interval
in $\mathcal N_C$ and a generic tangency parameter in it, supplied by
$(H_C)$. The conservative unfolding construction in
Theorem~\ref{t.thick-sets} produces, arbitrarily close to this
parameter, basic sets containing the continued saddle and having
dimension $>2-1/k$. Work in a sufficiently small parameter interval
so that its transverse connections to infinity persist. Thus $V_k$
meets the given interval.

For every $\mu\in\bigcap_kV_k$, Theorem~\ref{2-diml-hd} gives
oscillatory subsets of the section with dimensions approaching two,
including recurrence to a finite periodic orbit. Their union has
dimension two. Each such set lies in a compact rectangle in the
smooth collision-free section. A sufficiently short flow box over
that rectangle is smoothly diffeomorphic, with bounded derivatives
and bounded inverse derivatives, to its product with a time interval.
By \eqref{eq:interval-product-dimension}, its product with an
interval has dimension one greater;
hence the oscillatory set on the energy surface has dimension three.

There is no finite-time accumulation of the prescribed returns.
On a compact transverse rectangle the smooth flow has a uniform
positive time during which a local flow box cannot meet the section
again. Each visit therefore consumes a positive minimum physical
time. Infinitely many visits force physical time to tend to infinity
in each direction. The excursions and finite returns consequently
have the same oscillatory interpretation for this variable-return-time
map as for the flow. Finally, the dense tangencies in $(H_C)$ are
exactly the Newhouse property used in the stated definition.
\end{proof}

\paragraph{Available global result.}
Garrido, Mart\'in and Paradela~\cite[Theorem 1.1 and
Sections 2.3--2.6]{GMP2025} establish mass-parameter Newhouse domains
associated to infinity, using secondary tangencies and a renormalized
return map. This supplies an independent route to the global input
in their energy convention. Their argument establishing the domains
and the homoclinic connections does not use the oscillatory dimension
transfer in this manuscript; their later discussion of that transfer
must not be used as an independent proof of it.

Applying that theorem to the fixed-energy conclusion requires
reconciling the energy normalizations. Our constant is
$C=-2H_{\rm rot}$, and the parabolic branch used in
\cite{GMS2016} has $H_{\rm rot}=-G_0$, hence $C=2G_0>0$.
In \cite[Section 1.1]{GMP2025} the displayed conserved quantity is
$\mathcal J=H_{\rm inertial}-G=H_{\rm rot}$, followed by a
large-$J$ convention. With the displayed definitions this gives
$J=-C/2$, not $C/2$. We do not silently identify a positive large
value of one signed constant with a positive large value of the other.
One must check the branch and sign convention in that global theorem,
or give a direct construction on the $C>0$ branch. The latter is
provided in Section~\ref{ss:fixed-energy-completion} below. This
direct construction avoids identifying the two signed constants.

\paragraph{The fixed-mass assertion.}
For the fixed-mass conclusion, fix $\mu$ and vary $C$.
Theorem 2.2 of \cite{GMS2016} supplies transverse homoclinic
intersections for all sufficiently large $C$. It supplies neither
a sequence of quadratic tangencies unfolding in $C$ nor a
Newhouse domain on that parameter line. The mass-parameter theorem
in \cite{GMP2025} does not change this quantifier order.
In particular, a nonzero derivative of a splitting with respect to
$\mu$ says nothing by itself about its derivative with respect to $C$.

The sufficient additional input is a sequence $C_n\to\infty$ of
quadratic tangencies to hyperbolic saddles homoclinically related to
infinity, with nonzero unfolding derivative in $C$. The conservative
unfolding theorem and the preceding proof would then yield the
claimed Newhouse domains and residual full dimension. An alternative
is to construct secondary parabolic tangencies and prove their
parameter-uniform transfer to periodic saddles. Either route needs
quantitative spatial and $C$-derivative control at the relevant
exponentially small scale. The equal-mass case must use its second
harmonic. These estimates and the tangency construction are supplied
in Section~\ref{ss:rpc-uniform-energy}, following the reductions below.

\subsection{A criterion for unfolding in the Jacobi constant}
\label{ss:energy-contact-criterion}

We now give a more precise reduction of the fixed-mass assertion.
The result below is an abstract contact criterion, with a complete
proof. Its hypotheses on the RPC3BP passage maps are not assumed to
follow from compact-parameter estimates. 

The idea is to compare the exponential decrease of the splitting
amplitude with the algebraic dependence of a long local passage on its
duration. We keep the leading exponential in an explicit amplitude;
we never differentiate an unspecified asymptotic remainder.

\begin{Proposition}[Contact along an unbounded energy sequence]
\label{prop:energy-contact}
Fix $k>0$, $a_0>0$, and $p\in\mathbb R$, and set
\[
 A(C)=a_0 C^p\exp(-kC^3/24),\qquad
 I=[-\pi/6,7\pi/6].
\]
Choose $0<a<b$ such that $a^3<16/k<b^3$, and put
\[
 J_N=[a(\log N)^{1/3},b(\log N)^{1/3}].
\]
Suppose that, for every sufficiently large integer $N$, the difference
of two graphs, divided by a positive smooth factor, is
\[
 F_N(\theta,C)=A(C)(\sin\theta+\varepsilon(\theta,C))
                 -H_N(\theta,C),\qquad (\theta,C)\in I\times J_N.
\]
Assume that these functions are jointly $C^2$ in $(\theta,C)$
on neighborhoods of their domains. Suppose that there are constants $B\ge1$, $q\ge0$, independent
of $N$, such that
\begin{equation}\label{eq:energy-height-size}
 B^{-1}C^{-q}N^{-2/3}\le H_N(\theta,C)
                  \le BC^qN^{-2/3}.
\end{equation}
Assume also that, uniformly on these domains as $N\to\infty$,
\begin{align}
 \|\varepsilon(\cdot,C)\|_{C^2(I)}&=o(1),
 &|\partial_C\varepsilon(\theta,C)|&=o(C^2),
 \label{eq:energy-splitting-jets}\\
 |\partial_\theta H_N|+|\partial_\theta^2 H_N|
   &\le \tfrac1{100}H_N,
 &|\partial_C\log H_N|&=o(C^2).
 \label{eq:energy-height-jets}
\end{align}
Then there are $C_N\in\operatorname{int}J_N$ and
$\theta_N\in\operatorname{int}I$ such that
\[
 F_N=\partial_\theta F_N=0,\qquad
 \partial_\theta^2F_N<0,\qquad \partial_CF_N<0
\]
at $(\theta_N,C_N)$. Thus the graphs have a quadratic contact which
unfolds nondegenerately in $C$. Moreover,
\begin{equation}\label{eq:energy-contact-location}
 C_N^3=\frac{16}{k}\log N+O(\log\log N).
\end{equation}
\end{Proposition}

\begin{proof}
For sufficiently large $N$, the $C^2$ norm of $\varepsilon$ is at most
$1/100$. At the endpoints of $I$, the function $F_N$ is negative.
At the left endpoint $C=a(\log N)^{1/3}$ of $J_N$, the height bound gives
\[
 \frac{A(C)}{H_N(\pi/2,C)}
 \ge \frac{a_0}{B}C^{p-q}N^{2/3-ka^3/24}\longrightarrow\infty.
\]
Consequently $F_N(\pi/2,C)>0$ there. At the right endpoint
$C=b(\log N)^{1/3}$, uniformly in $\theta$,
\[
 \frac{A(C)}{H_N(\theta,C)}
 \le a_0BC^{p+q}N^{2/3-kb^3/24}\longrightarrow0.
\]
Therefore $F_N<0$ throughout $I$ at that endpoint.
The continuous function $C\mapsto\max_I F_N(\theta,C)$ has an
interior zero $C_N$. Any maximizing point $\theta_N$ is interior and
satisfies $F_N=\partial_\theta F_N=0$.

At this point $H_N/A=\sin\theta_N+\varepsilon\le1.01$. Hence
\[
 |\cos\theta_N|
 \le |\varepsilon_\theta|+|H_{N,\theta}|/A
 \le .01+.0101<.021.
\]
The only zero of the cosine in $I$ is $\pi/2$. It follows that
$\sin\theta_N>.99$ and $.98<H_N/A<1.01$. In particular,
\[
 \frac{F_{N,\theta\theta}}{A}
 =-\sin\theta_N+\varepsilon_{\theta\theta}
                       -H_{N,\theta\theta}/A
 <-.99+.01+.0101<-.96.
\]
This proves quadraticity. At the contact, differentiating in $C$
with the phase held fixed gives
\begin{align*}
 \frac{F_{N,C}}{H_N}
 &=\frac{A'}{A}+\frac{A}{H_N}\varepsilon_C
                          -\partial_C\log H_N\\
 &=-\frac{kC_N^2}{8}+\frac{p}{C_N}+o(C_N^2)<0.
\end{align*}
The derivative of the critical value is this same derivative since
$F_{N,\theta}=0$. Multiplication by a positive smooth factor does not
change the signs at a contact, nor does a smooth parameter-dependent
change of the phase destroy nondegeneracy: the extra first-derivative
term vanishes there.

Finally, the height bounds and $.98<H_N/A<1.01$ imply
\[
 -\frac{kC_N^3}{24}+p\log C_N+\log a_0
 =-\frac23\log N+O(\log C_N).
\]
Since $C_N\in J_N$, this is exactly
\eqref{eq:energy-contact-location}.
\end{proof}

\begin{Remark}
The exponent $2/3$ is an explicit hypothesis about the chosen graph
height. It must be checked after division by the spatial amplitude
and after the change of phase. If a different normalization gives
$N^{-\beta}$, $\beta>0$, the same proof replaces $16/k$ by
$24\beta/k$. Polynomial factors in $C$ affect only the
$O(\log\log N)$ term.
\end{Remark}

\paragraph{What the Fourier estimates provide.}
In the straightened coordinates of \cite[Section 6]{GMS2016}, the
difference of generating functions has the form
\[
 \Upsilon(G_0^3w+\xi),\qquad
 \Upsilon(\theta)=\Upsilon^{[0]}+
                   2\sum_{\ell\ge1}\Upsilon^{[\ell]}\cos(\ell\theta).
\]
The coefficient estimates in the proof of Lemma 6.8, together with
Proposition 3.1 of that paper, yield a useful strengthening in the
phase variable. For each fixed $\mu\in(0,1/2)$, the derivative
$\Upsilon'$ divided by its explicit leading first-harmonic amplitude
has, up to orientation, the form $\sin\theta+O(G_0^{-1/2})$ in every
fixed $C^j$ phase norm. For $\mu=1/2$, all odd harmonics vanish; using
the second harmonic and its phase gives $\sin\theta+O(G_0^{-1})$ in
every fixed $C^j$ norm. Constants here may depend on the fixed mass and
on $j$.

Indeed, the first-harmonic error is $O(G_0^{-2}e^{-G_0^3/3})$,
whereas its leading coefficient has size
$G_0^{-3/2}e^{-G_0^3/3}$. At equal masses the corresponding sizes are
$G_0^{-1/2}e^{-2G_0^3/3}$ and
$G_0^{1/2}e^{-2G_0^3/3}$. Every fixed number of phase derivatives
multiplies the remaining Fourier coefficients by a fixed power of
$\ell$; the exponential decay in $\ell G_0^3$ still dominates this
factor. This proves the assertions for the straightened potential.
Passing to the actual section graphs also requires differentiating
the coordinate change and the graph prefactor. We do not identify
these two objects without that step.

Since $C=2G_0$, the leading exponent is $-C^3/24$ for unequal masses
and $-C^3/12$ for equal masses. Thus $k=1$ and $k=2$, respectively,
in Proposition~\ref{prop:energy-contact}, provided the graph
normalization introduces only polynomial factors into the leading
amplitude.

\paragraph{Parameter control for the normalized potential.}
Section~\ref{ss:complex-splitting-control} constructs the splitting
potential holomorphically on fixed domains for $|z-G|<cG^{-2}$.
A uniform bound on this complex disk, combined with the small
remainder on its real diameter, yields the required parameter
derivative by interpolation. The transport inverse is controlled
before complexifying its oscillatory integral representation.

\paragraph{Uniformity of the local passage.}
The required range is not a fixed compact interval of energies:
it is $C\asymp(\log N)^{1/3}$. A statement valid for
$N\ge N_0(C)$, without a bound on $N_0(C)$, does not suffice.
For example, the threshold $N_0(C)=\exp(C^6)$ would exceed $N$
throughout the proposed range for large $N$.

A sufficient local statement would give the height bounds
\eqref{eq:energy-height-size}--\eqref{eq:energy-height-jets} for
$N\ge BC^s$, with constants independent of $C$. Such a threshold is
compatible with the entire interval $J_N$ for large $N$. One must
also control the $C$ dependence of the selected phase interval and
of the factor dividing the physical graph height. A polynomial
bound on the height alone does not control its logarithmic derivative.

\paragraph{Application to the RPC3BP.}
To apply Proposition~\ref{prop:energy-contact}, we verify its
estimates for the invariant-manifold graph and the iterated symmetry
graph of the RPC3BP. Section~\ref{ss:complex-splitting-control}
establishes the complex-parameter estimates for the splitting
potential on fixed domains. Section~\ref{ss:rpc-uniform-energy}
converts these estimates to section graphs and proves the required
coordinate and local passage bounds on the joint $(C,N)$ range.
Proposition~\ref{prop:energy-contact} then gives the contacts,
their quadraticity, their nondegenerate unfolding, and their
parameter location. The reversibility argument in
Proposition~\ref{prop:general-reversible-contact} turns these
contacts into parabolic homoclinic tangencies; this application is
completed in Proposition~\ref{prop:rpc-fixed-mass-tangencies}.

After each tangency has been constructed, its transfer to a periodic
saddle uses the parameter estimates of
Section~\ref{ss.saddlesnearhomoclinicpoint} on a neighborhood of
that particular finite $C_N$. The neighborhood and its constants
may depend on $N$. Thus this transfer does not require estimates
uniform as $C_N\to\infty$.

\subsection{Complex parameter control for the straightened splitting}
\label{ss:complex-splitting-control}

The complex extension can be treated without moving the singularities
of the unperturbed separatrix. We record the operator argument first.
Throughout this subsection the mass $\mu$ is fixed, and $G=G_0=C/2$
denotes the real angular momentum at parabolic infinity. The variable
$z$ denotes its complex continuation. Constants are not
claimed to be uniform as $\mu\to0$ or as an unequal mass tends to $1/2$.

\begin{Remark}[Inputs from the real-parameter construction]
\label{rem:gms-inputs}
We specify the real-parameter results of \cite{GMS2016} used below
and the additional arguments needed for their complex-parameter transfer.
All norms and domains are fixed at the real center $G$; the finite
domain margins are chosen before the disk radius constant $c$.
Endpoint values are inherited from the preceding construction.
Uniformity is in $G$ for each fixed mass, with a separate equal-mass
normalization. The source numbers in this remark refer to \cite{GMS2016}.

The transport inverses of Lemmas 5.6 and 5.11 and formula (66)
are transferred using Lemma~\ref{lem:transport-resolvent}, fixed Fourier
endpoint data, and the shear identities \eqref{eq:angular-shear}.
For the domain (46) and the weighted spaces of Sections 5.3--5.4,
the additional ingredients are the strip estimate
\eqref{eq:shear-strip-bound}, the embeddings
\eqref{eq:transfer-weight-embeddings}, and the integrable-weight
inverse \eqref{eq:integrable-weight-inverse}.
The Hamilton--Jacobi constructions of Propositions 5.7 and 5.12
are extended by uniform contractions at infinity and on the bounded
domain, giving \eqref{eq:complex-hj-bounds}.

The coordinate changes of Lemmas 5.14--5.15 and 5.17 and Theorem 5.16
use the conjugation \eqref{eq:shear-composition}, a ball of size
$K\mu G^{-4}$, and a contraction constant $K\mu G^{-1}$.
For the flow extension and return in Propositions 5.20--5.21,
formula (84), and Theorem 5.22, we use the unchanged fundamental
matrix and a normalized inverse-function argument with inherited
homogeneous data. The transport straightening in formulas (86), (90)
and Lemmas 6.4 and 6.6 is extended with an order-zero bound
$K\mu G^{-4}$, preserving the symmetric normalization and equal-mass
periodicity. Finally, the exponential estimates of Proposition 3.1
and Lemmas 6.7--6.8 are transferred by vertex evaluation and the
order-zero error bound $K\mu^2G^{-7/2}$, leading to
Proposition~\ref{prop:normalized-potential-parameter}.

The real-parameter results alone do not give a complex-parameter
theorem. The operator and fixed-point arguments below supply that
extension on $|z-G|<cG^{-2}$, and the final interpolation yields its
parameter derivative. The subsequent coordinate matching is proved
separately in Section~\ref{ss:rpc-uniform-energy}.
\end{Remark}

\begin{Lemma}[Perturbing the transport frequency]
\label{lem:transport-resolvent}
Let $G>0$ be real, let $L_G=\partial_v-G^3\partial_\xi$, and suppose
that a right inverse $K_G:Y_G\to X_G$ satisfies
\[
 \|K_Gh\|_{X_G}\le M\|h\|_{Y_G},\qquad
 \|\partial_\xi u\|_{Y_G}\le G^{-3}\|u\|_{X_G}.
\]
The spaces and their boundary normalization are fixed at $G$.
For $d=z^3-G^3$ with $M|d|G^{-3}<1$, put
\[
 K_z=(\operatorname{Id}-dK_G\partial_\xi)^{-1}K_G.
\]
Then $L_zK_z=\operatorname{Id}$, $K_z$ is holomorphic in $z$, and
\[
 \|K_z\|_{Y_G\to X_G}\le\frac{M}{1-M|d|G^{-3}}.
\]
In particular the bound is uniform for $|z-G|<cG^{-2}$ and large $G$.
\end{Lemma}
\begin{proof}
The norm of $dK_G\partial_\xi$ on $X_G$ is at most $M|d|G^{-3}$.
The Neumann series therefore converges in operator norm and is
holomorphic in $z$. If $u=K_zh$, then
$u=K_Gh+dK_G\partial_\xi u$. Applying $L_G$ gives
$L_Gu=h+d\partial_\xi u$, as required. Finally,
$|z-G|<cG^{-2}$ implies $|d|\le4c$ for large $G$.
\end{proof}

We write $\kappa_{\rm dom}$ for the domain parameter denoted by
$\kappa$ in \cite{GMS2016}; it is unrelated to the second-jet
coordinate used in the local passage estimates.
In the notation of \cite{GMS2016}, Lemma 5.6 supplies this pair of
estimates with $Y_G=Z_{8/3,\rho,\sigma}$ and
$X_G=\widetilde Z_{5/3,\rho,\sigma}$. Lemma 5.11 supplies it with
$Y_G=Y_{2,\rho,\kappa_{\rm dom},\delta,\sigma}$ and
$X_G=\widetilde Y_{1,\rho,\kappa_{\rm dom},\delta,\sigma}$, for the inverse
with zero data at the specified Fourier-dependent endpoints. Indeed,
the definition of each strengthened norm includes
$G^3\|\partial_\xi u\|_{\nu+1,\sigma}$. Thus the parameter
perturbation is small on the same spaces, with no loss of a Fourier
strip. This proves the needed inverse estimate both at infinity and
for these bounded-domain inverses.

For clarity, nonzero prescribed endpoint data must be included as
a separate homogeneous solution. They are not annihilated by an
arbitrary right inverse. At infinity the decaying normalization used
in the fixed-point construction is retained by the above series.
For a Fourier coefficient with prescribed value $h_\ell(z)$ at
a fixed endpoint $v_\ell$, the homogeneous term is
\[
 e^{i\ell z^3(v-v_\ell)}h_\ell(z).
\]
Its quotient by the corresponding exponential at the real center
is $e^{i\ell(z^3-G^3)(v-v_\ell)}$. On a bounded domain this
costs at most $e^{Kc|\ell|}$, absorbed by a fixed Fourier-strip
margin. The data $h_\ell(z)$ are obtained by evaluation of the
preceding holomorphic stage. Thus retaining the data preserves
both holomorphy and the stated weighted bounds; a zero-data inverse
alone would not establish analytic continuation of the chosen
manifold.

\paragraph{Fixed spaces on the bounded domains.}
We give the transfer argument for the remaining stages, including
composition and the nonzero endpoint data. All domains and norms in
this argument are fixed at the real center $G$. Choose the finite
sequence of strictly nested domains used in Sections 5.4--6.2 of
\cite{GMS2016}, with strict inequalities in all the domain and Fourier
margins. Their union is contained in $|v|\le R$, with $R$ independent
of $G$. In the following formulas $\|\cdot\|_{\nu,\sigma}$ denotes
the weighted Fourier norm of \cite{GMS2016} on the indicated domain:
\[
 \|h\|_{\nu,\sigma}=
 \sum_{\ell\in\mathbb Z}e^{|\ell|\sigma}
 \sup_v\left|
 |v-i/3|^{\nu-\ell/2}|v+i/3|^{\nu+\ell/2}
 h^{[\ell]}(v)\right|.
\]
Here and below the Fourier series may be formal. All identities
are identities of their analytic Fourier coefficients.

Put $d=z^3-G^3$ and define
\[
 H_z(v,\xi)=(v,\xi-dv),\qquad S_zh=h\circ H_z.
\]
Direct differentiation gives
\begin{equation}\label{eq:angular-shear}
 L_GS_z=S_zL_z,\qquad
 S_z\partial_v=(\partial_v+d\partial_\xi)S_z,\qquad
 S_z\partial_\xi=\partial_\xi S_z.
\end{equation}
Moreover $S_z(hk)=(S_zh)(S_zk)$, coefficientwise. Thus the
weighted product estimate is unchanged. If $\sigma'<\sigma$ and
$|d|R\le\sigma-\sigma'$, then
\begin{equation}\label{eq:shear-strip-bound}
 \|S_z^{\pm1}h\|_{\nu,\sigma'}\le\|h\|_{\nu,\sigma}.
\end{equation}
Indeed the multiplier of the $\ell$th coefficient is
$e^{\mp i\ell dv}$, of modulus at most $e^{|\ell||d|R}$.
The singularity weights are not changed.

We solve each bounded-domain equation for $S_zh$ in a fixed
Banach space. Known coefficients are first transferred from a
larger strip by \eqref{eq:shear-strip-bound}; the solution is
transferred back only after the fixed point has been obtained.
Choose $c$ so that $4cR$ is smaller than each reserved shear
margin. There are only finitely many such margins, all independent
of $G$. In particular, \eqref{eq:shear-strip-bound} is not applied
anew at each fixed-point iteration.

For the strengthened norms put
\[
 \lVert\!\lvert u\rvert\!\rVert_{\nu,\sigma}
 =\|u\|_{\nu,\sigma}+\|\partial_vu\|_{\nu+1,\sigma}
       +G^3\|\partial_\xi u\|_{\nu+1,\sigma}.
\]
Replacing $\partial_v$ by $\partial_v+d\partial_\xi$ gives an
equivalent norm, in both directions, with constant at most
$1+|d|G^{-3}$. Consequently the derivative estimates in the
Hamilton--Jacobi construction hold in the fixed transformed spaces.
The weighted embeddings used below are
\begin{equation}\label{eq:transfer-weight-embeddings}
 \|h\|_{\nu+\eta,\sigma}\le K\|h\|_{\nu,\sigma},\qquad
 \|h\|_{\nu-\eta,\sigma}\le KG^{3\eta}\|h\|_{\nu,\sigma},
 \qquad \eta\ge0.
\end{equation}
These follow from boundedness of the domain and its distance
$\kappa_{\rm dom} G^{-3}$ from the singularities.

\paragraph{Inverses and endpoint data.}
Let $K_G$ be one of the bounded-domain Fourier inverses in (66) or (90) of \cite{GMS2016}, with all endpoints fixed at $G$. In the transformed
equation the inverse is exactly $K_G$, not an operator between
successively narrower strips. In the original variables it is
\[
 K_z=S_z^{-1}K_GS_z.
\]
For example, for a nonzero harmonic with endpoint $v_\ell$,
\[
 (K_zh)^{[\ell]}(v)=
 \int_{v_\ell}^v e^{i\ell z^3(v-t)}h^{[\ell]}(t)\,dt.
\]
This follows by cancelling the factors $e^{\pm i\ell dv}$ inside
the integral. The zero harmonic is unaffected by $S_z$, including
the symmetric normalization in (90) of \cite{GMS2016}. The integrable-weight estimate needed for Theorem 5.16 of \cite{GMS2016} is
proved below. The symmetric inverse on the boomerang domain is
covered by Lemma 6.4 of \cite{GMS2016}. Conjugation transfers these established
bounds to the transformed spaces with the same constants.

Prescribed data $h_\ell(z)$ become
$e^{-i\ell dv_\ell}h_\ell(z)$. Their homogeneous extension is
therefore
\[
 e^{i\ell G^3(v-v_\ell)}e^{-i\ell dv_\ell}h_\ell(z)
\]
in the transformed equation. The endpoint multiplier costs only
a reserved fixed strip margin. The real-center homogeneous
extension estimates now apply verbatim. At each extension stage
these are the values of the preceding solution, not newly chosen
zero data. Uniqueness on the overlap therefore identifies the
extension with the same invariant manifold. This also fixes its
parameter dependence. At infinity we use
Lemma~\ref{lem:transport-resolvent}, rather than $S_z$, since the
half-plane there is unbounded.

\paragraph{The integrable-weight inverse on the extension domain.}
We justify separately the estimate needed for the change to flow
coordinates. On the unstable extension domain $D^u_{\rho,\kappa_{\rm dom},\delta}$
of (46) of \cite{GMS2016}, with fixed sufficiently large $\kappa_{\rm dom}$, the inverse
in (66) of \cite{GMS2016} satisfies
\begin{equation}\label{eq:integrable-weight-inverse}
 \|K_Gh\|_{0,\sigma}\le K_\nu\|h\|_{\nu,\sigma},
 \qquad 0<\nu<1.
\end{equation}
The constant is independent of $G$. This assertion is distinct
from the weight-lowering estimate for $\nu>1$ in Lemma 5.11 of \cite{GMS2016}.

Here is a contour proof. Write $a_i=\tan\beta_i>0$ and
$h_0=1/6-\delta$, so that the domain is described by
\[
 -\rho<\Re v,\qquad
 |\Im v|<-a_1\Re v+1/3-\kappa_{\rm dom} G^{-3},\qquad
 |\Im v|>a_2\Re v+h_0.
\]
Choose $\rho>h_0/a_2$ and fix
$0<a_*<\min(a_1,a_2)$. For a positive Fourier index, join the
bottom left vertex to $v$ by a vertical segment on $\Re t=-\rho$
and a segment of slope $a_*$ ending at $v$. The intermediate
vertex is
\[
 t_*=-\rho+i\bigl(\Im v-a_*(\Re v+\rho)\bigr).
\]
Both segments lie in the closure of the domain. For the outer
boundaries this follows from $a_*<a_1$; for the removed central
wedge it follows from $a_*<a_2$, checking the upper and lower
components separately. The vertical segment lies to the left of
that wedge. Boundary contours are interpreted as limits of interior contours.
The coefficientwise primitives have the required endpoint limits
by the integrable bounds below; for the coefficients occurring in
the construction they also agree with the analytic extensions from
the preceding stage. Along this path, oriented toward $v$,
\[
 \Im(v-t)\ge c_*\operatorname{length}(\gamma_{t,v}),
 \qquad c_*=\frac{a_*}{\sqrt{1+a_*^2}}>0.
\]
For negative Fourier indices use the conjugate path from the top
left vertex, with decreasing imaginary part.

Put
\[
 w_\ell(t)=|t-i/3|^{-\ell/2}|t+i/3|^{\ell/2}.
\]
The distance from the domain to either singularity is at least
$c_1\kappa_{\rm dom} G^{-3}$, where $c_1>0$ depends only on the fixed
boundary angles. Consequently, along either path,
\[
 |d\log w_\ell(t)|\le
       K|\ell|G^3\kappa_{\rm dom}^{-1}|dt|.
\]
For $\kappa_{\rm dom}$ sufficiently large, the exponential in the inverse
therefore dominates the change of Fourier weight:
\[
 \left|e^{i\ell G^3(v-t)}\right|
       \frac{w_\ell(v)}{w_\ell(t)}\le1.
\]
The paths have bounded length and at most two straight segments.
Since the singularities are separated and $0<\nu<1$, integration
along any such segment gives
\[
 \int_\gamma |t-i/3|^{-\nu}|t+i/3|^{-\nu}|dt|
       \le K_\nu.
\]
Indeed, near one singularity the other factor is bounded, and
arclength along a straight segment reduces the estimate to the
integrability of $|s|^{-\nu}$ on a bounded interval. Away from
both singularities the integrand is bounded. The same estimate
applies to the zero harmonic, using a polygonal path from
$-\rho$ to $v$ inside the domain; its exponential and Fourier
weight are both one. Multiplying the coefficientwise integral
estimate by $e^{|\ell|\sigma}$ and summing proves
\eqref{eq:integrable-weight-inverse}. Reflection proves the stable
version. In particular, $\nu=1/2$ gives exactly the inverse bound
used in the coordinate-change contraction below.

\paragraph{The Hamilton--Jacobi estimates.}
For $|z-G|<cG^{-2}$ we have $|z|\asymp G$ and
$|z^3-G^3|\le4c$. The expansion of the perturbing potential is
holomorphic on this disk. Indeed its expansion parameter on the
bounded singular domains satisfies
$|z^{-2}r_h^{-1}|\le KG^{-1/2}$; on the domain at infinity it is
smaller still. Thus the absolutely convergent potential expansion
and its differentiated versions have the same bounds as at real
$G$. Its leading part $U_0$ has only harmonics $0,\pm2$ and size
$O(\mu G^{-4})$; its remainder $U_1$ has size $O(\mu G^{-6})$
in the respective weighted norms. The shear changes neither the
harmonic support nor the powers of $G$.

At infinity, Proposition 5.7 of \cite{GMS2016}, with the inverse supplied by
Lemma~\ref{lem:transport-resolvent}, gives
\[
 \lVert\!\lvert Q\rvert\!\rVert_{5/3}\le K\mu G^{-6},\qquad
 \operatorname{Lip}\mathcal F_\infty\le K\mu G^{-4}.
\]
The proof is the same quadratic product estimate, multiplied by
the inverse bound $(1-M|d|G^{-3})^{-1}$. The leading linear
correction $Q_0$ and the linear remainder $L_1$ are constructed
with the same inverse and the same decaying normalization.

Here are the powers in the bounded-domain nonlinear estimate.
Write $u=S_zQ$, $u_0=S_zQ_0$, and
\[
 D_z=\partial_v+(d-r_h^{-2})\partial_\xi.
\]
The quadratic term, up to its irrelevant overall sign, is
\[
 N_z(w)=\frac{(D_zw)^2}{2y_h^2}
                    +\frac{(\partial_\xi w)^2}{2r_h^2}.
\]
The transferred linear estimates give
\[
 \lVert\!\lvert u_0\rvert\!\rVert_{1/2}\le K\mu G^{-4},\qquad
 \|S_z(U_1\circ A)\|_2\le K\mu G^{-6},
\]
where $A(v,\xi)=(r_h(v),\xi+\alpha_h(v))$ denotes the
substitution in the potential. Since $u_0$ has only three
harmonics, also $\|\partial_\xi u_0\|_{1/2}\le K\mu G^{-4}$.
Near either singularity $r_h^{-2}$ has weight $1$ and
$y_h^{-2}$ has weight $-1$. These bounds hold on the domains
of this stage, which avoid the zero of $y_h$. Thus
\[
 \|D_zu_0\|_{3/2}\le K\mu G^{-4},\qquad
 \|D_zu\|_2\le K\lVert\!\lvert u\rvert\!\rVert_1,\qquad
 \|\partial_\xi u\|_2\le G^{-3}\lVert\!\lvert u\rvert\!\rVert_1.
\]
For the second inequality the term $r_h^{-2}\partial_\xi u$
first has weight $3$ and size $G^{-3}\lVert\!\lvert u\rvert\!\rVert_1$;
\eqref{eq:transfer-weight-embeddings} cancels this $G^{-3}$.
Factoring differences of squares now gives, on a ball of radius
$K\mu G^{-6}$,
\begin{align*}
 \|N_z(u_0)\|_2&\le K\mu^2G^{-8},\\
 \|N_z(u_0+u)-N_z(u_0+\widetilde u)\|_2
 &\le K\mu G^{-5/2}\lVert\!\lvert u-\widetilde u\rvert\!\rVert_1.
\end{align*}
For example, the product containing $D_zu_0$ has weight
$3/2+2-1=5/2$; lowering it to $2$ costs $G^{3/2}$ and gives
$\mu G^{-4}G^{3/2}=\mu G^{-5/2}$. The product containing
$D_zu$ instead costs $G^3$ and gives $\mu G^{-3}$.
The angular products give the same or smaller bounds, using the
factor $G^{-3}$ just displayed. The additional $d\partial_\xi$
term is bounded by $|d|G^{-3}$ in the strengthened norm.
Applying the real-center inverse and retaining the endpoint data
therefore proves the complex version of Proposition 5.12 of \cite{GMS2016}:
\begin{equation}\label{eq:complex-hj-bounds}
 \lVert\!\lvert u\rvert\!\rVert_1\le K\mu G^{-6},\qquad
 \operatorname{Lip}\mathcal F_{\rm HJ}\le K\mu G^{-5/2},\qquad
 \|S_z(Q-L_1)\|_1\le K\mu^2G^{-8}.
\end{equation}
For the last estimate, subtract the linear equation for $L_1$.
The endpoint difference is $O(\mu^2G^{-8})$ by the preceding
stage, the forcing $N_z(u_0)$ has this same bound, and the
remaining nonlinear difference is $O(\mu^2G^{-17/2})$.

\paragraph{Composition and the flow coordinates.}
The exact composition identity is
\begin{equation}\label{eq:shear-composition}
 S_z\bigl(h\circ(\operatorname{Id}+g)\bigr)
 = (S_zh)\circ(\operatorname{Id}+\widehat g),\qquad
 \widehat g=(S_zg_1,S_zg_2+dS_zg_1).
\end{equation}
It is obtained by computing
$H_z^{-1}\circ(\operatorname{Id}+g)\circ H_z$.
In the vector norm of \cite{GMS2016}
$\|g\|_* =\|g_1\|_0+G^{-3}\|g_2\|_0$, the triangular change
$(a,b)\mapsto(a,b+da)$ and its inverse have norm at most
$1+|d|G^{-3}$. In the ordinary vector norm their norms are at
most $1+|d|$. Both bounds are uniform.

The composition estimates of Lemmas 5.14--5.15 of \cite{GMS2016} consequently
apply on the fixed reduced domains. In particular, for a known
coefficient $F$ on the larger domain and
$\|g\|_*,\|\widetilde g\|_*\le\eta G^{-3}$, they give
\[
 \|F\circ(\operatorname{Id}+g)
       -F\circ(\operatorname{Id}+\widetilde g)\|_\nu
 \le KG^3\|F\|_\nu\|g-\widetilde g\|_*.
\]
The domain margins, including the logarithmic Fourier-weight
margin in those lemmas, are fixed before choosing $c$.
The unknown changes have size $O(\mu G^{-4})=o(G^{-3})$,
so all the composition conditions hold for large $G$.
For the change from Hamilton--Jacobi to flow coordinates, the
forcing has weight $1/2$ and norm $K\mu G^{-4}$ (Lemma 5.17 of \cite{GMS2016}). Equations \eqref{eq:shear-composition} and the inverse
estimate used in Theorem 5.16 of \cite{GMS2016} give a ball of radius
$K\mu G^{-4}$ and contraction constant $K\mu G^{-1}$.

For the subsequent flow extension write
$\Gamma=\Gamma_0+\Gamma_1$. The transformed base is
\[
 S_z\Gamma_0=(r_h,y_h,\xi-dv+\alpha_h,1).
\]
The frequency occurs in the unperturbed vector field only as the
additive angular velocity $-z^3$. Its derivative along the base
is the same matrix $DX_0(\Gamma_0)$ as at $G$: the unperturbed
field is independent of the angle. The fundamental matrix
$\Psi(v)$ and its inverse are therefore unchanged. On the flow
extension domain they are uniformly bounded, and the inverse is
\[
 h\longmapsto\Psi K_G(\Psi^{-1}h).
\]
The homogeneous data are transformed exactly as above, with the
additional fixed factors $\Psi(v)\Psi^{-1}(v_\ell)$.
On this domain, which stays a fixed distance from the separatrix
singularities, the nonlinear forcing satisfies
\begin{align*}
 \|\mathcal R_z(w)\|&\le K(\mu G^{-4}+\|w\|^2),\\
 \|\mathcal R_z(w)-\mathcal R_z(\widetilde w)\|
 &\le K(\mu G^{-4}+\|w\|+\|\widetilde w\|)
                          \|w-\widetilde w\|.
\end{align*}
These are Taylor estimates for the unperturbed field and the
analytic perturbation, on a fixed larger domain. They give a
ball of radius $K\mu G^{-4}$ and contraction constant
$K\mu G^{-4}$, as in Proposition 5.20 of \cite{GMS2016}.

For completeness, the return to Hamilton--Jacobi coordinates can
be made by a uniformly normalized inverse-function argument.
On the domain of Proposition 5.21 of \cite{GMS2016} the base projection is
\[
 F_{0,z}(v,\xi)=(r_h(v),\xi-dv+\alpha_h(v)),\qquad
 DF_{0,z}=\begin{pmatrix}y_h&0\\\alpha_h'-d&1\end{pmatrix}.
\]
Here $y_h$ is bounded away from zero, and this matrix and its
inverse are uniformly bounded. If $E_z$ is the projected
perturbation, then $\|E_z\|_0\le K\mu G^{-4}$. On the
reserved smaller domain, the weighted Cauchy estimate gives the
sufficient bound $\|DE_z\|_0\le K\mu G^{-1}$; this allows
even the worst $G^3$ cost of a $v$ derivative. The required change
$\operatorname{Id}+f$ solves
\[
 f=-DF_{0,z}^{-1}\bigl(
 F_{0,z}\circ(\operatorname{Id}+f)-F_{0,z}-DF_{0,z}f
                   +E_z\circ(\operatorname{Id}+f)\bigr).
\]
The term in parentheses has norm
$O(\|f\|^2+\mu G^{-4})$ and Lipschitz constant
$O(\|f\|+\|\widetilde f\|+\mu G^{-1})$.
The ball has radius $K\mu G^{-4}$, while the contraction constant
is at most $K\mu G^{-1}$. This proves existence, uniqueness, and holomorphy with
$\|f\|\le K\mu G^{-4}$. In particular, no smallness of
$\alpha_h'$ itself is being assumed. Uniqueness on the overlap
identifies this change with the inverse of the earlier one.
Recovering the generating-function derivatives by (84) of \cite{GMS2016}
and joining the two domains gives the estimates of Theorem 5.22 of \cite{GMS2016} uniformly for the complex parameter.

\paragraph{Straightening and the exponential remainder.}
Write the difference equation as in (86) of \cite{GMS2016}, and set
$F=(A-B)/(1+B)$. The transferred bounds of Lemma 6.6 of \cite{GMS2016} are
\[
 \|S_zA\|_{1/2}+\|S_zB\|_{1/2}\le K\mu G^{-4}.
\]
By \eqref{eq:transfer-weight-embeddings},
$\|S_zB\|_0\le K\mu G^{-5/2}<1/2$. Hence the geometric
series for $(1+S_zB)^{-1}$ converges in the order-zero algebra,
and $\|S_zF\|_{1/2}\le K\mu G^{-4}$.
If $C$ is the original straightening displacement and $c_1=S_zC$,
its exact transformed equation is
\[
 L_Gc_1=(S_zF)(v+c_1,\xi+dc_1).
\]
Thus it is a two-component composition, even though $C$ is a
scalar displacement. Formula \eqref{eq:shear-composition}, the
composition estimate, and Lemma 6.4 of \cite{GMS2016} give
\[
 \|c_1\|_0\le K\mu G^{-4},\qquad
 \operatorname{Lip}\mathcal F_{\rm str}\le K\mu G^{-1}.
\]
The symmetric normalization is retained because $H_z$ commutes
with $(v,\xi)\mapsto(-v,-\xi)$. At equal masses it also
preserves $\pi$-periodicity in $\xi$.

All fixed-point maps above are holomorphic in $z$ on fixed spaces
and are contractions with a uniform constant less than $1/2$
for sufficiently large $G$. Their iterates, including the endpoint
data inherited from the preceding stage, converge locally uniformly
in $z$. Their solutions are therefore holomorphic. Returning by
$S_z^{-1}$ on a final smaller strip preserves every stated power
of $G$. This proves the required transfer of the construction,
without assuming convergence of the Fourier series near $\pm i/3$.

We also record explicitly the size of the error used in the
exponential estimate. The nonlinear remainder $Q-L_1$ has
weight $1$ and size $K\mu^2G^{-8}$ by
\eqref{eq:complex-hj-bounds}, so its order-zero bound is
$K\mu^2G^{-5}$. The change of argument of the splitting
difference has weight $3/2$ and size
\[
 K\|C\|_0\|\partial_v(T_1^u-T_1^s)\|_{3/2}
       \le K\mu^2G^{-8}.
\]
Here the derivative bound follows from the near-singularity
Hamilton--Jacobi estimates and the regular flow-extension estimates;
the composition takes place on the reserved smaller domain.
Lowering the weight to zero costs $G^{9/2}$ and gives
$K\mu^2G^{-7/2}$. Hence the straightened difference minus its
linear Melnikov potential has order-zero norm at most
$K\mu^2G^{-7/2}$, uniformly on the complex disk. This is the
complex counterpart of the estimate used in Lemma 6.8 of \cite{GMS2016}.

\begin{Proposition}[Parameter derivative of the normalized potential]
\label{prop:normalized-potential-parameter}
Fix $\mu\in(0,1/2]$. The straightened splitting potential of
\cite{GMS2016} can be chosen locally holomorphically in the angular
momentum parameter near every sufficiently large real $G$. Normalize
its phase derivative by its explicit leading harmonic, and choose
the orientation of the phase so that the leading term is $\sin\theta$.
Denote the result by $\sin\theta+\varepsilon_G(\theta,z)$.
For every fixed nonnegative integer $j$, uniformly at real $z=G$,
\[
 \|\varepsilon_G(\cdot,G)\|_{C^j}=O(G^{-1/2}),\qquad
 \|\partial_z\varepsilon_G(\cdot,G)\|_{C^j}=O(G^{7/4})
 \quad (0<\mu<1/2).
\]
At $\mu=1/2$, use the second harmonic and its phase. The respective
bounds are $O(G^{-1})$ and $O(G^{3/2})$. In both cases the parameter
derivative is $o(G^2)$.
\end{Proposition}

\begin{proof}
We give the bounds that are needed in addition to the operator
construction above. On $|z-G|<cG^{-2}$, the nonzero Fourier
coefficients of the straightened potential satisfy
\begin{align*}
 |\Upsilon_z^{[1]}|&\le K G^{-3/2}e^{-G^3/3},\\
 |\Upsilon_z^{[2]}|&\le K G^{1/2}e^{-2G^3/3},\\
 |\Upsilon_z^{[\ell]}|&\le K^{|\ell|}G^{3|\ell|/2-3}
                         e^{-|\ell|G^3/3},\qquad |\ell|\ge3.
\end{align*}
To see this directly from the weighted construction, $Q_0$ has
only harmonics $0,\pm2$ and weighted norm
$O(\mu G^{-4})$ of order $1/2$. The remaining correction has
weighted norm $O(\mu G^{-6})$ of order $1$. Evaluation at the
vertices $v=\pm i(1/3-\kappa_{\rm dom} G^{-3})$ costs respectively
$G^{3/2+3|\ell|/2}$ and $G^{3+3|\ell|/2}$, up to constants
exponential in $|\ell|$. The straightening error has order-zero
weighted norm $O(\mu^2G^{-7/2})$, as in the proof of Lemma 6.8 of \cite{GMS2016}, and contributes at most
$K^{|\ell|}G^{3|\ell|/2-7/2}$ before the exponential factor.
The kernel identity for the straightened transport equation gives
the factor $e^{-|\ell|G^3/3}$ by the same vertex evaluation as in
Lemma 6.7 of \cite{GMS2016}. The bounded shear costs only an additional
$e^{O(|\ell|c)}$, absorbed into $K^{|\ell|}$.

For unequal masses the positive leading amplitude for the phase
derivative is
\[
 A_1(z)=\frac{\sqrt\pi}{2\sqrt2}\mu(1-\mu)(1-2\mu)
                 z^{-3/2}e^{-z^3/3}.
\]
For equal masses it is
\[
 A_2(z)=8\sqrt\pi\mu(1-\mu)z^{1/2}e^{-2z^3/3};
\]
in this case the new phase is $2\theta+\pi$, if $\theta$ denotes
the old phase. The powers of $z$ use the branch positive on the
positive real axis. Since $z^3-G^3=O(c)$, these amplitudes are
nonzero and their moduli are comparable with the corresponding
real amplitudes at $G$.

At equal masses all odd coefficients vanish for complex $z$ as
well: the potential and the fixed-point equations are $\pi$-periodic
in the original phase. The displayed coefficient estimates now
give, for every fixed $j$,
\begin{equation}\label{eq:complex-normalized-bound}
 \sup_{|z-G|<cG^{-2}}
       \|\varepsilon_G(\cdot,z)\|_{C^j}\le M_j,
\end{equation}
with $M_j$ independent of $G$. The extra powers of $|\ell|$ from
phase derivatives are absorbed by the exponential Fourier decay.
The real-parameter estimates of Proposition 3.1 and Lemma 6.8 of \cite{GMS2016},
with the same fixed domain margins, give $O(G^{-1/2})$ on the real
diameter of a smaller disk, or $O(G^{-1})$ at equal masses.
Here the domains and the endpoint normalization are held fixed
at the real center. This causes no loss of the real estimate:
for real $z$ in this disk the singularity-distance parameter changes
from $\kappa_{\rm dom}$ to $\kappa_{\rm dom}(z/G)^3=\kappa_{\rm dom}+O(G^{-3})$.
All these domains lie between the same two strictly separated
choices of margins. The contractions and the proof of the real
remainder bound therefore apply to this same continued family,
not to an independently renormalized potential at each real $z$.

We use the following elementary interpolation argument, which avoids
requiring a small remainder throughout the complex disk. If $f$ is
holomorphic on $|z-G|<r$, $|f|\le M$ there, and $|f|\le\eta$ on
the real segment $|z-G|<r/2$, then
\[
 |f'(G)|\le K r^{-1}\sqrt{M\eta}
\]
when $0<\eta\le M$; the case $\eta=0$ follows by the identity
theorem. Indeed, Cauchy's estimate bounds $f''$ on the disk of
radius $r/2$ by $KM/r^2$. Taylor's formula along a real interval
of length $h\le r/4$ gives
$|f'(G)|\le2\eta/h+KMh/r^2$. Choose
$h=(r/4)\sqrt{\eta/M}$.

Apply this argument to each phase derivative of the normalized
remainder, with $r$ a fixed multiple of $G^{-2}$,
$\eta=O(G^{-1/2})$, and $M=M_j$. It gives $O(G^{7/4})$.
At equal masses $\eta=O(G^{-1})$ gives $O(G^{3/2})$.
\end{proof}

\begin{Remark}
This argument proves the parameter estimate for the straightened
splitting potential. It is stronger than mere real analyticity:
the size of the complex neighborhood and the uniform normalized
bound are essential. It does not yet identify this normalized
potential with the graph difference used in
Proposition~\ref{prop:energy-contact}. That difference also contains
a spatial prefactor and a parameter-dependent phase coordinate.
Their derivatives, and the local-passage height in that same
coordinate, are estimated in Section~\ref{ss:rpc-uniform-energy}
in the joint energy--passage range.
\end{Remark}

\paragraph{The coordinate matching required for the application.}
Proposition~\ref{prop:normalized-potential-parameter} supplies the
parameter estimate for the potential. Its application also requires
uniform local-passage bounds in the actual energy-dependent
straightening coordinates. In those coordinates the global invariant
curve must be written as
\[
 y=a(\theta,C)A(C)(\sin\theta+\varepsilon(\theta,C)),\qquad
 x=x(\theta,C),
\]
with $a>0$, and estimate the height
$H_N=Y_N(x(\theta,C),C)/a(\theta,C)$.
In particular its parameter derivative is
\[
 \partial_C\log H_N
 =\frac{Y_{N,C}}{Y_N}
  +\frac{Y_{N,x}}{Y_N}x_C-\partial_C\log a.
\]
Controlling the first term does not control the last two. Suitable
polynomial size bounds, small phase derivatives, and
$|x_C|/x+|\partial_C\log a|=o(C^2)$ would finish this part of the
reduction. Section~\ref{ss:rpc-uniform-energy} proves these estimates
and the uniform passage bounds directly on a fixed scaled square. We first complete the
fixed-energy part, which needs no joint-limit argument.

\subsection{Completion for fixed Jacobi constant and varying mass}
\label{ss:fixed-energy-completion}

For the fixed-energy assertion the preceding joint-limit estimates
are unnecessary. We fix $C$ first and let the mass tend to zero.
All local coordinate intervals and all constants in the following
proof may depend on this fixed $C$. This order of choices is the
reason the argument does not prove the fixed-mass assertion.

\begin{Proposition}\label{prop:rpc-fixed-energy-tangencies}
There exists $C_*>0$ such that, for every fixed $C>C_*$, there is a
sequence of positive mass ratios tending to zero at which the
RPC3BP return map has quadratic homoclinic tangencies to infinity,
unfolding nondegenerately with $\mu$. There is also a sequence of
positive mass ratios tending to zero with quadratic homoclinic
tangencies of hyperbolic periodic saddles, unfolding nondegenerately
with $\mu$, whose stable and unstable manifolds have transverse
connections to infinity in both directions.
\end{Proposition}

\begin{proof}
Put $G=C/2$ and keep $G$ fixed throughout the proof. The actual
splitting and the modified Melnikov potential have the same
first derivative at $\mu=0$. Indeed the difference, apart from its
irrelevant constant Fourier coefficient, is $O(\mu^2)$ by the
weighted error estimates in Lemma 6.8 of \cite{GMS2016}. Fourier decay permits
any fixed number of phase derivatives of this assertion. The
leading first harmonic of the phase derivative at $\mu=0$ has
positive amplitude
\[
 a_*(G)=\frac{\sqrt\pi}{2\sqrt2}G^{-3/2}e^{-G^3/3}.
\]
Let $\theta\in\mathbb T=\mathbb R/(2\pi\mathbb Z)$ denote the phase.
After division by $a_*(G)$ and the choice of phase orientation,
the normalized phase function $m_G(\theta)$ satisfies
\[
 \max_{0\le j\le2}\sup_{\theta\in\mathbb T}
 \left|\frac{d^j}{d\theta^j}
       \bigl(m_G(\theta)-\sin\theta\bigr)\right|
       \longrightarrow0\qquad\text{as }G\to\infty.
\]
Thus $m_G(\theta)$ and its first two phase derivatives converge
uniformly to $\sin\theta$, $\cos\theta$, and $-\sin\theta$,
respectively. In the section graph below, this phase is evaluated
at $\theta=\vartheta_C(x)$.
Here the $O(G^{-1})$ correction to the first harmonic is given by
Proposition 3.1 of \cite{GMS2016}. The higher harmonics are exponentially smaller.
Their first mass derivatives have the same Fourier decay: this
follows either from the explicit mass polynomials in the Appendix
of \cite{GMS2016}, or by Cauchy's estimate on a fixed small complex
mass disk in that convergent potential expansion. Thus this step
does not differentiate an unspecified $O$-term in the mass.
Choose $C_*$ so that the displayed $C^2$ difference is less than
$1/100$ for $C>C_*$.

Use the exact local normal form \eqref{RPC3BP-1}, with $\mu$ as
parameter in a small compact interval containing zero. The
straightening can be chosen equivariantly for the reverser
\[
 (r,P_r,\varphi,P_\varphi)\mapsto
       (r,-P_r,-\varphi,P_\varphi).
\]
On the section $\varphi=0$ it exchanges $x$ and $y$, fixes the
parabolic point at infinity, and its fixed curve is $x=y$.
Equivariance follows directly by choosing the two local separatrix
graphs as images of one another under the reverser in the
axis-straightening construction.

At $\mu=0$ the global stable branch and the local unstable branch
coincide. On a fixed outgoing tail interval the global stable graph
therefore has the expansion
\begin{equation}\label{eq:rpc-mass-tail-expansion}
 y=\mu a_C(x)m_G(\vartheta_C(x))+
                  \mu^2 r_C(x,\mu),\qquad a_C(x)>0.
\end{equation}
We choose the phase orientation to make $a_C$ positive. To verify
the form of the linear term, differentiate the exact graph formula
(29) of \cite{GMS2016} at $\mu=0$. At this parameter the
straightened phase is the Kepler phase $G^3v+\xi$. On the section,
$\xi=-\alpha_h(v)$, so its derivative contributes the factor
$G^3-r_h(v)^{-2}$. The remaining factors are the nonzero radial
velocity and the transverse coordinate change. They multiply
the derivative of the Melnikov potential by a positive spatial
factor after the above choice of orientation. A parameter-dependent
coordinate change cannot add a first-order splitting term when
the two branches coincide at $\mu=0$.

The Kepler tail and the smooth axis change give
\[
 a_C(x)\asymp x^{-1},\quad
 |a_C'(x)|\le K_C a_C(x)/x,\quad
 |a_C''(x)|\le K_C a_C(x)/x^2,
\]
\[
 |\vartheta_C'(x)|\asymp x^{-4},\qquad
 |\vartheta_C''(x)|\le K_Cx^{-5}.
\]
The constants in $a_C$ include $a_*(G)$, which is positive and
fixed. These bounds also follow from
$r_h(v)\asymp |v|^{2/3}$ and $y_h(v)\asymp |v|^{-1/3}$,
with the physical scaling $r=G^2r_h$ and $x\asymp r^{-1/2}$.
The remainder $r_C$ and its needed spatial and mass derivatives
are bounded on each fixed tail interval. This follows from smooth
parameter dependence of the local manifolds and their finite
global continuation to that interval. The interval is chosen
before the mass range is reduced; no bound all the way to the
vertex is being inferred from this Taylor expansion.

Choose a whole phase lobe on such an interval, sufficiently far
out that its inverse phase coordinate satisfies
$x_\theta=O(x^4)$ and $x_{\theta\theta}=O(x^7)$ and that the
relative phase derivatives below are less than $1/100$.
Let $y=\Delta_N(x,\mu)$ be the outgoing component of the image
of $x=y$. The local results already proved, on this fixed compact
mass range, give
\begin{gather*}
 \Delta_N>0,\qquad \Delta_N\to0,\qquad
 |\Delta_{N,\mu}|\le K_C\Delta_N,\\
 |\Delta_{N,x}|\le K_C\Delta_N/x,\qquad
 |\Delta_{N,xx}|\le K_C\Delta_N/x^2.
\end{gather*}
Divide the difference between \eqref{eq:rpc-mass-tail-expansion}
and this graph by $a_C(x)$, and use the fixed Kepler phase.
It has exactly the form
\[
 F_N(\theta,\mu)=\mu m_G(\theta)+\mu^2\rho_C(\theta,\mu)
                         -H_N(\theta,\mu),\qquad
 H_N=\Delta_N/a_C.
\]
The remainder has bounded derivatives on the chosen compact
rectangle. The relative phase derivatives of $H_N$ are bounded
by $K_C\delta^3$ and $K_C\delta^6$; its logarithmic mass derivative
is bounded by a constant depending on $C$ and the chosen interval.
Reduce the mass interval so that integration of these logarithmic
bounds gives the height comparison in Lemma~\ref{lem:separation}.
That lemma produces quadratic contacts at positive
$\mu_N\to0$, with nonzero mass-unfolding derivative.

Proposition~\ref{prop:general-reversible-contact} turns these
contacts with an iterated symmetry curve into homoclinic tangencies
to the same parabolic point. There are transverse homoclinic
intersections on the neighboring phase lobes for small positive
mass, by the simple zeros of $m_G$ and the expansion above
(equivalently, by Theorem 2.2 of \cite{GMS2016}). The periodic-saddle construction
and parameter transfer in Section~\ref{ss.saddlesnearhomoclinicpoint}
now apply with $\mu$ in place of the Sitnikov eccentricity. Their
inputs are the verified local normal form, smooth finite global
maps, transverse homoclinic intersections, and the quadratic
contact with nonzero unfolding derivative. Each contact is treated
on its own compact mass interval. Choose the periodic approximation
so close that the new tangency parameter lies within $\mu_N/2$
of $\mu_N$. These parameters remain positive and tend to zero.
The transverse connections to infinity are part of that construction.
\end{proof}

\begin{Corollary}\label{cor:rpc-fixed-energy-completion}
For every $C>C_*$, hypothesis $(H_C)$ holds for a nonempty open
set $\mathcal N_C\subset(0,1/2)$ accumulating at zero.
Consequently the fixed-$C$ conclusions of
Theorems~\ref{RPC3BP-main} and \ref{parameter-Newhouse-RPC3BP}
hold without an additional global hypothesis.
\end{Corollary}
\begin{proof}
Let $\widehat\mu_N>0$ be the hyperbolic saddle tangency parameters
from Proposition~\ref{prop:rpc-fixed-energy-tangencies}. They tend to
zero. For each $N$, choose a compact persistence interval $J_N$ with
$\widehat\mu_N$ in its interior, on which the saddle and both selected
transverse connections continue. Apply Theorem~\ref{t.thick-sets}
with $\delta_0=1$ inside
\[
 \operatorname{int}J_N\cap
 (\widehat\mu_N/2,3\widehat\mu_N/2)\cap(0,1/2).
\]
The resulting open set $\mathcal N_{C,N}$ has dense quadratic
tangencies unfolding in the mass parameter, with the connections
persisting throughout. The union $\mathcal N_C=\bigcup_N\mathcal N_{C,N}$
therefore satisfies $(H_C)$ and accumulates at zero. These domains
are fixed before the dimension tolerance is decreased.
Proposition~\ref{prop:rpc-conditional-completion}
then gives full dimension on a residual subset, including the flow
dimension three. No identification with the signed energy parameter
of \cite{GMP2025} is used in this proof.
\end{proof}

\begin{Remark}
The exponentially small value of $a_*(G)$ causes no obstruction
here because $G$ is fixed before $N\to\infty$. It merely changes
the constants and the mass interval in the contact argument.
This observation cannot be used when the mass is fixed and
$G\to\infty$; the next subsection supplies the estimates needed
for that different limit.
\end{Remark}

\subsection{Uniform estimates in the scaled coordinates}
\label{ss:rpc-uniform-energy}

We now keep the mass $\mu\in(0,1/2]$ fixed and write $G=C/2$.
All constants in this subsection may depend on this mass and on a
fixed, sufficiently small local square. They do not depend on $G$
or on the number $N$ of local iterates. The useful choice is to fix
the square in the \emph{scaled} McGehee coordinates. In the original
radial coordinate its distance from infinity then decreases as $G$
increases. This avoids a parameter-dependent local cutoff.

\paragraph{The scaled equations and averaging.}
Write $r=G^2R$, $P_r=G^{-1}p$, and $P_\varphi=GJ$.
Here $R$ is the scaled physical radius, rather than the reverser
used in the tangency construction, and $p$ is the scaled radial
momentum. The scaled physical time is $\tau=t/G^3$; it is distinct
from the local ratio $y/x$ used in the passage estimates. The equations are
\begin{align*}
 R'&=p,& p'&=J^2/R^3-1/R^2+V_R,\\
 \varphi'&=J/R^2-G^3,& J'&=V_\varphi,
\end{align*}
where a prime denotes differentiation in $\tau$ and
\[
 V(R,\varphi;G)=G^2U_\mu(G^2R,\varphi)-R^{-1}.
\]
Here $U_\mu$ is the positive Newtonian potential of the primaries in
rotating coordinates. The energy equation is
\begin{equation}\label{eq:uniform-scaled-energy}
 \frac{p^2}{2}+\frac{J^2}{2R^2}-\frac1R-V=G^3(J-1).
\end{equation}
Put $q=R^{-1/2}=Gu$, where $u=r^{-1/2}$ is the unscaled
McGehee coordinate, and use the angular time $\theta=-\varphi$.
On a fixed small
neighborhood of $(q,p)=(0,0)$ the implicit function theorem solves
(\ref{eq:uniform-scaled-energy}) uniquely for $J$ near one.
Moreover, $\theta'=G^3-Jq^4>0$. Thus $\theta$ is a legitimate
time coordinate and its return period is exactly $2\pi$.

For clarity, $O(G^{-a}\rho^b)$ below includes any specified finite
number of weighted spatial derivatives: a derivative of total order
$j$ in $(q,p)$ is bounded by $K_jG^{-a}\rho^{b-j}$, where
$\rho=q+|p|$. Angular derivatives do not change the weight.
Direct expansion of the potential gives
\[
 V=\frac{\mu(1-\mu)}{2G^4}q^6(3\cos^2\theta-1)
                         +O(G^{-6}q^8).
\]
Consequently, with brackets denoting angular average,
\begin{equation}\label{eq:uniform-averaging-input}
 J-1=O(G^{-3}\rho^2),\qquad
 J-\langle J\rangle=O(G^{-7}\rho^6).
\end{equation}
The $G$ derivatives of these bounds gain a factor $G^{-1}$.
These are estimates for the explicit implicit function $J$, not
for a separatrix splitting.

The reduced Hamiltonian for $\theta$ time is $J$ in the canonical
variables $(R,p)$. In $(q,p)$ its Hamiltonian vector field is
\[
 \frac{dq}{d\theta}=-\frac{q^3}{2}J_p,
 \qquad \frac{dp}{d\theta}=\frac{q^3}{2}J_q.
\]
Take the zero-mean angular primitive of $J-\langle J\rangle$ as
a generating function, with the sign chosen to cancel this term.
Its size is $O(G^{-7}\rho^6)$, so the associated change of variables
has size $O(G^{-7}\rho^8)$ and $G$ derivative
$O(G^{-8}\rho^8)$. One averaging step gives
\begin{equation}\label{eq:uniform-averaged-Hamiltonian}
 J_{\rm av}=\langle J\rangle+\mathcal E,\qquad
 \mathcal E=O(G^{-10}\rho^9),\qquad
 \mathcal E_G=O(G^{-11}\rho^9).
\end{equation}
Indeed, the bracket of a term of order $G^{-3}\rho^2$ with one
of order $G^{-7}\rho^6$ has order $G^{-10}\rho^9$:
the two derivatives lower the degree by two and the Poisson tensor
contributes $q^3$. All subsequent brackets have smaller bounds.
The same computation after a $G$ derivative proves the last estimate.
The Lie series converges on a slightly smaller fixed neighborhood;
alternatively its integral remainder gives these finite-order bounds.
The generating function is odd in $\theta$ and vanishes at $\theta=0$.
Thus this averaging change is the identity on the section in use
and respects the reversibility.

Notice that (\ref{eq:uniform-averaged-Hamiltonian}) does not estimate
the exponentially small splitting. We use averaging to explain and
control the local scales. For the global graph we retain the exact
splitting potential from Section~\ref{ss:complex-splitting-control}.

\begin{Lemma}[A uniform local square]\label{lem:rpc-scaled-square}
For all sufficiently large $G$, there are exact, reversible local
coordinates $(x,y)$ on a square independent of $G$ in which the
$\theta$-time equations are
\begin{equation}\label{eq:rpc-scaled-normal-form}
 \frac{dx}{d\theta}=a(G)x(x+y)^3(1+\mathcal A),\qquad
 \frac{dy}{d\theta}=-a(G)y(x+y)^3(1+\mathcal B),
 \qquad a(G)=\frac{1}{8\sqrt2\,G^3}.
\end{equation}
The axes are the actual local invariant manifolds, the symmetry
curve on $\theta=0$ is $x=y$, and, for $r=x+y$,
\begin{align}
 |\partial_x^i\partial_y^j\mathcal A|
 +|\partial_x^i\partial_y^j\mathcal B|
       &\le K r^{2-i-j},\qquad i+j\le2,
       \label{eq:rpc-uniform-spatial}\\
 |\mathcal A_G|+|\mathcal B_G|&\le K G^{-1}r^2.
       \label{eq:rpc-uniform-parameter}
\end{align}
The derivatives are taken at fixed $(x,y,\theta)$. The coefficients
are $2\pi$-periodic in $\theta$.
\end{Lemma}

\begin{proof}
We record the invariant-manifold estimates used in making the
coordinates exact. On the two local tails, the radial graphs in the
scaled variables satisfy
\begin{equation}\label{eq:rpc-local-tail-jets}
 p_\pm=\pm q\sqrt{2-q^2}+h_\pm(q,\theta;G),\qquad
 h_\pm=O(G^{-4}q^5),\quad
 \partial_\theta h_\pm=O(G^{-7}q^8).
\end{equation}
Both estimates include fixed finite spatial and angular jets.
Here the positive sign is the local stable branch and the negative
sign is the local unstable branch. To check the powers directly
from the estimates at infinity in Section 5.3 of \cite{GMS2016},
use $q\asymp |v|^{-1/3}$ and $|y_h|\asymp |v|^{-1/3}$.
The first correction to the generating function has weighted
norm $O(G^{-4})$ of order one, and its angular derivative is
$O(G^{-7}|v|^{-2})$. A $v$ derivative and multiplication by
$y_h^{-1}$ give, respectively, $G^{-4}q^5$ and $G^{-7}q^8$.
The next correction has weighted norm $O(G^{-6})$ of order $5/3$
and gives smaller terms. The term involving
$r_h^{-2}\partial_\xi T_1$ in the radial graph formula is smaller
as well. Further spatial jets follow by Cauchy estimates on
strictly smaller tail sectors. The substitution
$\xi=\varphi-\alpha_h(v)$ preserves these bounds.

Apply the fixed-domain complex construction of
Lemma~\ref{lem:transport-resolvent} to these same local tails.
On $|z-G|<cG^{-2}$ the bounds in (\ref{eq:rpc-local-tail-jets})
remain valid with the positive real center $G$ on the right-hand
sides.
Cauchy's estimate therefore gives the deliberately nonoptimal bounds
\begin{equation}\label{eq:rpc-local-tail-parameter}
 (h_\pm)_G=O(G^{-2}q^5),\qquad
 (h_\pm)_{\theta G}=O(G^{-5}q^8).
\end{equation}
In this display and the preceding disk statement the center $G$ is
real; $z$ is the complexified parameter. In particular no shrinking
real tail interval is being used.

One may straighten the two axes explicitly by setting
\[
 x=\frac{p_+(q,\theta;G)-p}{\sqrt2},\qquad
 y=\frac{p-p_-(q,\theta;G)}{\sqrt2}.
\]
The sum is $2q+O(q^3)$, so these equations define a uniformly
invertible coordinate change. The reflection
$(q,p,\theta)\mapsto(q,-p,-\theta)$ interchanges $x$ and $y$.
In particular its fixed curve on the section is the diagonal.
Invariance of the two graphs gives exact factors $x$ and $y$
in their respective equations. At $G=\infty$, their degree-four
terms are $x(x+y)^3/(8\sqrt2)$ and its negative counterpart;
the Kepler corrections have relative order two.

Here are bounds for the additional terms, including the potentially
large angular velocity. Work first in $\tau$ time. The explicit
energy equation gives $J-1=O(G^{-3}\rho^2)$. The perturbing radial
force is $O(G^{-4}q^8)$. The angular part of the coordinate
change contributes $G^3(h_\pm)_\theta=O(G^{-4}q^8)$.
After a parameter derivative this last bound becomes
$O(G^{-2}q^8)$ by (\ref{eq:rpc-local-tail-parameter}).
All other parameter derivatives of the straightening are
$O(G^{-2}q^5)$ with the corresponding spatial weights.
After division by the exact axis factors and by $(x+y)^3$,
these errors satisfy (\ref{eq:rpc-uniform-spatial}) and
(\ref{eq:rpc-uniform-parameter}). Division by
$\theta'=G^3-Jq^4$ extracts $G^{-3}$ and leaves a factor
$1+O(G^{-3}r^4)$, with the same required parameter bounds.
The inverse coordinate change has the stated weighted bounds by
implicit differentiation. This proves the lemma.
\end{proof}

\begin{Proposition}[The uniform passage estimate]
\label{prop:rpc-uniform-passage}
Fix a sufficiently small $\delta>0$ in the square of
Lemma~\ref{lem:rpc-scaled-square}. Let $y=\Delta_N(x,G)$ be
the outgoing local image of $x=y$ after $N$ returns. On
$\delta/2\le x\le\delta$, if $N\delta^3/G^3$ exceeds a fixed
sufficiently large constant, then
\begin{align}
 K^{-1}\frac{G^2N^{-2/3}}x
 &\le\Delta_N(x,G)\le K\frac{G^2N^{-2/3}}x,
 \label{eq:rpc-uniform-height}\\
 |\partial_x\Delta_N|&\le 2\Delta_N/x,
 &|\partial_x^2\Delta_N|&\le K\Delta_N/x^2,
 \label{eq:rpc-uniform-height-spatial}\\
 |\partial_G\log\Delta_N(x,G)|&\le K/G.
 \label{eq:rpc-uniform-height-parameter}
\end{align}
These constants are independent of $N$ and $G$.
\end{Proposition}

\begin{proof}
The spatial estimates follow from the diagonal barriers and curvature
estimate in Section~\ref{sec:revised-general}, since the weighted
remainders on this fixed square are uniformly small. For the size,
start at $(u,u)$ and put $ds=a(G)(x+y)^3d\theta$. The integrable
remainder estimate gives
\[
 x(s)\asymp ue^s,\quad y(s)\asymp ue^{-s},\quad
 xy\asymp u^2,\quad
 2\pi N=a(G)^{-1}\int_0^S(x+y)^{-3}ds\asymp G^3u^{-3}.
\]
Thus $u\asymp GN^{-1/3}$ and
(\ref{eq:rpc-uniform-height}) follows. The condition in the
statement ensures $u\ll\delta$ and the required outgoing passage.

The leading coefficient now depends on the parameter, so the
parameter proof requires one additional estimate. Write
$\tau=y/x$, $\lambda=\partial_x\Delta$, and
$w=1+\lambda/\tau$. The diagonal barriers give $|w|\le2$.
The slope equation, with $k=-\lambda/\tau=1-w$, gives
\[
 w_s=-3w+\frac{6\tau w-3\tau w^2}{1+\tau}+O(r^2).
\]
Since $\tau$ decreases exponentially in $s$, it enters
$[0,1/20]$ in a bounded time. Thereafter
\[
 D^+|w|\le-2|w|+Kr^2,
 \qquad \int_0^S |w(s)|\,ds\le K.
\]
The last bound follows by integration and
$\int_0^Sr^2ds\le K\delta^2$. It is independent of the passage
length, including the limit $u\to0$.

Differentiate the graph at fixed $x$ and fixed terminal
$\theta=2\pi N$, and put $W=\Delta_G/\Delta$. The same graph
variation calculation as in Proposition~\ref{prop:parameter-graph-bound}
gives $W_s=K_0W+L$, where
\begin{align*}
 |K_0|&\le K(\tau+r^2),\\
 L&=-\frac{a'}a\left((1+\mathcal B)
              +\frac{\lambda}{\tau}(1+\mathcal A)\right)
       -\mathcal B_G-\frac{\lambda}{\tau}\mathcal A_G.
\end{align*}
Here $a'/a=-3/G$. Therefore
\[
 |L|\le\frac K G(|w|+r^2),\qquad
 \int_0^S|K_0|ds\le K,\qquad
 \int_0^S|L|ds\le K/G.
\]
Since the initial diagonal is parameter independent, $W(0)=0$.
Variation of constants proves (\ref{eq:rpc-uniform-height-parameter}).
In particular there is no factor $N$ or $\log N$ in this estimate.
The orbit-dependent rescaled time is used only after the fixed-time
graph derivative has been taken.
\end{proof}

\paragraph{Transport of the exact splitting phase.}
It remains to express the global graph and the passage height in
the same coordinates. We give the estimates, since control of the
potential alone would not suffice.

Fix a compact interval of the negative Kepler tail, away from the
turning point $y_h=0$, corresponding to an outgoing strip strictly
inside the local square. Enlarge it slightly before making any
estimates. The interval, though far along the tail, is independent
of $G$. The global stable graph and the local unstable graph both
exist there for all large $G$. There is a domain choice to make
explicit here. The common boomerang domain in equation (43) of
\cite{GMS2016} contains a positive, not a negative, real interval.
First reflect that construction by
$(v,\xi)\mapsto(-v,-\xi)$ and interchange the stable and unstable
branches. In radial coordinates this is the reverser
$(R,p,\varphi,J)\mapsto(R,-p,-\varphi,J)$. It is independent
of $G$, so all complex-parameter bounds are preserved. We obtain
a common domain containing a fixed negative real interval.
Continuation from this reflected interval to the chosen negative
tail has fixed finite $v$ length; it never crosses $v=0$, and the
radial velocity stays away from zero. The equations
in $v$ have angular velocity $G^3+O(1)$ but bounded spatial
variations after removing this constant angular velocity. Integral
equations on this fixed interval give $O(G^{-4})$ perturbations
and fixed angular jets with constants depending on the interval.
The same continuation works for $|z-G|<cG^{-2}$ after reducing
$c$: the additional imaginary angular displacement is bounded by
$Kc$ times the fixed length. This proves uniform estimates on a
smaller fixed angular strip. No interval of length tending to
infinity with $G$ is used here.

Let $D=T_1^u-T_1^s$. Subtracting the two Hamilton--Jacobi equations,
as in equations (86)--(87) of \cite{GMS2016}, gives the exact equation
\begin{equation}\label{eq:rpc-exact-phase-transport}
 D_v-G^3b(v,\xi;G)D_\xi=0,\qquad b=1+O(G^{-4}).
\end{equation}
The bound includes every fixed number of angular derivatives.
It holds on the real interval and its complex-parameter continuation.
We will also use the weaker bounds
\[
 b_v=O(G^{-1}),\qquad b_{vv}=O(G^2),\qquad b_G=O(G^{-2}),
\]
with fixed angular jets. For the first two, write the radial graph
equations with $v$ as independent variable. Each $v$ derivative
can contribute at most one factor $G^3$ through the angular
transport term, while the angular derivatives of the perturbation
are $O(G^{-4})$. The algebraic formulas for $b$ have denominators
bounded away from zero on the fixed tail. The last bound follows
from Cauchy on the parameter disk. These deliberately coarse bounds
avoid assuming that the transported graphs have slow $v$ dependence.
The radial graph formula is
\[
 p^u-p^s=y_h^{-1}(D_v-r_h^{-2}D_\xi)
 \quad\hbox{on}\quad \xi=-\alpha_h(v).
\]
The potential is even in its phase by the reversibility
normalization of \cite{GMS2016}. If $\psi_+$ is the straightened phase on the positive
real interval, set
\[
 \psi_-(v,\xi;G)=-\psi_+(-v,-\xi;G).
\]
The generating-function symmetry gives
$D(v,\xi)=D(-v,-\xi)$, and hence
$D=\Upsilon(\psi_-)$ on the reflected interval. Also
$\psi_-=G^3v+\xi+O(G^{-1})$, with the same angular and parameter
bounds. Extend this phase from the reflected common domain by the
transport equation, and denote the extension by $\psi$. It satisfies
\begin{equation}\label{eq:rpc-phase-estimates}
 D=\Upsilon(\psi),\qquad
 \psi=G^3v+\xi+\eta,\qquad
 \psi_v=G^3b\psi_\xi,\qquad
 \|\partial_\xi^j\eta\|\le K_j/G.
\end{equation}
An irrelevant additive constant in $D$ may be included in
$\Upsilon$. To verify the last estimate, substitute this expression
for $\psi$ in its equation. The forcing is $G^3(b-1)=O(G^{-1})$;
the coefficient of the differentiated characteristic equation is
$G^3b_\xi=O(G^{-1})$. The initial phase correction is $O(G^{-1})$
by the straightening estimate of \cite{GMS2016}. Integration on the fixed interval
and induction over $j$ prove the claim. The preceding complex
continuation gives the same bounds on a disk of radius $cG^{-2}$.
Cauchy's estimate now yields
\begin{equation}\label{eq:rpc-phase-parameter}
 |\partial_G\partial_\xi^j\eta|\le K_jG.
\end{equation}
Only this polynomial bound, not a sharp expansion of the phase,
is needed below.

On the section, put $\psi_0(v,G)=\psi(v,-\alpha_h(v);G)$.
Since $\alpha_h'=r_h^{-2}$, we have the exact identities
\begin{align}
 \partial_v\psi_0&=(G^3b-r_h^{-2})\psi_\xi\asymp G^3,
 \label{eq:rpc-section-phase}\\
 p^u-p^s&=y_h^{-1}(G^3b-r_h^{-2})\psi_\xi\,
                                      \Upsilon'(\psi_0).
 \label{eq:rpc-section-splitting}
\end{align}
In particular we have not differentiated a Melnikov approximation
or substituted the algebraic averaging remainder for the splitting.

\paragraph{A single phase family on the energy window.}
The preceding holomorphic constructions use domains fixed at a
real center. Their phases need not agree on overlaps. We now make
a coherent smooth choice; otherwise a derivative estimate at each
center would not by itself give one family on the whole interval
$J_N$.

Choose a locally finite cover of the large positive $G$ axis by
intervals of length comparable to $G^{-2}$, with bounded overlap,
and a smooth partition of unity $\chi_i(G)$ subordinate to smaller
intervals in this cover. This is obtained from a fixed mesh in
$G^3$; in particular $|\chi_i'|\le KG^2$. Let
\[
 \psi_i(v,\xi;G)=G^3v+\xi+\eta_i(v,\xi;G)
\]
be the local phase constructions on one fixed tail interval.
Use the lifts for which $\eta_i=O(G^{-1})$. The interpolation
argument above also applies at every real point of the smaller
interval, by using a smaller disk contained in its original disk.
Thus all the local bounds, including the normalized-potential
parameter bound, hold throughout each support of $\chi_i$.
Set
\[
 \psi=\sum_i\chi_i\psi_i
      =G^3v+\xi+\sum_i\chi_i\eta_i.
\]
Since the transport equation is linear and differentiates only
$v,\xi$, this is still an exact phase for that equation.
It has degree one in $\xi$, and $\psi_\xi=1+O(G^{-1})>0$.
The angular jet bounds are preserved. Also
\[
 \partial_G\bigl(\psi-G^3v-\xi\bigr)
 =\sum_i\chi_i\partial_G\eta_i+\sum_i\chi_i'\eta_i=O(G).
\]
The same argument applies after any fixed number of angular
derivatives. Reversibility and, at equal masses, the half-period
symmetry are preserved by this linear combination.

There is a single periodic function $\Upsilon(\cdot,G)$ with
$D=\Upsilon(\psi,G)$: define it on a fixed transverse line $v=v_*$
and then use the common transport equation for $D$ and $\psi$.
We must check that its normalized derivative still has the
estimates used below. On an overlapping local chart write the
local phase on this transverse line as a function of the new phase:
\[
 \psi_i=\psi+\rho_i(\psi,G),\qquad
 \|\rho_i\|_{C^j}\le K_j/G,\qquad
 \|\partial_G\rho_i\|_{C^j}\le K_jG.
\]
For the last bound, differentiate
$\eta_i-\sum_j\chi_j\eta_j$ at fixed new phase. Its derivative
at fixed $\xi$ is $O(G)$, while its angular derivative is
$O(G^{-1})$ and $\xi_G=O(G^2)$ on $v=v_*$.
This proves the bound, including angular jets by differentiation
of the inverse coordinate.

The normalized derivative of the new potential is consequently
\[
 \bigl(\sin(\vartheta+\widetilde\rho_i)
       +\varepsilon_i(\vartheta+\widetilde\rho_i,G)\bigr)
       (1+\partial_\vartheta\widetilde\rho_i),
\]
where $\vartheta=\psi$ for unequal masses, and
$\vartheta=2\psi+\pi$ for equal masses; the corresponding
$\widetilde\rho_i$ obeys the same bounds. The last factor is the
Jacobian from differentiating the potential. Subtracting
$\sin\vartheta$ gives a new error with $C^2$ norm
$O(G^{-1/2})$ for unequal masses and $O(G^{-1})$ for equal
masses. Its parameter derivative is respectively $O(G^{7/4})$
and $O(G^{3/2})$: every new term from $\widetilde\rho_i$ or its
Jacobian is at most $O(G)$, using one additional fixed angular jet.
Thus the normalized-potential and phase estimates hold for this
single smooth family on the entire large-$G$ axis. Holomorphy
of the patched phase is not needed; the local holomorphic families
have already supplied the derivative bounds.

\begin{Proposition}[Uniform coordinate matching]
\label{prop:rpc-uniform-matching}
On a fixed outgoing strip the exact global stable graph can be
parametrized as
\begin{equation}\label{eq:rpc-matched-global-graph}
 y=a(\vartheta,G)\mathcal A_\mu(G)
             (\sin\vartheta+\varepsilon(\vartheta,G)),
 \qquad x=x(\vartheta,G),
\end{equation}
on any phase lobe whose inverse image stays in that strip. Here
$a>0$, $a\asymp1/x$, and
\[
 \mathcal A_\mu(G)=
 \begin{cases}
 \displaystyle\frac{\sqrt\pi}{2\sqrt2}
 \mu(1-\mu)(1-2\mu)G^{3/2}e^{-G^3/3},&0<\mu<1/2,\\[4pt]
 \displaystyle8\sqrt\pi\mu(1-\mu)G^{7/2}e^{-2G^3/3},&\mu=1/2.
 \end{cases}
\]
Uniformly over these lobes,
\begin{align}
 |x_\vartheta|+G|x_{\vartheta\vartheta}|&\le KG^{-3},&
 |x_G|/x&\le K/G,
 \label{eq:rpc-matching-x}\\
 \frac{|a_\vartheta|+|a_{\vartheta\vartheta}|}{a}
     &\le K/G,&
 |\partial_G\log a|&\le KG.
 \label{eq:rpc-matching-a}
\end{align}
Furthermore $\|\varepsilon\|_{C^2}=o(1)$ and
$|\varepsilon_G|=o(G^2)$.
\end{Proposition}

\begin{proof}
For unequal masses use $\vartheta=\psi_0$ modulo $2\pi$;
for equal masses use $\vartheta=2\psi_0+\pi$ modulo $2\pi$.
The normalized potential derivative and its error are those of
Proposition~\ref{prop:normalized-potential-parameter}, with the
coherent phase choice just constructed.
Equation (\ref{eq:rpc-section-splitting}) supplies the extra factor
$G^3$, explaining the powers in $\mathcal A_\mu$.

The section coordinate change is exact and makes the local
unstable graph the axis $y=0$. In fact its definition above gives
$y=(p^s-p^u)/\sqrt2$ on the global stable graph at the same $q$.
Thus the positive prefactor before using $x$ as independent
coordinate is exactly
\[
 -\frac{1}{\sqrt2\,y_h}
       (b-G^{-3}r_h^{-2})\psi_\xi.
\]
On the negative tail $y_h<0$, so it is positive and comparable
to $1/x$. The other coordinate is
$x=(p_+-p^s)/\sqrt2$; it differs from its value on the local
unstable graph by the exponentially small radial splitting.
All fixed jets of this difference are bounded by a polynomial in
$G$ times the splitting amplitude. These terms can therefore be
included in the bounds below. In particular the phase parametrization
is a graph for all sufficiently large $G$.

At fixed phase, (\ref{eq:rpc-section-phase}) and
(\ref{eq:rpc-phase-parameter}) give
\[
 v_\vartheta=O(G^{-3}),\qquad v_G=O(G^{-1}).
\]
For the second derivative it suffices to use
$\psi_{\xi v}=G^3(b_\xi\psi_\xi+b\psi_{\xi\xi})=O(G^2)$.
Together with bounded $b_v$ this gives
$\psi_{0,vv}=O(G^5)$ and hence
$v_{\vartheta\vartheta}=O(G^{-4})$.
Here $b_v$ is bounded on the fixed real tail: this follows either
by differentiating the continued equations, or from the radial
graph equations, whose angular perturbation derivatives are
$O(G^{-4})$ and whose only large coefficient is $G^3$.
The local coordinate jets are uniformly bounded on this fixed strip
by (\ref{eq:rpc-local-tail-jets})--(\ref{eq:rpc-local-tail-parameter}).
This proves (\ref{eq:rpc-matching-x}).

We spell out why the rapid $v$ dependence of the transported phase
causes no difficulty for the prefactor. Along the section,
\[
 \frac{d\psi_\xi}{dv}
   =G^3b_\xi\psi_\xi+(G^3b-r_h^{-2})\psi_{\xi\xi}.
\]
Divide by $d\vartheta/dv\asymp G^3$ and use
$\psi_{\xi\xi}=O(G^{-1})$. This proves an $O(G^{-1})$
first phase derivative. Differentiating again and using
$\psi_{\xi\xi\xi}=O(G^{-1})$ proves the same bound for the
second derivative; the angular derivatives of $b$ are $O(G^{-4})$.
The slow radial factors have derivatives $O(G^{-3})$ in phase.
For their second derivatives the weaker $O(G^{-4})$ bound for
$v_{\vartheta\vartheta}$ is sufficient. The factors involving
$b$ obey the same claimed bounds by their transport equations.
Finally, at fixed phase,
\[
 \frac{d\psi_\xi}{dG}
   =\partial_G\psi_\xi+
       \frac{d\psi_\xi}{dv}v_G=O(G)+O(G^2)O(G^{-1})=O(G).
\]
All other prefactor derivatives are smaller or bounded by this
quantity. This proves (\ref{eq:rpc-matching-a}). The assertion
about $\varepsilon$ follows from the normalized-potential estimate
and the preceding phase-patching calculation. Subtracting an integer
multiple of $2\pi$ from the phase does not affect these bounds.
\end{proof}

\begin{Corollary}[The height in splitting coordinates]
\label{cor:rpc-matched-height}
Put
\[
 H_N(\vartheta,G)=
       \frac{\Delta_N(x(\vartheta,G),G)}{a(\vartheta,G)}.
\]
On a fixed outgoing strip, uniformly for $G$ comparable to
$(\log N)^{1/3}$, one has
\begin{align}
 H_N&\asymp G^2N^{-2/3},
 \label{eq:rpc-final-height-size}\\
 |\partial_\vartheta H_N|+|\partial_\vartheta^2H_N|
     &\le (K/G)H_N,
 \label{eq:rpc-final-height-phase}\\
 |\partial_G\log H_N|&\le KG=o(G^2).
 \label{eq:rpc-final-height-parameter}
\end{align}
\end{Corollary}
\begin{proof}
The size follows from $a\asymp1/x$ and
(\ref{eq:rpc-uniform-height}). Apply the chain rule twice in phase
using (\ref{eq:rpc-uniform-height-spatial}) and
(\ref{eq:rpc-matching-x})--(\ref{eq:rpc-matching-a}). All $x$
values lie in a fixed strip bounded away from zero, so the resulting
bound is $K/G$. For the parameter derivative use the exact identity
\[
 \partial_G\log H_N
 =\frac{\Delta_{N,G}}{\Delta_N}
   +\frac{\Delta_{N,x}}{\Delta_N}x_G-\partial_G\log a.
\]
The three terms are bounded respectively by $K/G$, $K/G$, and
$KG$. Finally $N\delta^3/G^3\to\infty$ in the stated joint limit,
so the passage estimates apply throughout the range.
\end{proof}

\paragraph{Choice of a common lobe and the parameter conclusion.}
For completeness, one must choose a lobe that stays on the fixed
tail while $G$ varies over a whole contact window. In terms of
$C=2G$, choose $a<b$ close to $(16/k)^{1/3}$ with
$a^3<16/k<b^3$, where $k=1$ for unequal masses and $k=2$ for
equal masses. The unwrapped phase on the tail is
$kG^3v+O(1)$. Choose $a,b$ sufficiently close that the images
of an interior tail interval under multiplication by $ka^3/8$
and $kb^3/8$ overlap with positive length. For large $N$, their
common phase range has length comparable to $\log N$.
It therefore contains a translate
$[-\pi/6,7\pi/6]+2\pi j_N$ with fixed margins from its endpoints.
The same integer $j_N$ works for every
$C\in[a(\log N)^{1/3},b(\log N)^{1/3}]$.
The bounded phase correction does not remove these margins.
The inverse images stay in the chosen outgoing strip.

\begin{Proposition}\label{prop:rpc-fixed-mass-tangencies}
For each fixed $\mu\in(0,1)$ there is a sequence of positive
Jacobi constants tending to infinity at which the RPC3BP has
quadratic parabolic homoclinic tangencies unfolding nondegenerately
with $C$. For $0<\mu\le1/2$, the sequence constructed above satisfies
\[
 C_N^3=\frac{16}{k}\log N+O(\log\log N),\qquad
 k=\begin{cases}1,&\mu<1/2,\\2,&\mu=1/2.\end{cases}
\]
There are also quadratic tangencies of hyperbolic periodic saddles,
unfolding with $C$, at parameters tending to infinity. The saddles
have the transverse connections to infinity used in
Section~\ref{ss:sparse-excursions}.
\end{Proposition}
\begin{proof}
Divide the difference between (\ref{eq:rpc-matched-global-graph})
and the local passage graph by $a$. Corollary~\ref{cor:rpc-matched-height}
and Proposition~\ref{prop:normalized-potential-parameter} verify
all the hypotheses of Proposition~\ref{prop:energy-contact}, with
$p=3/2$ or $7/2$ after replacing $G$ by $C/2$.
The phase derivative bound is smaller than $1/100$ for all large
$G$. Moreover $\partial_C\log H_N=O(C)=o(C^2)$, whereas
\[
 \partial_C\log\mathcal A_\mu(C/2)
       =-kC^2/8+p/C.
\]
Thus the contact is quadratic and its unfolding derivative is
nonzero. Proposition~\ref{prop:general-reversible-contact} turns
this contact with an iterate of the symmetry curve into a
parabolic homoclinic tangency.

The saddle approximation in
Section~\ref{ss.saddlesnearhomoclinicpoint} applies at each such
finite parameter. Its constants need only be uniform on a small
compact neighborhood of that parameter; the unscaled local
coordinates used earlier provide exactly this situation. Choose
the approximating saddle tangency within $1/N$ of $C_N$.
Nondegenerate unfolding and both transverse connections follow
from that approximation. The other phase lobes supply the
transverse parabolic intersections. Finally exchanging the
primaries replaces $\mu$ by $1-\mu$ and preserves $C$, proving
the assertion for $\mu>1/2$ as well.
\end{proof}

\begin{Corollary}\label{cor:rpc-fixed-mass-completion}
For every fixed $\mu\in(0,1)$ there is an unbounded open set
$\mathcal N_\mu$ of positive Jacobi constants and a residual
subset $\mathcal R_\mu\subset\mathcal N_\mu$ such that for
$C\in\mathcal R_\mu$ the set of two-sided oscillatory motions
on $\Pi_C$ has Hausdorff dimension three. The set
$\mathcal N_\mu$ may be chosen as a union of parameter Newhouse
domains accumulating at infinity.
\end{Corollary}
\begin{proof}
Let $\widehat C_N\to\infty$ be the hyperbolic saddle tangency
parameters supplied by Proposition~\ref{prop:rpc-fixed-mass-tangencies}.
Choose a compact interval $J_N$ about each $\widehat C_N$ on which
its saddle and both transverse connections persist. Apply
Theorem~\ref{t.thick-sets} with the fixed tolerance $\delta_0=1$
inside $\operatorname{int}J_N\cap(\widehat C_N-1,\widehat C_N+1)$.
Write $\mathcal N_{\mu,N}$ for the resulting open parameter set.
Then $\mathcal N_\mu=\bigcup_N\mathcal N_{\mu,N}$ is unbounded,
and all selected connections persist on the domain where they are
used. The density argument in Remark~\ref{rem:parameter-density},
applied on these fixed parameter domains, gives hyperbolic basic sets
containing the continued saddles and with dimensions tending to two
on a residual subset of their union. The sparse-excursion argument then gives a
set of two-sided oscillatory section points of dimension two.
The flow-box argument raises its dimension to three on the energy
surface. These are the same steps as in the fixed-energy proof;
only the one-dimensional parameter family has changed.
\end{proof}

\begin{Remark}
The uniformity needed here is for each fixed positive mass.
None of these estimates asserts a common threshold as $\mu\to0$
or as $\mu\to1/2$ through unequal masses. At exactly equal masses
we use the second harmonic and its separate normalization.
\end{Remark}


\begin{thebibliography}{Chazy}

\bibitem[Af]{Af} Alekseev, V.~M., Sur l'allure finale du mouvement dans le probl\`eme des trois corps, in \emph{Actes du Congr\`es International des Math\'ematiciens (Nice, 1970)}, Tome 2, Gauthier-Villars, Paris, 1971, 893--907.

\bibitem[Ae]{Ae} Alekseev, V.~M., Final motions in the three-body problem and symbolic dynamics, \emph{Uspekhi Mat. Nauk} \textbf{36} (1981), no.~4, 161--176; English translation in \emph{Russian Math. Surveys} \textbf{36} (1981), no.~4, 181--200.

\bibitem[A68II]{A1968II} Alekseev, V.~M., Quasirandom dynamical systems. II. One-dimensional nonlinear oscillations in a field with periodic perturbation, \emph{Math. USSR-Sb.} \textbf{6} (1968), no.~4, 505--560.

\bibitem[A69III]{A1969III} Alekseev, V.~M., Quasirandom dynamical systems. III. Quasirandom oscillations of one-dimensional oscillators, \emph{Math. USSR-Sb.} \textbf{7} (1969), no.~1, 1--43.

\bibitem[AKN]{AKN} Arnol'd, V.~I., Kozlov, V.~V., Neishtadt, A.~I., \emph{Dynamical Systems III: Mathematical Aspects of Classical and Celestial Mechanics}, Encyclopaedia of Mathematical Sciences, vol.~3, Springer-Verlag, Berlin, 1988.

\bibitem[Bak]{Bak} Bakker, L.~F., \emph{Arnold Diffusion in the Elliptic Restricted Planar Three-Body Problem}, M.Sc. thesis, University of Victoria, 1993.

\bibitem[BFM]{BFM} Baldom\'a, I., Fontich, E., Mart\'in, P., Invariant manifolds of parabolic fixed points (I). Existence and dependence on parameters, \emph{J. Differential Equations} \textbf{268} (2020), no.~9, 5516--5573. 

\bibitem[BGG]{BGG2023} Baldom\'a, I., Giralt, M., Guardia, M., Coorbital homoclinic and chaotic dynamics in the Restricted 3-Body Problem, preprint (2023), \href{https://arxiv.org/abs/2312.13819}{arXiv:2312.13819}.

\bibitem[BNRR]{BNRR2022} Barrientos, P.~G., Nakano, Y., Raibekas, A., Roldan, M., Topological entropy and Hausdorff dimension of irregular sets for non-hyperbolic dynamical systems, \emph{Dynamical Systems} \textbf{37} (2022), no.~2, 181--210. 

\bibitem[Bo2]{Bo2} Bowen, R., \emph{Equilibrium States and the Ergodic Theory of Anosov Diffeomorphisms}, second revised edition, edited by J.-R. Chazottes, Lecture Notes in Mathematics, vol.~470, Springer, Berlin, 2008.

\bibitem[CG]{CG2023} Capi\'nski, M.~J., Gidea, M., Arnold diffusion, quantitative estimates, and stochastic behavior in the three-body problem, \emph{Comm. Pure Appl. Math.} \textbf{76} (2023), 616--681. 

\bibitem[CGMSZ]{CGMSZ2022} Capi\'nski, M.~J., Guardia, M., Mart\'in, P., Seara, T.~M., Zgliczy\'nski, P., Oscillatory motions and parabolic manifolds at infinity in the planar circular restricted three body problem, \emph{J. Differential Equations} \textbf{320} (2022), 316--370.

\bibitem[CP]{CP2026} Capi\'nski, M.~J., Pasiut, A., Oscillatory collision approach in the Earth--Moon restricted three body problem, \emph{Commun. Nonlinear Sci. Numer. Simul.} \textbf{152} (2026), article 109173.

\bibitem[Chazy]{Chazy} Chazy, J., Sur l'allure du mouvement dans le probl\`eme des trois corps quand le temps cro\^it ind\'efiniment, \emph{Ann. Sci. \'Ecole Norm. Sup.} (3) \textbf{39} (1922), 29--130.

\bibitem[DH]{DH} Dankowicz, H., Holmes, P., The existence of transverse homoclinic points in the Sitnikov problem, \emph{J. Differential Equations} \textbf{116} (1995), no.~2, 468--483. 

\bibitem[DKRS]{DKRS2019} Delshams, A., Kaloshin, V., de la Rosa, A., Seara, T.~M., Global instability in the restricted planar elliptic three body problem, \emph{Comm. Math. Phys.} \textbf{366} (2019), 1173--1228. 

\bibitem[DN]{DN2005} Downarowicz, T., Newhouse, S., Symbolic extensions and smooth dynamical systems, \emph{Invent. Math.} \textbf{160} (2005), no.~3, 453--499.

\bibitem[Du3]{Du3} Duarte P.,  Plenty of elliptic islands for the standard
family of area preserving maps, {\it   Ann. Inst. H. Poincar\'e Anal.
Non Lin\'eaire, } vol.~11  (1994),  no.~4, pp.~359--409.

\bibitem[Du1]{Du1} Duarte P., Abundance of elliptic isles at conservative
bifurcations, {\it Dynamics and Stability of Systems}, vol.~14, no.~4, (1999), pp.~339--356.

\bibitem[Du2]{Du2} Duarte P., Persistent homoclinic tangencies for
conservative maps near the identity, {\it Ergodic Theory Dynam.
Systems}, vol.~20 (2000),  pp.~393--438.

\bibitem[Du4]{Du4} Duarte P., Elliptic Isles in Families of Area Preserving
Maps,  {\it Ergodic Theory and Dynamical Systems} vol.~28 (2008),
pp.~1781--1813.

\bibitem[FGLP]{FGLP2016} Franco-P\'erez, L., Gidea, M., Levi, M., P\'erez-Chavela, E., Stability interchanges in a curved Sitnikov problem, \emph{Nonlinearity} \textbf{29} (2016), 1056--1079. 

\bibitem[GaK]{GaK} Galante, J., Kaloshin, V., Destruction of invariant curves in the restricted circular planar three-body problem by using comparison of action, \emph{Duke Math. J.} \textbf{159} (2011), no.~2, 275--327.

\bibitem[GP]{GP} Garc\'ia, A., P\'erez-Chavela, E., Heteroclinic phenomena in the Sitnikov problem, in \emph{Hamiltonian Systems and Celestial Mechanics (P\'atzcuaro, 1998)}, World Scientific Monograph Series in Mathematics, vol.~6, World Scientific, River Edge, NJ, 2000, 174--185.

\bibitem[GMP]{GMP2025} Garrido, M., Mart\'in, P., Paradela, J., Parabolic saddles and Newhouse domains in celestial mechanics, \emph{Comm. Math. Phys.} \textbf{406} (2025), no.~7, article 173. 

\bibitem[Go]{Go} Gorodetski A., On stochastic sea of the standard map,  {\it  Comm. Math. Phys.}, vol.~309, no.~1 (2012), pp.~155--192.

\bibitem[GoK]{GK} Gorodetski A., Kaloshin V., How often surface diffeomorphisms
have infinitely many sinks and hyperbolicity of periodic points near
a homoclinic tangency, {\it Advances in Mathematics}, {\bf 208} (2007),
pp.~710--797.

\bibitem[GK09]{GK2009} Gorodetski, A., Kaloshin, V., Conservative homoclinic bifurcations and some applications, \emph{Trudy Mat. Inst. Steklova} \textbf{267} (2009), 82--96; English translation in \emph{Proc. Steklov Inst. Math.} \textbf{267} (2009), 76--90. 

\bibitem[GMPS]{GMPS2022} Guardia, M., Mart\'in, P., Paradela, J., Seara, T.~M., Hyperbolic dynamics and oscillatory motions in the 3 Body Problem, preprint (2022), \href{https://arxiv.org/abs/2207.14351}{arXiv:2207.14351}.

\bibitem[GMS]{GMS2016} Guardia, M., Mart\'in, P., Seara, T.~M.,
Oscillatory motions for the restricted planar circular three body problem,
\emph{Invent. Math.} \textbf{203} (2016), 417--492.

\bibitem[GPSV]{GPSV2021} Guardia, M., Paradela, J., Seara, T.~M., Vidal, C., Symbolic dynamics in the restricted elliptic isosceles three body problem, \emph{J. Differential Equations} \textbf{294} (2021), 143--177. 

\bibitem[GSMS]{GSMS2017} Guardia, M., Seara, T.~M., Mart\'in, P., Sabbagh, L., Oscillatory orbits in the restricted elliptic planar three body problem, \emph{Discrete Contin. Dyn. Syst.} \textbf{37} (2017), 229--256. 

\bibitem[K]{k2} Kaloshin V., Generic diffeomorphisms with superexponential
growth of number of periodic orbits, {\it  Comm. Math. Phys.} 211 (2000), no.~1, pp.~253--271.

\bibitem[KaH]{kh} Katok A., Hasselblatt  B., Introduction to
the Modern Theory of Dynamical Systems, {\it Cambridge University Press}, 1995.

\bibitem[LGS]{LGS2025} Lamas, J., Guardia, M., Seara, T.~M., Oscillatory motions, parabolic orbits and collision orbits in the planar circular restricted three-body problem, \emph{Comm. Math. Phys.} \textbf{406} (2025), article 106. 

\bibitem[LS]{LS} Llibre J.,  Sim\'o C., Oscillatory solutions in the planar
restricted three-body problem,{\it Math. Ann.} {\bf 248}, (1980),
pp.~153--184.

\bibitem[LS1]{ls1} Llibre J.,  Sim\'o C., Some homoclinic phenomena in
the three-body problem, J.Diff Equations, 37 (1980), pp.~444--465.

\bibitem[MMerr]{MMerr} Manning, A., Errata to `Hausdorff dimension for horseshoes', \emph{Ergodic Theory Dynam. Systems} \textbf{5} (1985), no.~2, 319.

\bibitem[Mar]{Mar} Marstrand J.M., The dimension of cartesian product sets,
{\it Proc. Cambridge Phil. Soc.} {\bf 50} (1954), pp.~198--202.

\bibitem[MP]{MP} Mart\'inez, R., Pinyol, C., Parabolic orbits in the elliptic restricted three body problem, \emph{J. Differential Equations} \textbf{111} (1994), no.~2, 299--339.

\bibitem[MM]{MM} McCluskey, H., Manning, A., Hausdorff dimension for horseshoes, \emph{Ergodic Theory Dynam. Systems} \textbf{3} (1983), no.~2, 251--260.

\bibitem[McG]{McG} McGehee R., A stable manifold theorem for degenerate fixed
points with applications to celestial mechanics. {\it J. Differential Equations}
14 (1973), pp.~70--88.

\bibitem[Moe]{Moe2007} Moeckel, R., Symbolic dynamics in the planar three-body problem, \emph{Regul. Chaotic Dyn.} \textbf{12} (2007), no.~5, 449--475.

\bibitem[MR]{MR1} Mora, L., Romero, N., Moser's invariant curves and homoclinic bifurcations, \emph{Dynamic Systems and Applications} \textbf{6} (1997), 29--42.

\bibitem[Mor]{Mor} Morrisey, T.~J., A degenerate Hartman theorem, \emph{Israel J. Math.} \textbf{95} (1996), 157--167.

\bibitem[Mo]{Mo} Moser, J., \emph{Stable and Random Motions in Dynamical Systems: With Special Emphasis on Celestial Mechanics}, Annals of Mathematics Studies, vol.~77, Princeton University Press, Princeton, NJ, 1973.

\bibitem[N1]{n1} Newhouse S., Non-density
of Axiom A(a) on $S^2$, {\it Proc. A.M.S. Symp. Pure Math.}, {\bf
14}, (1970), 191--202.

\bibitem[N2]{n2} Newhouse S., Diffeomorphisms with infinitely many sinks. {\it
Topology} {\bf 13}, (1974), pp.~9--18.

\bibitem[N3]{n3} Newhouse S., The abundance of wild hyperbolic sets and nonsmooth
stable sets of diffeomorphisms, {\it Publ. Math. I.H.E.S.}, {\bf
50}, (1979), pp.~101--151.

\bibitem[N5]{N5} Newhouse S., Lectures on dynamical systems, {\it Dynamical
systems (C.I.M.E. Summer School, Bressanone, 1978), Progr. Math., 8,
Birkh\"auser, Boston, Mass.}, 1980, pp.~1--114.

\bibitem[PT]{PT} Palis J.,  Takens F., Hyperbolicity and sensitive chaotic
dynamics at homoclinic bifurcations. Cambridge University Press,
1993.

\bibitem[PV]{PV} Palis, J., Viana, M., On the continuity of Hausdorff dimension and limit capacity for horseshoes, in \emph{Dynamical Systems}, Lecture Notes in Mathematics, vol.~1331, Springer, Berlin, 1988, 150--160.

\bibitem[Poi]{p} Poincar\'e, H., Sur le probl\`eme des trois corps et les \'equations de la dynamique, \emph{Acta Math.} \textbf{13} (1890), 1--270.

\bibitem[R1]{R} Robinson, C., Bifurcation to infinitely many sinks, \emph{Comm. Math. Phys.} \textbf{90} (1983), no.~3, 433--459.

\bibitem[R2]{R2} Robinson, C., Homoclinic orbits and oscillation for the planar three-body problem, \emph{J. Differential Equations} \textbf{52} (1984), no.~3, 356--377. 

\bibitem[SX]{sx} Saari D., Xia J., The existence of oscillatory and
superhyperbolic motion in Newtonian systems, {\it J. Differential
Equations}, 82 (1989), no.~2, pp.~342--355.

\bibitem[SZ]{SZ2020} Seara, T.~M., Zhang, J., Oscillatory orbits in the restricted planar four body problem, \emph{Nonlinearity} \textbf{33} (2020), 6985--7015. 

\bibitem[Si]{sit} Sitnikov K., The existence of oscillatory motions in the
three-body problems.  {\it Dokl. Akad. Nauk SSSR } 133,  pp.~303--306 (Russian); translated as {\it Soviet Physics. Dokl.}  5
(1960), pp.~647--650.

\bibitem[Sl]{Sl} Slotnick D., Asymptotic behavior of solutions of canonical systems near a closed, unstable orbit, {\it Contributions to theory of non-linear oscillations, Vol. IV, Annals of Mathematics Studies}, no.~41, Princeton University Press, Princeton, N.J., 1958, pp.~85--110.

\bibitem[X]{X} Xia J., Melnikov method and transversal homoclinic points in
the restricted three-body problem,  {\it  J. of Diff. Equations},
vol.~96, no.~1 (1992), pp.~170--184.

\bibitem[Xi2]{Xi2} Xia J., Arnold diffusion in the elliptic restricted
three-body problem, {\it J. Dynam. Differential Equations} 5 (1993), no.~2, pp.~219--240.

\bibitem[Xi3]{Xi3} Xia J., Arnold diffusion and oscillatory solutions in the
planar three-body problem, {\it J. Differential Equations} 110 (1994), no.~2, pp.~289--321.

\end{thebibliography}
\end{document}